\documentclass[final,reqno]{amsart}

\usepackage[bookmarks=false,draft=false,breaklinks,colorlinks]{hyperref}
\usepackage{mathtools,amssymb,mathrsfs}
\usepackage{fullpage,enumitem,here,braket}
\usepackage[only llbracket,rrbracket,longmapsfrom]{stmaryrd}
\usepackage{tikz-cd}
\usepackage{showkeys,comment}
\usepackage[l2tabu,orthodox]{nag}
\usepackage[all, warning]{onlyamsmath}
\numberwithin{equation}{subsection}
\numberwithin{figure}{subsection}
\allowdisplaybreaks

\usepackage{enumitem}
\setenumerate{label=(\arabic*),nosep}
\setitemize{nosep}
\newlist{clist}{enumerate}{1}
\setlist*[clist]{label=(\roman*), nosep}

\usepackage{cleveref}
\crefname{thm}{Theorem}{Theorems}
\crefname{thm0}{Theorem}{Themrems}
\crefname{dfn}{Definition}{Definitions}
\crefname{prp}{Proposition}{Propositions}
\crefname{lem}{Lemma}{Lemmas}
\crefname{cor}{Corollary}{Corollaries}
\crefname{clm}{Claim}{Claims}
\crefname{fct}{Fact}{Facts}
\crefname{rmk}{Remark}{Remarks}
\crefname{eg}{Example}{Examples}
\crefname{figure}{Figure}{Figures}
\crefname{table}{Table}{Tables}
\crefname{section}{\S\!}{\S\S\!}
\crefname{subsection}{\S\!}{\S\S\!}
\crefname{subsubsection}{\S\!}{\S\S\!}
\crefname{appendix}{Appendix}{Appendices}
\crefname{equation}{}{}

\usepackage{amsthm}
\theoremstyle{definition}
\newtheorem{thm}{Theorem}[subsection]
\newtheorem{thm0}{Theorem}
\newtheorem{dfn}[thm]{Definition}
\newtheorem{prp}[thm]{Proposition}
\newtheorem{lem}[thm]{Lemma}
\newtheorem{cor}[thm]{Corollary}

\newtheorem{rmk}[thm]{Remark}
\newtheorem{eg}[thm]{Example}
\newtheorem*{rmk*}{Remark}
\newtheorem*{thm*}{Theorem}
\newtheorem*{Notation}{Notations}
\newtheorem*{Ackn}{Acknowledgements}
\newtheorem*{Org}{Organization}

\newcommand{\ol}{\overline}

\newcommand{\wt}{\widetilde}

\newcommand{\ve}{\varepsilon}
\newcommand{\pdd}{\partial}
\newcommand{\ceq}{\coloneqq} %using mathtool package
\newcommand{\bs}{\mathbin{\setminus}}

\newcommand{\xr}{\xrightarrow}

\newcommand{\inj}{\hookrightarrow}

\newcommand{\lsto}{\xr{\sim}}

\newcommand{\rarr}{\ratio\Longleftrightarrow}

\newcommand{\bbN}{\mathbb{N}}
\newcommand{\bbZ}{\mathbb{Z}}

\newcommand{\bbR}{\mathbb{R}}
\newcommand{\bbC}{\mathbb{C}}

\newcommand{\clF}{\mathcal{F}}
\newcommand{\clG}{\mathcal{G}}
\newcommand{\clE}{\mathcal{E}}

\newcommand{\clJ}{\mathcal{J}}

\newcommand{\frB}{\mathfrak{B}}

\newcommand{\frU}{\mathfrak{U}}

\newcommand{\vac}{\ket{0}}

\newcommand{\bfF}{\mathbf{F}}
\newcommand{\bfV}{\mathbf{V}}

\newcommand{\loc}{\mathrm{loc}}

\newcommand{\sfM}{\mathsf{M}}
\newcommand{\Lin}{\mathsf{Mod}\,\bbC}
\newcommand{\CAlg}{\mathsf{CAlg}}
\newcommand{\sfA}{\mathsf{A}}

\newcommand{\PFA}{\mathsf{PFA}}
\newcommand{\FA}{\mathsf{FA}}
\newcommand{\PCSh}{\mathsf{PCSh}}
\newcommand{\bgrVA}{\mathsf{grVA}^{\mathsf{bb}}}
\newcommand{\cbgrVA}{\mathsf{grComVA}^{\mathsf{bb}}}
\newcommand{\fcbgrVA}{\mathsf{grComVA}^{\mathsf{bb}, \mathsf{fin}}}
\newcommand{\LPFA}{\mathsf{PFA}^{\mathsf{loc}}(\bbC)}
\newcommand{\LFA}{\mathsf{FA}^{\mathsf{loc}}(\bbC)}
\newcommand{\HolFA}{\mathsf{HolFA}_{\bbC}}
\newcommand{\HolPFA}{\mathsf{HolPFA}_{\bbC}}

\newcommand{\LHFA}{\mathsf{HolFA}^{\mathsf{loc}}_{\bbC}}
\newcommand{\LCS}{\mathsf{LCS}_{\bbC}}

\newcommand{\abs}[1]{\left| #1 \right|}
\newcommand{\dbr}[1]{\llbracket #1 \rrbracket} %using stmaryrd package
\newcommand{\dpr}[1]{(\!( #1 )\!)}

\DeclareMathOperator{\id}{id}

\DeclareMathOperator{\End}{End}
\DeclareMathOperator{\Hom}{Hom}

\DeclareMathOperator{\Ker}{Ker}
\DeclareMathOperator{\Cok}{Coker}
\DeclareMathOperator{\Res}{Res}

\DeclareMathOperator{\Spn}{Span}
\DeclareMathOperator{\Fld}{Field}
\DeclareMathOperator{\Conf}{Conf}
\DeclareMathOperator{\Ob}{Ob}
\DeclareMathOperator{\oCom}{\mathcal{C}{\kern-0.1em}\mathit{om}}
\DeclareMathOperator{\oEnd}{\mathcal{E}{\kern-0.1em}\mathit{nd}}

\DeclareMathOperator{\conv}{conv}

\begin{document}

\title{A note on vertex algebras and Costello-Gwilliam factorization algebras}
\author{Yusuke Nishinaka}
\date{2025.01.29} 
\address{Graduate School of Mathematics, Nagoya University.
 Furocho, Chikusaku, Nagoya, Japan, 464-8602.}
\email{m21035a@math.nagoya-u.ac.jp}
\thanks{This work is supported by JSPS Research Fellowship for Young Scientists (No.\,23KJ1120).}

\begin{abstract}
We show that the construction of vertex algebras from Costello-Gwilliam factorization algebras on $\mathbb{C}$ can be achieved without the discreteness condition on the weight spaces. Furthermore, we construct locally constant factorization algebras from commutative vertex algebras, and discuss the relationship between this construction and the jet algebras of commutative algebras. 
\end{abstract}

\maketitle
\tableofcontents
%{\small \tableofcontents}

%%%%%%%%%%%%%%%%%%%%%%%%%%%%%%%%%%%%%%%%%%%%%%%%%%%%%%%%%%%%%%%%%%%%%%%%%%%%%%%%%%%%%%%%%%
%%%%%%%%%%%%%%%%%%%%%%%%%%%%%%%%%%%%%%%%%%%%%%%%%%%%%%%%%%%%%%%%%%%%%%%%%%%%%%%%%%%%%%%%%%
\setcounter{section}{-1}
\section{Introduction}\label{s:0}

The notion of factorization algebras, introduced by Costello and Gwilliam, is a mathematical formulation of the structure of observables in a field theory, whether in classical or quantum field theory. In \cite{CG1, CG2}, they developed the theory of factorization algebras and the Batalin-Vilkovisky (BV) formalism for perturbative quantum field theory. One of the main results of \cite{CG1, CG2} is the deformation quantization theorem, which states that the factorization algebra associated with a classical field theory admits a quantization when certain obstruction groups vanish. This quantization theorem is proved using the techniques of perturbative renormalization developed mathematically by Costello \cite{C}. 

On the other hand,  it is known that the notion of vertex algebras, introduced by Borcherds \cite{Bo}, is an algebraic framework of the chiral part of two-dimensional conformal field theory. Thus, it is naturally expected that factorization algebras on the complex plane $\bbC$ with suitable conditions recover vertex algebras. In fact, Costello and Gwilliam showed in \cite[\S5.3]{CG1} how to construct vertex algebras from prefactorization algebras called $S^1$-equivariant holomorphically translation invariant prefactorization algebras. Conversely, it is also expected that the construction of factorization algebras from vertex algebras is possible. Costello and Gwilliam  constructed using factorization envelopes, the Kac-Moody factorization algebra and the quantum observables of the $\beta\gamma$ system in \cite[\S5.4, \S5.5]{CG1}, which correspond respectively to the affine vertex algebra and the $\beta\gamma$ vertex algebra. Using the same method, Williams \cite{W} constructed Virasoro factorization algebra, which corresponds to the Virasoro vertex algebra. For general vertex algebras, Bruegmann \cite{B} constructed factorization algebras from vertex algebras. However, the connection between the factorization algebras constructed by him and those constructed as factorization envelops in \cite[\S5.4, \S5.5]{CG1} and \cite{W} is not established (see the example in the introduction of \cite{B}). 

We now remark that there is another notion of factorization algebras introduced by Beilinson and Drinfeld \cite{BD}. Beilinson-Drinfeld factorization algebras are essentially the same as chiral algebras, which provide an algebro-geometric formulation of two-dimensional conformal field theory. Furthermore, it is known that the category of translation-equivariant chiral algebras on the affine line $\mathbb{A}^1$ is equivalent to that of vertex algebras \cite[\S0.15]{BD} (see also \cite[Corollary A.2]{BDHK}).
In \cite[\S1.4.1]{CG1}, Costello and Gwilliam heuristically explain the relationship between the two notions of factorization algebras, but establishing this connection as a rigorous mathematical theorem remains an open problem. 

\medskip

This note began with the motivation to identify the class of factorization algebras that gives a categorical equivalence to vertex algebras. Unfortunately, this goal has not yet been achieved and appears to be very challenging. However, we have uncovered several new aspects of the relationship between vertex algebras and factorization algebras, which are outlined below. 

The first result is about the construction of vertex algebras from prefactorization algebras with values in the category $\LCS$ of locally convex spaces: 

\begin{thm0}[{\cref{thm:PFAVA}, \cref{prp:LPFAcomVA}}]
If  $S^1\ltimes \bbC$-equivariant prefactorization algebra $\clF$ on the complex number plane $\bbC$ with values in $\LCS$ is holomorphic (\cref{dfn:holoPFA}), then the linear space 
\[
\bfV(\clF)\ceq \bigoplus_{\Delta\in \bbZ}\clF^0_\Delta
\]
has the structure of a $\bbZ$-graded vertex algebra induced by the prefactorization algebra structure and the $S^1\ltimes \bbC$-equivariant structure of $\clF$. Here $\clF^0$ denotes the costalk at $0\in \bbC$, and $\clF^0_\Delta\subset \clF^0$ denotes the weight space with respect to the rotation action of $S^1$. Additionally, if $\clF$ is locally constant (\cref{dfn:LPFA}), then the vertex algebra $\bfV(\clF)$ is commutative. 
\end{thm0}

Our construction is similar to the one proposed by Costello and Gwilliam in \cite[\S5.3]{CG1} or the one by Bruegmann in \cite[\S1.3]{B}, but it offers some advantages: 

\begin{itemize}
\item 
As we mentioned above, Costello and Gwilliam constructed vertex algebras from prefactorization algebras called $S^1$-equivariant holomorphically translation invariant prefactorization algebras. These prefactorization algebras take values in the category of cochain complexes of differentiable vector spaces (see \cite[\S3.5.1]{CG1} for the definition of differentiable vector spaces). A notable feature of the category of differentiable vector spaces is that it is an abelian category containing many classes of topological vector spaces as subcategories. This makes it much easier to perform homological algebra since, as is well-known, the category of topological vector spaces is not abelian. However, as shown in \cite[Theorem 5.3.3]{CG1}, the vertex algebra structure emerges at the level of cohomology, so it is not essential for prefactorization algebras to take values in cochain complex in order to produce vertex algebras. Since the notion of  differentiable vector spaces is relatively unfamiliar, especially to algebraists, it seems better to adopt $\LCS$ as the target category of prefactorization algebras. 

\item
Bruegmann also constructed vertex algebras from prefactorization algebras in \cite[\S1.3]{B}. He adopted the category of bornological vector spaces as the target for prefactorization algebras. However, what he actually constructed from prefactorization algebras were geometric vertex algebras, and he also proved in \cite{B2} the categorical equivalence between geometric vertex algebras and vertex algebras, which is a simplified version of Huang's result in \cite{H}. Consequently, his construction is divided into two parts, making the direct connection between vertex algebras and prefactorization algebras unclear. In contrast, we construct vertex algebras directly. 

\item
In both the constructions by Costello and Gwilliam, and by Bruegmann, a discreteness condition is imposed on prefactorization algebras (see condition (iii) in \cite[Theorem 5.3.3]{CG1} and the first paragraph of \cite[\S1.3]{B}). We show that this discreteness condition is unnecessary, as expected by Bruegmann, using the complex analysis of functions with values in a locally convex space. 
\end{itemize}

\medskip

The second result concerns the inverse construction. We construct locally constant factorization algebras from commutative vertex algebras: 

\begin{thm0}[{\cref{prp:constLHFA}, \cref{prp:VbfVisom}}]\label{thm0:2}
For a commutative $\bbZ$-graded vertex algebra 
\[
V=\bigoplus_{\Delta\in \bbZ}V_\Delta
\]
 such that $V_\Delta=0$ ($\Delta\ll 0$) and $\dim_{\bbC}V_\Delta<\infty$ ($\Delta\in \bbZ$), one can construct a locally constant holomorphic factorization algebra $\bfF^{\loc}_{\ol{V}}$, which satisfies
\[
\bfV(\bfF^{\loc}_{\ol{V}})\cong V
\]
as $\bbZ$-graded vertex algebras. 
\end{thm0}

The construction of $\bfF^{\loc}_{\ol{V}}$ is briefly as follows: In \cref{ss:constLCFA}, we construct a locally constant factorization algebra $\bfF^{\loc}_A$ on $\bbC$ with values in the category $\Lin$ of linear spaces from a commutative algebra $A$. The condition $V_\Delta=0$ for $\Delta\ll 0$ guarantees that $\ol{V}\ceq \prod_{\Delta\in \bbZ}V_\Delta$ has the canonical structure of a commutative algebra induced from that of $V$. Thus, we obtain a locally constant factorization algebra $\bfF^{\loc}_{\ol{V}}$ with values in $\Lin$. Additionally, by using the conformal weight and the translation operator of $V$, we can define an $S^1\ltimes \bbC$-equivariant structure on $\bfF^{\loc}_{\ol{V}}$. Imposing the condition $\dim_{\bbC}V_\Delta<\infty$ for $\Delta\in \bbZ$, the canonical topology of $V_\Delta$ ensures that $\bfF^{\loc}_{\ol{V}}$ becomes a factorization algebra with values in $\LCS$. Finally, it can be verified that $\bfF^{\loc}_{\ol{V}}$ is holomorphic. 

As mentioned above, Bruegmann constructed factorization algebras from vertex algebras in \cite[\S1.4]{B}. The relationship between our construction and his has not been established yet. Furthermore, unfortunately, our construction cannot yet be extended to general vertex algebras. However, our approach leads to the following result about jet construction of differential algebras: 

\begin{thm0}[{\cref{prp:bfFjet}}]\label{thm0:3}
Let $A$ be a finitely generated commutative algebra, and denote by $\clJ\!A$ the (infinite) jet algebra of $A$. For a commutative $\bbZ$-graded vertex algebras $V=\bigoplus_{\Delta\in\bbZ}V_\Delta$ satisfying the conditions in \cref{thm0:2} and a morphism $\varphi\colon \bfF^{\loc}_A\to \bfF^{\loc}_{V}$ of factorization algebras, there exists a unique morphism $\wt{\varphi}\colon \bfF^{\loc}_{\ol{\clJ\!A}}\to \bfF^{\loc}_{\ol{V}}$ of $S^1\ltimes \bbC$-equivariant factorization algebras which commutes
\[
\begin{tikzcd}[row sep=huge, column sep=huge]
\bfF^{\loc}_A \arrow[r] \arrow[d, "\varphi"'] & 
\bfF^{\loc}_{\ol{\clJ\!A}} \arrow[d, "\wt{\varphi}"]\\
\bfF^{\loc}_{V} \arrow[r] & 
\bfF^{\loc}_{\ol{V}}
\end{tikzcd}
\]
Here horizontal arrows denote the morphisms of factorization algebras induced from the canonical injections $A\inj \clJ\!A\inj \ol{\clJ\!A}$ and $V\inj\ol{V}$. 
\end{thm0}

The jet construction plays an important role in the geometric study of vertex algebras: It is known that every vertex algebra is canonically filtered \cite{L}, and the spectrum of the associated graded algebra is called the singular support. The singular support is a vertex Poisson scheme, which can be viewed as a chiral analogue of Poisson schemes. Since vertex algebras are quite complicated, it is often hoped that the problems can be reduced to the study of geometric objects associated with a vertex algebra, such as the singular support. Research from this perspective has been growing; we mention here only \cite{A2, AM, A3}. Jet schemes of affine Poisson schemes are a typical example of vertex Poisson schemes \cite{A}. Therefore, it might be expected that \cref{thm0:3} offers new insights into such geometric studies of vertex algebras. 

\begin{Org}
\cref{s:FA} is an exposition of the notion of factorization algebras. Sections \cref{ss:Pcsh} to \cref{ss:equivFA} cover preliminaries on precosheaves, factorization algebras, and equivariant factorization algebras, respectively. In \cref{ss:exten}, we construct prefactorization algebras from those defined on a given open basis. This extension construction from an open basis generalizes the idea presented in the example of \cite[\S3.1.1]{CG1}, where the construction of factorization algebras on $\bbR$ from associative algebras over $\bbR$ is discussed. We note that our setting of construction differs slightly from that in \cite[\S7.2]{CG1} (see \cref{exten}). In \cref{ss:LFA}, we recall the definition of locally constant factorization algebras on $\bbC$, and discuss the basic properties of this class of factorization algebras. Although these properties may be well-known to experts, we could not find them in the literature, including \cite{CG1, CG2}. 

\cref{s:LCFAcomalg} provides the categorical equivalence between locally constant factorization algebras on $\bbC$ with values in $\Lin$ and commutative algebras over $\bbC$. Although this is a special case of the equivalence of $(\infty, 1)$-categories between locally constant factorization algebras on $\bbR^n$ and $E_n$-algebras stated in \cite[Theorem 6.4.2]{CG1}, we give an explicit proof since a comprehensive proof for the general case is not available. In \cref{ss:constcomalg}, we construct commutative algebras from locally constant prefactorization algebras. Conversely, in \cref{ss:constLCFA}, we construct locally constant factorization algebras from commutative algebras using the extension from an open basis discussed in \cref{ss:exten}. Finally, \cref{ss:adjcomalgLCPFA} gives a proof of the categorical equivalence. 

\cref{s:constVALFA} is the main part of this note. \cref{ss:VA} provides a preliminary on vertex algebras, including the jet construction of differential algebras. In \cref{ss:constVA}, we construct vertex algebras from prefactorization algebras on $\bbC$ with values in $\LCS$, called holomorphic. As mentioned above, our construction does not require prefactorization algebras to be discrete and is not divided into two parts via geometric vertex algebras. In \cref{ss:constcomVA}, we show that if holomorphic prefactorization algebras are assumed to be locally constant, then the resulting vertex algebras are commutative. \cref{ss:constLocFA} gives a partial inverse construction, namely, the construction of locally constant holomorphic factorization algebras from commutative vertex algebras. As explained briefly above, this construction is achieved by modifying the method in \cref{ss:constLCFA}. Furthermore, in this \cref{ss:constLocFA}, we discuss the relationship between our construction and jet algebras of commutative algebras. 

\cref{s:App} develops the calculus of functions with values in a locally convex space and is structured to prove the Laurent series expansion (\cref{thm:Laurent}). This appendix is used to remove the discreteness condition imposed on prefactorization algebras by Costello-Gwilliam and Bruegmann. While it is well-known that complex analysis holds for functions with values in a Banach space, we could not find literature that extends this to the setting of locally convex spaces. Although this generalization is a straightforward exercise, we give detailed proofs to ensure the completeness of this note and to make the content more accessible to algebraists familiar with vertex algebras. 
\end{Org}

\begin{Ackn}
I would like to thank my advisor Shintarou Yanagida for his kind support. This note was developed following several discussions held with him. I would also like to thank Alexander Schenkel for pointing out my misunderstandings in the first version of this note. Furthermore, I am grateful to Yuto Moriwaki for valuable discussions we had in Hakone and Nagoya. This work is supported by JSPS Research Fellowship for Young Scientists (No.\,23KJ1120).
\end{Ackn}

\begin{Notation}
Here we list the notations used throughout in this note. Some of them are also introduced in each section. 

\begin{itemize}
\item 
The symbol $\varnothing$ denotes the empty set. 

\item 
The symbol $\bbN$ denotes the set $\{0, 1, 2, \ldots\}$ of non-negative integers. 

\item 
The symbol $\delta_{m, n}$ denotes the Kronecker delta. 

\item 
For a positive integer $m\in \bbZ_{>0}$, we denote $[m]\ceq \{1, \ldots, m\}$. 

\item 
For sets $E_1, \ldots, E_n$, we denote $E_1\sqcup \cdots \sqcup E_n\ceq E_1\cup \cdots \cup E_n$ if $E_i\cap E_j=\varnothing$ for all $i\neq j$. 

\item 
For a topological space $X$, we denote by $\frU_X$ the set of all open subsets of $X$, and by $\frU_X(x)$ the set of all open neighborhood of $x\in X$. 

\item 
For a topological space $X$, an an open basis $\frB$ of $X$ and an open subset $U\subset X$, we denote 
\[
\frB(U)\ceq \{L\in \frB\mid L\subset U\}. 
\]

\item 
For $R\in \bbR_{>0}$ and $z\in \bbC$, we denote by $D_R(z)\ceq\{\zeta\in \bbC\mid |\zeta-z|<R\}$ the open disk, and by $D_R^{\times}(z)\ceq D_R(z)\bs\{z\}$ the punctured disk. Also, by convention we denote $D_{\infty}(z)\ceq \bbC$ for all $z\in \bbC$. 

\item
For $n\in \bbZ_{>0}$ and a subset $E\subset\bbC$, we denote 
\[
\Conf_n(E)\ceq\{(z_1, \ldots, z_n)\in E^n\mid z_i\neq z_j\ (i, j\in[n], i\neq j)\}.
\] 

\item 
We denote by $\Lin$ the category of linear spaces over $\bbC$. 

\item 
We denote by $\LCS$ the category of locally convex spaces over $\bbC$, whose morphisms are continuous linear maps.  

\item
A commutative $\bbC$-algebra always means a unital and associative one, and a morphism of commutative $\bbC$-algebras always means a $\bbC$-linear map which preserves the multiplications and the units. We denote by $\CAlg_{\bbC}$ the category of commutative $\bbC$-algebras. 
\end{itemize}
\end{Notation}

\section{Factorization algebras}\label{s:FA}

This section is an exposition of the notions of precosheaves, prefactorization algebras, and equivariant prefactorization algebras, which is minimum for our usage in this note. 

\subsection{Precosheaves}\label{ss:Pcsh}

In this subsection we recall the notion of precosheaves and costalks. 
For a topological space $X$, we denote by $\frU_X$ the set of all open subsets of $X$. 

\begin{dfn}\label{dfn:pcoshf}
Let $X$ be a topological space, and $\sfA$ a category. A map $\clF\colon \frU_X\to \Ob\sfA$ together with a family of morphisms
\[
\clF^U_V\colon \clF(U)\to \clF(V) \quad
(U, V\in \frU_X, U\subset V)
\] 
of $\sfA$ is called a \emph{precosheaf} on $X$ with valued in $\sfA$ if the following conditions hold: 
\begin{clist}
\item 
For open subsets $U, V, W\subset X$ such that $U\subset V\subset W$, we have $\clF^V_W\circ \clF^U_V=\clF^U_W$. 

\item 
For an open subset $U\subset X$, we have $\clF^U_U=\id_{\clF(U)}$.  
\end{clist}
\end{dfn}

Let $X$ be a topological space and $\sfA$ a category. The set $\frU_X$ is an ordered set with respect to inclusions, so it can be viewed as a category. A precosheaf $\clF$ is nothing but a functor $\clF\colon \frU_X\to\sfA$.

\begin{dfn}
Let $\clF, \clG\colon \frU_X\to \sfA$ be precosheaves and 
\[
\varphi_U\colon \clF(U)\to \clG(U)
\]
be a morphism of $\sfA$ for each open subset $U\subset X$. A family $\varphi=\{\varphi_U\}_{U\in \frU_X}$ is called a \emph{morphism of precosheaves} if 
\[
\clG^U_V\circ \varphi_U=\varphi_V\circ \clF^U_V
\]
for all open subsets $U\subset V\subset X$. 
%\[
%\begin{tikzcd}[row sep=huge, column sep=huge]
%\clF(U) \arrow[r, "\varphi_U"] \arrow[d, "\clF^U_V"'] & 
%\clG(U)\arrow[d, "\clG^U_V"]\\
%\clF(V) \arrow[r, "\varphi_V"']&
%\clG(V)
%\end{tikzcd}
%\]
In other words, a morphism $\varphi\colon \clF\to \clG$ of precosheves is a natural transformation from the functor $\clF$ to $\clG$. 
\end{dfn}

We denote by $\PCSh(X, \sfA)$ the category of precosheaves $\clF\colon \frU_X\to \sfA$. 

Now, fix a topological space $X$ and a category $\sfA$ which admits projective limits. 
For $x\in X$ and a precosheaf $\clF\colon \frU_X\to \sfA$, notice that 
\[
(\{\clF(U)\}_{U\in \frU_X(x)}, \{\clF^U_V\}_{V\supset U})
\]
is a projective system in $\sfA$. Here $\frU_X(x)$ denotes the set of all open neighborhood of $x$, and is ordered by reverse inclusion. 

\begin{dfn}
For $x\in X$ and a precosheaf $\clF\colon \frU_X\to \sfA$, let
\[
\clF^x\ceq \varprojlim_{U\in \frU_X(x)}\clF(U), 
\]
and denote by 
\[
\clF^x_U\colon \clF^x\to \clF(U)  
\]
the canonical projection for each $U\in \frU_X(x)$. The object $\clF^x$ is called the \emph{costalk} of $\clF$ at $x\in X$. 
\end{dfn}

Let  $x\in X$, and $\clF, \clG\colon \frU_X\to \sfA$ be precosheaves. For a morphism $\varphi\colon \clF\to \clG$ of precosheaves, there exists a unique morphism $\varphi^x\colon \clF^x\to\clG^x$ such that 
\[
\clG^x_U\circ \varphi^x=\varphi_U\circ \clF^x_U
\]
 for all $U\in \frU_X(x)$. 
%\[
%\begin{tikzcd}[row sep=huge, column sep=huge]
%\clF^x \arrow[r, "\varphi^x"] \arrow[d, "\clF^x_U"'] &
%\clG^x \arrow[d, "\clG^x_U"]\\
%\clF(U) \arrow[r, "\varphi_U"]&
%\clG(U)
%\end{tikzcd}
%\]
Thus, the map $\clF\mapsto \clF^x$ gives a functor 
\[
(-)^x\colon
\PCSh(X, \sfA)\to \sfA. 
\]

\subsection{Prefactorization algebras}

In this subsection we cite from \cite{CG1} the notion of prefactorization algebras. 
Let $X$ be a topological space, and $\sfM$ a symmetric monoidal category. We use the following notations: 
\begin{itemize}
\item 
For subsets $U_1, \ldots, U_n\subset X$, we denote $U_1\sqcup \cdots \sqcup U_n\ceq U_1\cup\cdots \cup U_n$ if $U_i  \cap U_j=\varnothing$ for all $i\neq j$. 

\item 
We denote by $\otimes$ the tensor product of $\sfM$. 

\item 
We denote by $1_{\sfM}$ the unit object of $\sfM$. 
\end{itemize}

\medskip

The following definition of prefactorization algebras is the same as \cite[Definition 1.2.1]{B}, and it appears different from that in \cite{CG1}. However, they are essentially the same. See \cref{rmk:PFA} after the definition. 

\begin{dfn}[{\cite[\S3.1.1, \S3.1.2]{CG1}}]\label{dfn:PFA}
Consider: 
\begin{itemize}
\item 
a precosheaf $\clF\colon \frU_X\to \sfM$, 

\item 
a family of morphisms in $\sfM$
\[
\clF^{U, V}_W\colon \clF(U)\otimes \clF(V)\to \clF(W) \quad
(U, V, W\in \frU_X, U\sqcup V\subset W),
\]
called the multiplication,

\item 
a morphism $\eta\colon 1_{\sfM}\to \clF(\varnothing)$, called the unit.
\end{itemize}
A tuple $(\clF, \{\clF^{U, V}_W\}, \eta)$ is called a \emph{prefactorization algebra} on $X$ with valued in $\sfM$ if it satisfies the following conditions: 
\begin{clist}
\item 
For open subsets $U_1, U_2, V, W\subset X$ such that $U_1\sqcup U_2\subset V\subset W$, we have
\[
\clF^V_W\circ \clF^{U_1, U_2}_V=\clF^{U_1, U_2}_W. 
\]
%\[
%\begin{tikzcd}[row sep=huge, column sep=huge]
%\clF(U_1)\otimes \clF(U_2) \arrow[r, "\clF^{U_1, U_2}_V"] \arrow[rd, "\clF^{U_1, U_2}_W"'] & \clF(V) \arrow[d, "\clF^V_W"]\\
%& \clF(W)
%\end{tikzcd}
%\]
Also, for open subsets $U_1, U_2, V_1, V_2, W\subset X$ such that
\[
U_1\subset V_1, \quad
U_2\subset V_2, \quad
V_1\sqcup V_2\subset W, 
\]
we have
\[
\clF^{V_1, V_2}_W\circ (\clF^{U_1}_{V_1}\otimes \clF^{U_2}_{V_2})=\clF^{U_1, U_2}_W. 
\]
%\[
%\begin{tikzcd}[row sep=huge, column sep=huge]
%\clF(U_1)\otimes \clF(U_2) \arrow[r, "\clF^{U_1}_{V_1}\otimes \clF^{U_2}_{V_2}"] \arrow[rd, "\clF^{U_1, U_2}_W"'] & \clF(V_1)\otimes \clF(V_2) \arrow[d, "\clF^{V_1, V_2}_W"] \\
%& \clF(W)
%\end{tikzcd}
%\]

\item 
For open subsets $U_1, U_2, V\subset X$ such that $U_1\sqcup U_2\subset V$, the following diagram commutes: 
\[
\begin{tikzcd}[row sep=huge, column sep=huge]
\clF(U_1)\otimes \clF(U_2) \arrow[r, "\sim"] \arrow[rd, "\clF^{U_1, U_2}_V"'] & \clF(U_2)\otimes \clF(U_1) \arrow[d, "\clF^{U_2, U_1}_V"]\\
& \clF(V)
\end{tikzcd}
\]
Here $\clF(U_1)\otimes \clF(U_2)\lsto \clF(U_2)\otimes \clF(U_1)$ denotes the brading of $\sfM$. 

\item (Associativity) 
For open subsets $U_1, U_2, U_3, V_1, V_2, W\subset X$ such that
\[
U_1\sqcup U_2\subset V_1,\quad 
U_2\sqcup U_3\subset V_2,\quad 
U_1\sqcup V_2\subset W, \quad
U_3\sqcup V_1\subset W, 
\]
the following diagram commutes: 
\begin{center}
\begin{tikzpicture}[auto]
\node (1) at (-3., 2) {$(\clF(U_1)\otimes \clF(U_2))\otimes \clF(U_3)$}; \node (2) at (3., 2) {$\clF(U_1)\otimes (\clF(U_2)\otimes \clF(U_3))$};
\node (4) at (-3, 0) {$\clF(V_1)\otimes \clF(U_3)$};
\node (5) at (3, 0) {$\clF(U_1)\otimes \clF(V_2)$};
\node (3) at (0, -1.5) {$\clF(W)$}; 
\draw[->] (1) to node {$\sim$} (2);
\draw[->] (1) to node [swap]{$\scriptstyle \clF^{U_1, U_2}_{V_1}\otimes \id_{\clF(U_3)}$} (4);
\draw[->] (2) to node {$\scriptstyle \id_{\clF(U_1)}\otimes\clF^{U_2, U_3}_{V_2}$} (5);
\draw[->] (4) to node [swap]{$\scriptstyle \clF^{V_1, U_3}_W$} (3);
\draw[->] (5) to node {$\scriptstyle \clF^{U_1, V_2}_W$} (3); 
\end{tikzpicture}
\end{center}
Here $(\clF(U_1)\otimes \clF(U_2))\otimes \clF(U_3)\lsto \clF(U_1)\otimes (\clF(U_2)\otimes \clF(U_3))$ denotes the associator of $\sfM$. 

\item (Unit) 
For an open subset $U\subset X$, the following diagrams commute: 
\[
\begin{tikzcd}[row sep=huge, column sep=huge]
1_{\sfA}\otimes \clF(U) \arrow[r, "\eta\otimes \id_{\clF(U)}"] \arrow[rd, "\sim"', sloped] &  \clF(\varnothing) \otimes \clF(U) \arrow[d, "\clF^{\varnothing, U}_U"] \\
& \clF(U)
\end{tikzcd}
\]
Here $1_{\sfM}\otimes \clF(U)\lsto \clF(U)$ denotes the left unitor of $\sfM$. 
\end{clist}
\end{dfn}

\begin{rmk}\label{rmk:PFA}
Let $(\clF, \{m^{U_1, \ldots, U_n}_V\}, \eta)$ be an unital prefactorization algebra in the sense of \cite[\S3.1.1, \S3.1.2]{CG1}, i.e., $(\clF, \{m^{U_1, \ldots, U_n}_V\}, \eta)$ is a data consisting of
\begin{itemize}
\item 
a map $\clF\colon \frU_X\to \Ob\sfM$, 

\item 
a morphism 
\[
m^{U_1, \ldots, U_n}_V\colon \bigotimes_{i=1}^n\clF(U_i)\to \clF(V)
\]
for open subsets $U_1, \ldots, U_n, V\subset X$ such that $U_1\sqcup \cdots \sqcup U_n\subset V$, 

\item 
a morphism $\eta\colon 1_{\sfM}\to \clF(\varnothing)$,
\end{itemize}
which satisfies the certain compatibility. Then $\clF$ is a precosheaf by $\clF^U_V\ceq m^U_V$, and the precosheaf $\clF\colon \frU_X\to \sfM$ is a prefactorization algebra in the sense of \cref{dfn:PFA} by the multiplication $\clF^{U, V}_W\ceq m^{U, V}_W$ and the unit $\eta$. 
Conversely, let $(\clF, \{\clF^{U, V}_W\}, \eta)$ be a prefactorization algebra in the sense of \cref{dfn:PFA}. For $n\in \bbN$, $n\ge 2$ and open subsets $U_1, \ldots, U_n, V\subset X$ such that $U_1\sqcup \cdots \sqcup U_n\subset V$, define a morphism 
\[
\clF^{U_1, \ldots, U_n}_V\colon \bigotimes_{i=1}^n\clF(U_i)\to \clF(V)
\]
by inductively 
\[
\clF^{U_1, \ldots, U_n}_V\ceq
\clF^{U_1\sqcup \cdots \sqcup U_{n-1}, U_n}_V\circ
\bigl(\clF^{U_1, \ldots, U_{n-1}}_{U_1\sqcup \cdots \sqcup U_{n-1}}\otimes \id_{\clF(U_n)}\bigr). 
\]
Then $(\clF, \{\clF^{U_1, \ldots, U_n}_V\}, \eta)$ is an unital prefactorization algebra in the sense of \cite[\S3.1.1, \S3.1.2]{CG1}. 
\end{rmk}

\begin{dfn}[{\cite[\S3.1.4]{CG1}}]\label{dfn:morPFA}
Let $\clF, \clG\colon \frU_X\to \sfM$ be prefactorization algebras. A morphism $\varphi\colon \clF\to \clG$ of precosheaves is called a \emph{morphism of prefactorization algebras} if the following conditions hold: 
\begin{clist}
\item 
For open subsets $U, V, W\subset X$ such that $U\sqcup V\subset W$, we have 
\[
\varphi_W\circ \clF^{U, V}_W=\clG^{U, V}_W\circ (\varphi_U\otimes \varphi_V). 
\]
%\[
%\begin{tikzcd}[row sep=huge, column sep=huge]
%\clF(U)\otimes \clF(V) \arrow[r, "\clF^{U, V}_W"] \arrow[d, "\varphi_U\otimes \varphi_V"'] & 
%\clF(W) \arrow[d, "\varphi_W"]\\
%\clG(U)\otimes \clG(V) \arrow[r, "\clG^{U, V}_W"'] &
%\clG(W)
%\end{tikzcd}
%\]

\item 
We have $\varphi_{\varnothing}\circ \eta_{\clF}=\eta_{\clG}$, where $\eta_{\clF}$ and $\eta_{\clG}$ denotes the units. 
%\[
%\begin{tikzcd}[row sep=huge, column sep=huge]
%1_{\sfM} \arrow[r, "\eta_{\clF}"] \arrow[rd, "\eta_{\clG}"']& 
%\clF(\varnothing) \arrow[d, "\varphi_{\varnothing}"] \\
%& 
%\clG(\varnothing)
%\end{tikzcd}
%\]
\end{clist}
\end{dfn}

For a topological space $X$, and a symmetric monoidal category $\sfM$, we denote by $\PFA(X, \sfM)$ the category of prefactorization algebras $\clF\colon \frU_X\to \sfM$. 

\medskip

Next, we turn to define the notion of factorization algebras, following \cite[\S6.1]{CG1}.  

Let $X$ be a topological space, $\sfA$ a category which admits coproducts, and $\clF\colon \frU_X\to \sfA$ a precosheaf. 
For an open subset $U\subset X$ and an open cover $\{U_i\}_{i\in I}$ of $U$, consider the morphisms 
\begin{align*}
&\iota_i\circ \clF^{U_i\cap U_j}_{U_i}\colon \clF(U_i\cap U_j)\to \bigoplus_{i\in I}\clF(U_i) \quad (i, j\in I), \\
&\iota_j\circ \clF^{U_i\cap U_j}_{U_j}\colon \clF(U_i\cap U_j)\to \bigoplus_{i\in I}\clF(U_i) \quad (i, j\in I), \\
&\clF^{U_i}_U\colon \clF(U_i)\to \clF(U) \quad(i\in I). 
\end{align*}
Here $\iota_i\colon \clF(U_i)\to \bigoplus_{i\in I}\clF(U_i)$ denotes the canonical morphism of the coproduct. These morphisms induce the following morphisms in $\sfA$ respectively: 
\begin{align}\label{eq:coeq}
\begin{split}
&p\colon \bigoplus_{i, j\in I}\clF(U_i\cap U_j)\to \bigoplus_{i\in I}\clF(U_i), \\
&q\colon \bigoplus_{i, j\in I}\clF(U_i\cap U_j)\to \bigoplus_{i\in I}\clF(U_i), \\
&\pi\colon \bigoplus_{i\in I}\clF(U_i)\to \clF(U). 
\end{split}
\end{align}

\begin{rmk}
Suppose that $\sfA$ admits any colimits. A precosheaf $\clF$ is called a cosheaf if for an open subset $U\subset X$ and an open cover $\{U_i\}_{i\in I}$ of $U$, the diagram \eqref{eq:defcosh} below gives a coequalizer. A factorization algebra is a kind of cosheaf but we adopt a different Grothendieck topology. 
\end{rmk}

\begin{dfn}[{\cite[Definition 6.1.1]{CG1}}]
Let $X$ be a topological space and $U\subset X$ an open subset. An open cover $\{U_i\}_{i\in I}$ of $U$ is called a \emph{Weiss cover} if for finitely many points $x_1, \ldots, x_n\in U$, there exists $i\in I$ such that $x_1, \ldots, x_n\in U_i$. 
\end{dfn}

Now, fix a topological space $X$ and a symmetric monoidal category $\sfM$ which admits colimits. 

\begin{dfn}[{\cite[Definition 6.1.3]{CG1}}]\label{dfn:FA}
A prefactorization algebra $\clF\colon \frU_X\to \sfM$ is called a \emph{factorization algebra} if it satisfies the following conditions:
\begin{clist}
\item 
For an open subset $U\subset X$ and Weiss cover $\{U_i\}_{i\in I}$, 
\begin{equation}\label{eq:defcosh}
\begin{tikzcd}
\displaystyle\bigoplus_{i, j\in I}\clF(U_i\cap U_j) \arrow[r, shift left=.75ex,"p"] \arrow[r, shift right=.75ex,swap,"q"] 
& \displaystyle\bigoplus_{i\in I}\clF(U_i) \arrow[r, "\pi"]
& \clF(U)
\end{tikzcd}
\end{equation}
is a diagram of coequalizer. Here we use the notations defined in \eqref{eq:coeq}. 

\item 
For open subsets $U, V\subset X$ such that $U\cap V=\varnothing$, the multiplication 
\[
\clF^{U, V}_{U\sqcup V}\colon \clF(U)\otimes \clF(V)\to \clF(U\sqcup V)
\]
is an isomorphism. 
\end{clist} 
\end{dfn}

Let $\clF\colon \frU_X\to \sfM$ be a factorization algebra and $\frB$ an open basis of $X$ such that for each open subset $U\subset X$, the family 
\[
\frB(U)\ceq \{L\in \frB\mid L\subset U\}
\]
is a Weiss cover of $U$. For an open subset $U$, notice that 
\[
\bigl(\{\clF(L)\}_{L\in \frB(U)}, \{\clF^L_M\}_{L\subset M}\bigr)
\]
is an inductive system in $\sfM$, where $\frB(U)$ is ordered by inclusion. 

\begin{prp}\label{prp:FUindlim}
Let $\clF\colon \frU_X\to \sfM$ be a factorization algebra and $\frB$ an open basis of $X$ such that for each open subset $U\subset X$, the family $\frB(U)$ is a Weiss cover of $U$. For an open subset $U\subset X$, we have
\[
\clF(U)=\varinjlim_{L\in \frB(U)}\clF(L)
\]
with the canonical morphism $\clF^L_U\colon \clF(L)\to \clF(U)$. 
\end{prp}

\begin{proof}
It is clear that $\clF^M_U\circ \clF^L_M=\clF^L_U$ for $L, M\in \frB(U)$, $L\subset M$. Let $A\in \Ob\sfM$, and $\{f_L\colon \clF(L)\to A\}_{L\in \frB(U)}$ be a family of morphisms such that
$f_M\circ \clF^L_M=f_M$ for $L, M\in\frB(U)$, $L\subset M$. To construct a morphism $\clF(U)\to A$, we use the coequalizer diagram \eqref{eq:defcosh}. By the universality of the coproduct $\bigoplus_{L\in \frB(U)}\clF(L)$, there exists a unique morphism $f\colon \bigoplus_{L\in \frB(U)}\clF(L)\to A$ which commutes
\[
\begin{tikzcd}[row sep=huge, column sep=huge]
\clF(L) \arrow[r, hookrightarrow, "\iota_L"] \arrow[rd, "f_L"']& 
\displaystyle\bigoplus_{L\in \frB(U)}\clF(L) \arrow[d, "f"] \\
& 
A
\end{tikzcd}
\]
We claim that $f\circ p=f\circ q$. To prove this, it is enough to show that $f\circ p\circ \iota_{L, M}=f\circ q\circ \iota_{L, M}$ for each $L, M\in \frB(U)$, where $\iota_{L, M}\colon \clF(L\cap M) \to \bigoplus_{L, M\in \frB(U)}\clF(L\cap M)$ denotes the canonical inclusion, or equivalently $f_L\circ \clF^{L\cap M}_L=f_M\circ \clF^{L\cap M}_M$. Since we have a coequalizer diagram
\[
\begin{tikzcd}
\displaystyle\bigoplus_{L', M'\in \frB(L\cap M)}\clF(L'\cap M') \arrow[r, shift left=.75ex,"p'"] \arrow[r, shift right=.75ex,swap,"q'"] 
& \displaystyle\bigoplus_{L'\in \frB(L\cap M)}\clF(L') \arrow[r, "\pi'"]
& \clF(L\cap M), 
\end{tikzcd}
\]
the equation $f_L\circ \clF^{L\cap M}_L=f_M\circ \clF^{L\cap M}_M$ follows from $f_L\circ \clF^{L\cap M}_L\circ \pi'=f_M\circ \clF^{L\cap M}_M\circ \pi'$, but this is clear since 
\[
f_L\circ \clF^{L\cap M}_L\circ \pi'\circ \iota_{L'}=
f_{L'}=
f_M\circ \clF^{L\cap M}_M\circ \pi'\circ \iota_{L'} \quad
(L'\in \frB(L\cap M)). 
\]
Thus, there is a unique morphism $\wt{f}\colon \clF(U)\to A$ such that $\wt{f}\circ \pi=f$. We find that $\wt{f}$ is a unique morphism satisfying $f\circ \clF^L_U=f_L$ for each $L\in \frB(U)$. 
\end{proof}

\subsection{Equivariant prefactorization algebras}\label{ss:equivFA}

In this subsection we cite from \cite{CG1} the notion of equivariant prefactorization algebras.  
Let $X$ be a topological space, $\sfM$ a symmetric monoidal category, and $G$ a group acting on $X$. For a precosheaf $\clF\colon \frU_X\to \sfM$ and $a\in G$, a map
\[
a\clF\colon \frU_X\to \Ob\sfM, \quad U\mapsto \clF(aU)
\]
is a precosheaf by letting
\[
(a\clF)^U_V\ceq \clF^{aU}_{aV}\colon (a\clF)(U)\to (a\clF)(V)
\]
for each open subsets $U\subset V\subset X$. 

\begin{dfn}[{\cite[Definition 3.7.1]{CG1}}]\label{dfn:equivPFA}
A prefactorization algebra $\clF\colon \frU_X\to \sfM$ equipped with a morphism
\[
\sigma_a=\{\sigma_{a, U}\}_{U\in \frU_X}\colon 
\clF\to a\clF
\]
of precosheaves for each $a\in G$ is called a \emph{$G$-equivariant prefactorization algebra} if the following conditions hold: 
\begin{clist}
\item 
For $a, b\in G$ and an open subset $U\subset X$, we have $\sigma_{ab, U}=\sigma_{a, bU}\circ \sigma_{b, U}$. 

\item 
For the unit $1_G\in G$ and an open subset $U\subset X$, we have $\sigma_{1_G, U}=\id_{\clF(U)}$. 

\item 
For $a\in G$ and open subsets $U, V, W\subset X$ such that $U\sqcup V\subset W$, we have 
\[
\sigma_{a, W}\circ \clF^{U, V}_W=\clF^{aU, aV}_{aW}\circ (\sigma_{a, U}\otimes \sigma_{a, V}).
\]

\item 
For $a\in G$, we have $\sigma_{a, \varnothing}\circ \eta=\eta$, where $\eta\colon 1_{\sfM}\to \clF(\varnothing)$ denotes the unit of $\clF$. 
\end{clist}
\end{dfn}

\begin{dfn}\label{dfn:morequivPFA}
Let $\clF, \clG\colon \frU_X\to \sfM$ be $G$-equivariant prefactorization algebras. A morphism $\varphi\colon \clF\to \clG$ of prefactorization algebras is called a \emph{morphism of $G$-equivariant prefactorization algebras} if 
\[
\tau_{a, U}\circ \varphi_U=\varphi_{aU}\circ \sigma_{a, U}
\]
%\[
%\begin{tikzcd}[row sep=huge, column sep=huge]
%\clF(U) \arrow[r, "\varphi_U"] \arrow[d, "\sigma_{a, U}"']& 
%\clG(U) \arrow[d, "\tau_{a, U}"] \\
%\clF(aU) \arrow[r, "\varphi_{aU}"']&
%\clG(aU)
%\end{tikzcd}
%\]
for each $a\in G$ and an open $U\subset X$. Here $\sigma_a\colon \clF\to a\clF$ and $\tau_a\colon\clG\to a\clG$ denote the $G$-equivariant structure on $\clF$ and $\clG$ respectively. 
\end{dfn}

For a topological space $X$, a symmetric monoidal category $\sfM$, and a group $G$ acting on $X$,  we denote by $\PFA_G(X, \sfM)$ the category of $G$-equivariant prefactorization algebras $\clF\colon \frU_X\to \sfM$. 

\subsection{Extension from an open basis}\label{ss:exten}

In this subsection we discuss the extension construction of prefactorization algebras from those defined on a give open basis. This construction is slightly different from that in \cite[\S7.2]{CG1} to use later in \cref{ss:constLCFA} (see \cref{exten} below). 
Let $X$ be a topological space, and $\sfA$ a category which admits inductive limits. 
For an open basis $\frB$ of $X$ and an open subset $U\subset X$, we denote
\[
\frB(U)\ceq \{L\in \frB\mid L\subset U\}.
\]

Let $\frB$ be an open basis, and $F\colon \frB\to \sfA$ a precosheaf, i.e., a functor from $\frB$ to $\sfA$. For an open subset $U\subset X$, 
\[
(\{F(L)\}_{L\in \frB(U)}, \{F^L_M\}_{L\subset M})
\]
is an inductive system in $\sfA$, where $\frB(U)$ is ordered by inclusion. We define
\[
\clE_F(U)\ceq \varinjlim_{L\in \frB(U)}F(L)
\]
and denote by
\[
\wt{F}^L_U\colon F(L)\to \clE_F(U)
\]
the canonical morphism for each $L\in \frB(U)$. For open subsets $U\subset V \subset X$, the universality of $\clE_F(U)=\varinjlim_{L\in \frB(U)}F(L)$ yields a unique morphism
\[
(\clE_F)^U_V\colon \clE_F(U)\to \clE_F(V)
\]
which satisfies
\[
(\clE_F)^U_V\circ \wt{F}^L_U=\wt{F}^L_V
\]
%\[
%\begin{tikzcd}[row sep=huge, column sep=huge]
%F(L) \arrow[r, "\wt{F}^L_U"] \arrow[rd, "\wt{F}^L_V"']& 
%\clE_F(U)\arrow[d, "(\clE_F)^U_V"]\\
%&
%\clE_F(V), 
%\end{tikzcd}
%\]
for each $L\in \frB(U)$. Hence the map 
\[
\clE_F\colon \frU_X\to \Ob\sfA, \quad
U\mapsto \clE_F(U)
\]
is a precosheaf by $(\clE_F)^U_V$. Notice that $\wt{F}^L_L\colon F(L)\to \clE_F(L)$ is an isomorphism for $L\in \frB$. 

Let $F, G\colon \frB\to \sfA$ be precosheaves, and $\varphi\colon F\to G$ be a morphism of precosheaves, i.e., a natural transformation from $F$ to $G$. For an open subset $U\subset X$, the universality of $\clE_F(U)$ induces a unique morphism 
\[
\clE_{\varphi, U}\colon \clE_F(U)\to \clE_G(U), 
\]
which satisfies
\[
\clE_{\varphi, U}\circ \wt{F}^L_U=\wt{G}^L_U\circ \varphi_L
\]
%\[
%\begin{tikzcd}[row sep=huge, column sep=huge]
%F(L) \arrow[r, "\wt{F}^L_U"] \arrow[d, "\varphi_L"']& 
%\clE_F(U)\arrow[d, "\clE_{\varphi, U}"]\\
%G(L) \arrow[r, "\wt{G}^L_U"']& 
%\clE_G(U)
%\end{tikzcd}
%\]
for all $L\in \frB(U)$. We find that $\clE_\varphi\colon \clE_F\to \clE_G$ is a morphism of precosheaves. Thus, the map $F\mapsto \clE_F$ gives a functor
\[
\clE\colon \PCSh(\frB, \sfA)\to \PCSh(X, \sfA), 
\]
where $\PCSh(\frB, \sfA)$ denotes the category of precosheaves on $\frB$. 

\medskip

Now, fix 
\begin{itemize}
\item
a symmetric monoidal category $\sfM$ such that inductive limits exists, and any inductive limits commute with the tensor product, 

\item
an open basis $\frB$ satisfying: 
\begin{clist}
\item
$\varnothing\in \frB$. 

\item
For $L, M\in \frB$ such that  $L\cap M=\varnothing$, we have $L\sqcup M\in \frB$. 
\end{clist}
\end{itemize}

Let $F\colon \frB\to \sfM$ be a prefactorization algebra. 
For open subsets $U, V, W\subset X$ such that $U\sqcup V\subset W$, the universality of 
\[
\clE_F(U)\otimes \clE_F(V)=
\varinjlim_{\substack{L\in \frB(U)\\ M\in \frB(V)}}F(L)\otimes F(M)
\]
yields an unique morphism
\[
(\clE_F)^{U, V}_W\colon \clE_F(U)\otimes \clE_F(V)\to \clE_F(W)
\]
which satisfies
\[
(\clE_F)^{U, V}_W\circ (\wt{F}^L_U\otimes \wt{F}^M_V)=
\wt{F}^{L\sqcup M}_W\circ F^{L, M}_{L\sqcup M}
\]
%\[
%\begin{tikzcd}[row sep=huge, column sep=huge]
%F(L)\otimes F(M) \arrow[r, "\wt{F}^L_U\otimes \wt{F}^M_V"] \arrow[d, "F^{L, M}_{L\sqcup M}"']&
%\clE_F(U)\otimes \clE_F(V) \arrow[d, "(\clE_F)^{U, V}_W"]\\
%F(L\sqcup M) \arrow[r, "\wt{F}^{L\sqcup M}_W"']&
%\clE_F(W)
%\end{tikzcd}
%\]
for all $L\in \frB(U)$ and $M\in \frB(V)$. The precosheaf $\clE_F\colon \frU_X\to \sfM$ is a prefactorization algebra by the multiplication $(\clE_F)^{U, V}_W$ and the unit
\[
\begin{tikzcd}
1_{\sfA} \arrow[r]&
F(\varnothing) \arrow[r, "\wt{F}^{\varnothing}_{\varnothing}"] &
\clE_F(\varnothing),  
\end{tikzcd}
\]
where $1_{\sfA}\to F(\varnothing)$ indicates the unit of $F$. 

\begin{prp}\label{prp:extmultisom}
Let $F\colon \frB\to \sfM$ be a prefactorization algebra such that 
\[
F^{L, M}_{L\sqcup M}\colon F(L)\otimes F(M)\to F(L\sqcup M)
\]
is an isomorphism for each $L, M\in \frB$, $L\cap M=\varnothing$. For open subsets $U, V\subset X$ such that $U\cap V=\varnothing$, if 
\[
\frB(U\sqcup V)=\{L\sqcup M\mid L\in \frB(U), M\in \frB(V)\}, 
\]
then
\[
(\clE_F)^{U, V}_{U\sqcup V}\colon \clE_F(U)\otimes \clE_F(V)\to \clE_F(U\sqcup V)
\]
is an isomorphism. 
\end{prp}

\begin{proof}
For $L\in \frB(U)$ and $M\in \frB(V)$, denote
\[
\varphi_{L, M}
\ceq 
\wt{F}^{L\sqcup M}_{U\sqcup V}\circ F^{L, M}_{L\sqcup M}
\colon 
F(L)\otimes F(M)\to \clE_F(U\sqcup V). 
\]
It is enough to prove that $(\clE_F(U\sqcup V), \{\varphi_{L, M}\}_{L\in \frB(U), M\in \frB(V)})$ is the inductive limit of 
\[
\bigl(\{F(L)\otimes F(M)\}_{L\in \frB(U), M\in \frB(V)}, \{F^L_{L'}\otimes F^M_{M'}\}_{L\subset L', M\subset M'}
\bigr). 
\]
First, a direct calculation shows that 
\[
\varphi_{L', M'}\circ (F^L_{L'}\otimes F^M_{M'})=\varphi_{L, M} 
\quad
(L, L'\in \frB(U), M, M'\in \frB(V), L\subset L', M\subset M'). 
\]
Let $A\in \Ob\sfM$, and $\{\alpha_{L, M}\colon F(L)\otimes F(M)\to A\}_{L\in \frB(U), M\in \frB(V)}$ be a family of morphisms satisfying
\[
\alpha_{L', M'}\circ (F^L_{L'}\otimes F^M_{M'})=\alpha_{L, M} 
\quad
(L, L'\in \frB(U), M, M'\in \frB(V), L\subset L', M\subset M'). 
\]
Denote for $L\in \frB(U)$ and $M\in \frB(V)$, 
\[
\wt{\alpha}_{L, M}\ceq \alpha_{L, M}\circ (F^{L, M}_{L\sqcup M})^{-1}\colon F(L\sqcup M)\to A, 
\]
then we have
\[
\wt{\alpha}_{L', M'}\circ F^{L\sqcup M}_{L'\sqcup M'}
=
\wt{\alpha}_{L, M}
\quad (L, L'\in \frB(U), M, M'\in \frB(V), L\subset L', M\subset M'). 
\]
Thus, by the assumption for $\frB(U\sqcup V)$, the universality of $\clE_F(U\sqcup V)$ yields an unique morphism
\[
\alpha\colon \clE_F(U\sqcup V)\to A
\]
such that 
\[
\alpha\circ \wt{F}^{L\sqcup M}_{U\sqcup V}
=
\wt{\alpha}_{L, M}
\quad (L\in \frB(U), M\in \frB(V)). 
\]
We find that $\alpha$ is an unique morphism satisfying 
\[
\alpha\circ \varphi_{L, M}
=
\alpha_{L, M}
\quad(L\in \frB(U), M\in \frB(V)), 
\]
and hence the claim follows. 
\end{proof}

Let $F, G\colon \frB\to \sfM$ be prefactorization algebras, and $\varphi\colon F\to G$ be a morphism of prefactorization algebras. Then, we find that $\clE_\varphi\colon \clE_F\to \clE_G$ is a morphism of prefactorization algebras. Thus, the functor $\clE\colon \PCSh(\frB, \sfA)\to \PCSh(X, \sfA)$ gives rise to the functor
\[
\clE\colon \PFA(\frB, \sfM)\to \PFA(X, \sfM), 
\]
where $\PFA(\frB, \sfM)$ denotes the category of prefactorization algebras on $\frB$. 

\begin{rmk}\label{exten}
In \cite[\S7.2]{CG1}, Costello and Gwilliam show that the categorical equivalence between factorization algebras on $X$ and factorization algebras on a given factorization basis of $X$ (\cite[Proposition 7.2.3]{CG1}, see also \cite[Definition 7.2.1]{CG1}). Our construction $\clE_F$ from an open basis $\frB$ does not require $\frB$ to be a factorization basis since we will construct a locally constant factorization algebra in \cref{ss:constLCFA} from an open basis which is not closed under finite intersections. 
\end{rmk}

Next, fix a group $G$ acting on $X$, and suppose that $aL\in \frB$ for all $a\in G$ and $L\in \frB$. 

Let $F\colon \frB\to \sfM$ be a $G$-equivariant prefactorization algebra. We denote its structure morphism by
\[
\sigma_{a, L}\colon F(L)\to F(aL) \quad (a\in G, L\in \frB). 
\]
For $a\in G$ and an open subset $U\subset X$, the universality of $\clE_F(U)=\varinjlim_{L\in \frB(U)}F(L)$ yields an unique morphism
\[
\wt{\sigma}_{a, U}\colon 
\clE_F(U)\to \clE_F(aU)
\]
which satisfies
\[
\wt{\sigma}_{a, U}\circ \wt{F}^L_U=
\wt{F}^{aL}_{aU}\circ \sigma_{a, L}
\]
%\[
%\begin{tikzcd}[row sep=huge, column sep=huge]
%F(L) \arrow[r, "\wt{F}^L_U"] \arrow[d, "\sigma_{a, L}"']& 
%\clE_F(U) \arrow[d, "\wt{\sigma}_{a, U}"] \\
%F(aL) \arrow[r, "\wt{F}^{aL}_{aU}"']&
%\clE_F(aU)
%\end{tikzcd}
%\]
for each $L\in \frB(U)$. The prefactorization algebra $\clE_F\colon \frU_X\to \sfM$ is a $G$-equivariant prefactorization algebra by the structure morphism $\wt{\sigma}_{a, U}$. 

Let $F, F'\colon \frB\to \sfM$ be $G$-equivariant prefactorization algebras, and $\varphi\colon F\to F'$ be a morphism of $G$-equivariant prefactorization algebras. 
Then, we find that $\clE_\varphi\colon \clE_F\to \clE_{F'}$ is a morphism of $G$-equivariant prefactorization algebras. Thus, the functor $\clE\colon\PFA(\frB, \sfA)\to \PFA(X, \sfA)$ gives rise to the functor
\[
\clE\colon \PFA_G(\frB, \sfM)\to \PFA_G(X, \sfM), 
\]
where $\PFA_G(\frB, \sfM)$ denotes the category of $G$-equivariant prefactorization algebras on $\frB$. 

\subsection{Locally constant factorization algebras on $\bbC$}\label{ss:LFA}

In this subsection we recall the notion of locally constant factorization algebras on $\bbC$, and discuss the basic properties of this class of factorization algebras. 
Fix a symmetric monoidal category $\sfM$ which admits any colimits. 

\begin{dfn}[{\cite[Definition 6.4.1]{CG1}}]\label{dfn:LPFA}
A prefactorization algebra $\clF\colon \frU_{\bbC}\to \sfM$ is called locally constant if for $0<r<R\le \infty$ and $z, w\in \bbC$ such that $D_r(z)\subset D_R(w)$, the morphism $\clF^{D_r(z)}_{D_R(w)}\colon \clF(D_r(z))\to \clF(D_R(w))$ is an isomorphism in $\sfM$. 
\end{dfn}

We denote by $\PFA^{\mathsf{loc}}(\bbC, \sfM)$ the full subcategory of $\PFA(\bbC, \sfM)$ whose objects are locally constant prefactorization algebras. Also, denote by $\FA^{\mathsf{loc}}(\bbC, \sfM)$ the full subcategory of $\PFA^{\mathsf{loc}}(\bbC, \sfM)$ consisting of locally constant factorization algebras. 

\begin{lem}\label{lem:locFA}
A factorization algebra $\clF\colon \frU_{\bbC}\to \sfM$ is locally constant if and only if the following condition holds: For $0<r<R<\infty$ and $z, w\in \bbC$ such that $D_r(z)\subset D_R(w)$, the morphism $\clF^{D_r(z)}_{D_R(w)}\colon \clF(D_r(z))\to \clF(D_R(w))$ is an isomorphism. 
\end{lem}

\begin{proof}
It is trivial that the condition is necessary. To prove that the condition is sufficient, it is enough to show that $\clF^{D_R(z)}_{\bbC}$ is an isomorphism for $R\in \bbR_{>0}$ and $z\in \bbC$. Since $\{D_r(z)\}_{r\in \bbR_{>0}}$ is a Weiss cover of $\bbC$, we have a coequalizer diagram
\[
\begin{tikzcd}
\displaystyle\bigoplus_{r, s\in \bbR_{>0}}\clF(D_r(z)\cap D_s(z)) \arrow[r, shift left=.75ex,"p"] \arrow[r, shift right=.75ex,swap,"q"] 
& \displaystyle\bigoplus_{r\in \bbR_{>0}}\clF(D_r(z)) \arrow[r, "\pi"]
& \clF(\bbC),
\end{tikzcd}
\]
where the morphism $p, q, \pi$ are defined as in \eqref{eq:coeq}. For $r\in \bbR_{>0}$, define a morphism $f_r\colon \clF(D_r(z))\to \clF(D_R(z))$ by
\[
f_r\ceq 
\begin{cases}
\clF^{D_r(z)}_{D_R(z)} & (r\le R) \\
(\clF^{D_R(z)}_{D_r(z)})^{-1} & (r>R), 
\end{cases}
\]
then, the family $\{f_r\}_{r\in \bbR_{>0}}$ induces the morphism $f\colon \bigoplus_{r\in \bbR_{>0}}\clF(D_r(z))\to \clF(D_R(z))$. Since one can check $f\circ p=f\circ q$, the universality of the coequalizer $\clF(\bbC)$ yields a unique morphism $\wt{f}\colon \clF(\bbC)\to \clF(D_R(z))$ such that $\wt{f}\circ \pi=f$. We find that $\wt{f}$ is the inverse of $\clF^{D_R(z)}_{\bbC}$. 
\end{proof}

We denote
\begin{equation}\label{eq:basisofC}
\frB\ceq
\{D_{r_1}(z_1)\sqcup\cdots \sqcup D_{r_l}(z_l)\mid 
l\in \bbN, r_i\in \bbR_{>0}, z_i\in \bbC,\, D_{r_i}(z_i)\cap D_{r_j}(z_j)=\varnothing\ (i\neq j)\},
\end{equation}
then $\frB$ is an open basis of $\bbC$ such that for each open subset $U\subset \bbC$, the family $\frB(U)$ is a Weiss cover of $U$. 

\begin{lem}\label{lem:fLindep}
Let $\clF\colon \frU_{\bbC}\to\sfM$ be a locally constant prefactorization algebra, and fix $l\in \bbZ_{>0}$. For a domain $U\subset \bbC$ and 
\[
L=D_{r_1}(z_1)\sqcup \cdots \sqcup D_{r_l}(z_l), \quad
M=D_{s_1}(w_1)\sqcup \cdots \sqcup D_{s_l}(w_l)\ \in \frB(U), 
\]
we have
\[
\clF^{\{D_{r_i}(z_i)\}_{i=1}^l}_U\circ (\otimes_{i=1}^l(\clF^{D_{r_i}(z_i)}_{\bbC})^{-1})=
\clF^{\{D_{s_i}(w_i)\}_{i=1}^l}_U\circ (\otimes_{i=1}^l(\clF^{D_{s_i}(w_i)}_{\bbC})^{-1}).
\]
%\[
%\begin{tikzcd}[row sep=huge, column sep=huge]
%\clF(\bbC)^{\otimes l} \arrow[r, "\otimes_{i=1}^l(\clF^{D_{r_i}(z_i)}_{\bbC})^{-1}"] \arrow[d, equal] & [1.3cm]
%\bigotimes_{i=1}^l\clF(D_{r_i}(z_i)) \arrow[r, "\clF^{\{D_{r_i}(z_i)\}_{i=1}^l}_U"] & [0.8cm]
%\clF(U) \arrow[d, equal] \\
%\clF(\bbC)^{\otimes l} \arrow[r, "\otimes_{i=1}^l(\clF^{D_{s_i}(w_i)}_{\bbC})^{-1}"'] &
%\bigotimes_{i=1}^l\clF(D_{s_i}(w_i)) \arrow[r, "\clF^{\{D_{s_i}(w_i)\}_{i=1}^l}_U"'] &
%\clF(U)
%\end{tikzcd}
%\]
\end{lem}

\begin{proof}
For $L=D_{r_1}(z_1)\sqcup \cdots \sqcup D_{r_l}(z_l)\in \frB(U)$, denote by $f_L$ the composition of morphisms 
\[
\begin{tikzcd}[row sep=huge, column sep=huge]
\clF(\bbC)^{\otimes l} \arrow[r, "\otimes_{i=1}^l(\clF^{D_{r_i}(z_i)}_{\bbC})^{-1}"] & [1.3cm]
\bigotimes_{i=1}^l\clF(D_{r_i}(z_i)) \arrow[r, "\clF^{\{D_{r_i}(z_i)\}_{i=1}^l}_U"] & [0.8cm]
\clF(U). 
\end{tikzcd}
\]
Fix $L=D_{r_1}(z_1)\sqcup \cdots \sqcup D_{r_l}(z_l)\in\frB(U)$, and set
\[
E\ceq \{(w_1, \ldots, w_l)\in \Conf_l(U)\mid \forall s_i\in \bbR_{>0},\, M=D_{s_1}(w_1)\sqcup \cdots \sqcup D_{s_l}(w_l)\in \frB(U) \Longrightarrow f_L=f_M\}. 
\]
Then, for $(w_1, \ldots, w_l)\in \Conf_l(U)$, we have
\[
(w_1, \ldots, w_l)\in E\Longleftrightarrow
\exists s_i\in \bbR_{>0},\, M=D_{s_1}(w_1)\sqcup \cdots \sqcup D_{s_l}(w_l)\in \frB(U) \land 
f_L=f_M. 
\]
Hence we find that $(z_1, \ldots, z_l)\in E$, and that both $E$ and $\Conf_l(U)\bs E$ are open subsets of $\bbC^l$. Thus $\Conf_l(U)=E$ since $\Conf_l(U)\subset \bbC^l$ is a connected open subset. Now the lemma is clear. 
\end{proof}

\begin{lem}\label{lem:locFAdomain}
Let $\clF\colon \frU_{\bbC}\to \sfM$ be a locally constant prefactorization algebra. For $R\in \bbR_{>0}$, $z\in \bbC$ and a domain $U\subset \bbC$ such that $D_R(z)\subset U$: 
\begin{enumerate}
\item 
The morphism $\clF^{D_R(z)}_U\colon \clF(D_R(z))\to \clF(U)$ is a split monomorphism. 

\item 
If $\clF$ is a factorization algebra, then $\clF^{D_R(z)}_U\colon \clF(D_R(z))\to \clF(U)$ is an isomorphism. 
\end{enumerate}
\end{lem}

\begin{proof}
(1) This is clear since $\clF^U_{\bbC}\circ \clF^{D_R(z)}_U=\clF^{D_R(z)}_{\bbC}$ is an isomorphism. 

(2) It is enough to show that $\clF^{D_R(z)}_U$ is a epimorphism. Take $A\in \Ob\sfM$ and morphisms $f, g\colon \clF(U)\to A$ such that $f\circ \clF^{D_R(z)}_U=g\circ \clF^{D_R(z)}_U$. To prove $f=g$, by \cref{prp:FUindlim}, it suffices to show that $f\circ \clF^L_U=g\circ \clF^L_U$ for all $L=D_{r_1}(z_1)\sqcup \cdots \sqcup D_{r_l}(z_l)\in \frB(U)$. Take $r'\in \bbR_{>0}$ and $z_1', \ldots, z_l'\in \bbC$ such that
\[
L'\ceq D_{r'}(z_1')\sqcup \cdots \sqcup D_{r'}(z_l')\subset D_R(z),  
\]
then we get
\[
f\circ \clF^{L'}_U=
f\circ \clF^{D_R(z)}_U\circ \clF^{L'}_{D_R(z)}=
g\circ \clF^{D_R(z)}_U\circ \clF^{L'}_{D_R(z)}=
g\circ \clF^{L'}_U. 
\]
Also, by \cref{lem:fLindep}, there exists an isomorphism $\clF(L)\lsto \clF(L')$ which commutes
\[
\begin{tikzcd}[row sep=huge, column sep=huge]
\clF(L) \arrow[r, "\clF^L_U"] \arrow[d, "\sim"' sloped] & \clF(U) \\
\clF(L') \arrow[ru, "\clF^{L'}_U"']
\end{tikzcd}
\]
Thus, we have $f\circ \clF^L_U=g\circ \clF^L_U$. 
\end{proof}

\section{Commutative algebras and locally constant factorization algebras}
\label{s:LCFAcomalg}

In the remaining of this note we will work over the field $\bbC$ of complex numbers, so that all linear spaces and linear maps are defined over $\bbC$, unless otherwise stated. 

Although it is known that the equivalence of $(\infty, 1)$-categories between locally constant factorization algebras on $\bbR^n$ and $E_n$-algebras, we explicitly prove this equivalence for the case $n=2$ with the target category $\Lin$. We use the following notations: 
\begin{itemize}
\item 
We denote by $\LPFA\ceq \PFA^{\mathsf{loc}}(\bbC, \Lin)$ the category of locally constant prefactorization algebras with values in $\Lin$. 

\item 
We denote by $\LFA\ceq \FA^{\mathsf{loc}}(\bbC, \Lin)$ the full subcategory of $\LPFA$ whose objects are factorization algebras. 
\end{itemize}

\subsection{Commutative algebras from locally constant prefactorization algebras}
\label{ss:constcomalg}

We construct a commutative algebra from a locally constant prefactorization algebra. 
Let $\clF\colon \frU_{\bbC}\to \Lin$ be a locally constant prefactorization algebra. 

\begin{dfn}\label{dfn:commult}
For $l\in\bbZ_{>0}$, define a linear map $\mu_l$ as the composition
\[
\begin{tikzcd}[column sep=huge]
\clF(\bbC)^{\otimes l} \arrow[r, "\otimes_{i=1}^n(\clF^{D_R(z_i)}_{\bbC})^{-1}"] & [1.2cm]
\bigotimes_{i=1}^l\clF(D_R(z_i)) \arrow[r, "\clF^{D_R(z_1), \ldots, D_R(z_l)}_{\bbC}"]& [1.2cm]
\clF(\bbC)
\end{tikzcd}
\]
taking $R\in \bbR_{>0}$ and $z_1, \ldots, z_l\in \bbC$ such that $D_R(z_i)\cap D_R(z_j)=\varnothing$ ($i, j\in[l]$, $i\neq j$). 
Note that by \cref{lem:fLindep}, the definition of $\mu_l$ is independent of the choice of $R\in \bbR_{>0}$ and $z_1, \ldots, z_l\in \bbC$. 
\end{dfn}

\begin{lem}\label{lem:asscomuni}
\ 
\begin{enumerate}
\item 
For $a, b, c\in \clF(\bbC)$, we have $\mu_2(\mu_2(a\otimes b)\otimes c)=\mu_3(a\otimes b\otimes c)=\mu_2(a\otimes \mu_2(b\otimes c))$. 

\item
For $a, b\in \clF(\bbC)$, we have $\mu_2(a\otimes b)=\mu_2(b\otimes a)$. 

\item 
For $a\in \clF(\bbC)$, we have $\mu_2(a\otimes \clF^{\varnothing}_{\bbC}(1_\clF))=a$, where $\bbC\to \clF(\varnothing)$, $1\mapsto 1_{\clF}$ denotes the unit. 
\end{enumerate}
\end{lem}

\begin{proof}
(1) We prove $\mu_2(\mu_2(a\otimes b)\otimes c)=\mu_3(a\otimes b\otimes c)$. Take $0<r<R$ and $z_1, z_2, z, w\in \bbC$ such that 
\[
D_r(z_1)\sqcup D_r(z_2)\subset D_R(z), \quad
D_R(z)\cap D_R(w)=\varnothing, 
\]
then 
\begin{align*}
&\mu_2(\mu_2(a\otimes b)\otimes c)\\
&=
\clF^{D_R(z), D_R(w)}_{\bbC}
\Bigl(
(\clF^{D_R(z)}_{\bbC})^{-1}\clF^{D_r(z_1), D_r(z_2)}_{\bbC}
	\bigl(
	(\clF^{D_r(z_1)}_{\bbC})^{-1}(a)\otimes (\clF^{D_r(z_2)}_{\bbC})^{-1}(b)
	\bigr)
\otimes (\clF^{D_R(w)}_{\bbC})^{-1}(c)
\Bigr)\\
&=
\clF^{D_R(z), D_R(w)}_{\bbC}
\Bigl(
\clF^{D_r(z_1), D_r(z_2)}_{D_R(z)}
	\bigl(
	(\clF^{D_r(z_1)}_{\bbC})^{-1}(a)\otimes (\clF^{D_r(z_2)}_{\bbC})^{-1}(b)
	\bigr)
\otimes (\clF^{D_R(w)}_{\bbC})^{-1}(c)
\Bigr)\\
&=
\clF^{D_r(z_1), D_r(z_2), D_R(w)}_{\bbC}
\bigl(
(\clF^{D_r(z_1)}_{\bbC})^{-1}(a)\otimes (\clF^{D_r(z_2)}_{\bbC})^{-1}(b)\otimes (\clF^{D_R(w)}_{\bbC})^{-1}(c)
\bigr)\\
&=
\mu_3(a\otimes b\otimes c),  
\end{align*}
where we used the first part of \cref{dfn:PFA} (i) in the second equality. 
Similarly, one can prove $\mu_2(a\otimes \mu_2(b\otimes c))=\mu_3(a\otimes b\otimes c)$. 

(2) This follows from \cref{dfn:PFA} (ii). 

(3) Take $R\in \bbR_{>0}$ and $z_1, z_2\in \bbC$ such that $D_R(z_1)\cap D_R(z_2)=\varnothing$, then
\begin{align*}
\mu_2(a\otimes \clF^{\varnothing}_{\bbC}(1_{\clF}))
&=
\clF^{D_R(z_1), D_R(z_2)}_{\bbC}
\bigl(
(\clF^{D_R(z_1)}_{\bbC})^{-1}(a)\otimes (\clF^{D_R(z_2)}_{\bbC})^{-1}\clF^{\varnothing}_{\bbC}(1_\clF)
\bigr)\\
&=
\clF^{D_R(z_1), D_R(z_2)}_{\bbC}
\bigl(
(\clF^{D_R(z_1)}_{\bbC})^{-1}(a)\otimes \clF^{\varnothing}_{D_R(z_2)}(1_{\clF})
\bigr)\\
&=
\clF^{D_R(z_1), \varnothing}_{\bbC}
\bigl(
(\clF^{D_R(z_1)}_{\bbC})^{-1}(a)\otimes 1_{\clF}
\bigr)\\
&=a, 
\end{align*}
where we used the second part of \cref{dfn:PFA} (i) in the third equality and \cref{dfn:PFA} (iv) in the forth equality. 
\end{proof}

\begin{rmk}\label{rmk:ntermop}
Similar to the proof of \cref{lem:asscomuni} (1), one can prove 
\[
\mu_2(\mu_n(a_1\otimes \cdots \otimes a_n)\otimes a_{n+1})=
\mu_{n+1}(a_1\otimes \cdots \otimes a_{n+1})
\]
for $n\in \bbZ_{>0}$ and $a_1, \ldots, a_{n+1}\in \clF(\bbC)$. As a result, we have
\[
a_1\cdots a_n=\mu_n(a_1\otimes \cdots \otimes a_n),
\]
where we denote
\[
a_1\cdots a_n\ceq (a_1\cdots a_{n-1})a_n, \quad a_1a_2\ceq \mu_2(a_1\otimes a_2). 
\]
\end{rmk}

\begin{prp}
For a locally constant prefactorization algebra $\clF\colon \frU_{\bbC}\to \Lin$, the linear space $\clF(\bbC)$ has the structure of commutative algebra as follows: 
\begin{itemize}
\item 
The multiplication is $\mu_2\colon \clF(\bbC)\otimes \clF(\bbC)\to \clF(\bbC)$ defined in \cref{dfn:commult}. 

\item 
The unit is $\clF^{\varnothing}_{\bbC}(1_\clF)$, where $\bbC\to \clF(\varnothing)$, $1\mapsto 1_\clF$ denotes the unit. 
\end{itemize}
\end{prp}

\begin{proof}
This follows from \cref{lem:asscomuni}. 
\end{proof}

\begin{rmk}\label{rmk:clFDRcalg}
Let $\clF\colon \frU_{\bbC}\to \Lin$ be a locally constant prefactorization algebra. 
For $R\in \bbR_{>0}$ and $z\in \bbC$, since $\clF^{D_R(z)}_{\bbC}\colon \clF(D_R(z))\to \clF(\bbC)$ is a linear isomorphism, the linear space $\clF(D_R(z))$ inherits the structure of a commutative algebra from $\clF(\bbC)$. The multiplication on $\clF(D_R(z))$ is equal to
\[
a\cdot b=
\clF^{D_r(z_1), D_r(z_2)}_{D_R(z)}\bigl((\clF^{D_r(z_1)}_{D_R(z)})^{-1}(a)\otimes (\clF^{D_r(z_2)}_{D_R(z)})^{-1}(b)\bigr) 
\]
for any $r\in \bbR_{>0}$ and $z_1, z_2\in \bbC$ such that $D_r(z_1)\sqcup D_r(z_2)\subset D_R(z)$. Also, the unit of $\clF(D_R(z))$ is $\clF^{\varnothing}_{D_R(z)}(1_\clF)$. 
\end{rmk}

\begin{prp}
For locally constant prefactorization algebras $\clF, \clG\colon \frU_{\bbC}\to \Lin$ and a morphism $\varphi\colon \clF\to \clG$ of prefactorization algebras, the linear map $\varphi_{\bbC}\colon \clF(\bbC)\to \clG(\bbC)$ is a morphism of commutative algebras. Thus, the map $\clF\mapsto \clF(\bbC)$ gives rise to the functor
\[
(-)(\bbC)\colon \LPFA\to \CAlg_{\bbC}
\]
\end{prp}

\begin{proof}
We proceed the following steps: 
\begin{itemize}
\item
The linear map $\varphi_{\bbC}$ preserves the multiplication: For $a, b\in \clF(\bbC)$, take $R\in \bbR_{>0}$ and $z_1, z_2\in \bbC$ such that $D_R(z_1)\cap D_R(z_2)=\varnothing$, then
\begin{align*}
\varphi_{\bbC}(ab)
&=
\varphi_{\bbC}\clF^{D_R(z_1), D_R(z_2)}_{\bbC}
\bigl(
(\clF^{D_R(z_1)}_{\bbC})^{-1}(a)\otimes (\clF^{D_R(z_2)}_{\bbC})^{-1}(b)
\bigr)\\
&=
\clG^{D_R(z_1), D_R(z_2)}_{\bbC}
\bigl(
\varphi_{D_R(z_1)}(\clF^{D_R(z_1)}_{\bbC})^{-1}(a)\otimes \varphi_{D_R(z_2)}(\clF^{D_R(z_2)}_{\bbC})^{-1}(b)
\bigr)\\
&=
\clG^{D_R(z_1), D_R(z_2)}_{\bbC}
\bigl(
(\clG^{D_R(z_1)}_{\bbC})^{-1}\varphi_{\bbC}(a)\otimes (\clG^{D_R(z_2)}_{\bbC})^{-1}\varphi_{\bbC}(b)
\bigr)\\
&=
\varphi_{\bbC}(a)\varphi_{\bbC}(b). 
\end{align*}

\item
The linear map $\varphi_{\bbC}$ preserves the unit: We have
\[
\varphi_{\bbC}\clF^{\varnothing}_{\bbC}(1_{\clF})=
\clG^{\varnothing}_{\bbC}\varphi_{\varnothing}(1_{\clF})=
\clG^{\varnothing}_{\bbC}(1_{\clG}). 
\qedhere
\]
\end{itemize}
\end{proof}

\subsection{Locally constant factorization algebras from commutative algebras}
\label{ss:constLCFA}

Next, we construct a locally constant factorization algebra from a commutative algebra. Let $A$ be a commutative algebra. We use the notation $\frB$ defined in \eqref{eq:basisofC}. 
Note that $\frB$ is an open basis of $\bbC$ satisfying: 
\begin{clist}
\item 
$\varnothing\in \frB$. 

\item 
For $L, M\in \frB$ such that $L\cap M=\varnothing$, we have $L\sqcup M\in \frB$. 
\end{clist}

\begin{dfn}
Define a map $F^{\loc}_A\colon \frB\to \Ob(\Lin)$ by
\[
F^{\loc}_A(D_{r_1}(z_1)\sqcup \cdots \sqcup D_{r_l}(z_l))\ceq
\bigotimes_{i=1}^lF^{\loc}_A(D_{r_i}(z_i)), \quad 
F^{\loc}_A(D_{r_i}(z_i))\ceq A
\]
with the convention $F^{\loc}_A(\varnothing)\ceq\bbC$ for $l=0$. Also, for 
\[
L=D_{r_1}(z_1)\sqcup \cdots \sqcup D_{r_l}(z_l),\ 
M=D_{R_1}(w_1)\sqcup \cdots \sqcup D_{R_m}(w_m)\in \frB
\]
such that $L\subset M$, define a linear map $(F^{\loc}_A)^L_M\colon F^{\loc}_A(L)\to F^{\loc}_A(M)$ as follows: 
\begin{enumerate}
\item 
If $l=m=0$, then define $(F^{\loc}_A)^L_M(1_{\bbC})\ceq 1_{\bbC}$. 

\item
If $l=0$, $m>0$, then define $(F^{\loc}_A)^L_M(1_{\bbC})\ceq \otimes_{j=1}^m1_A$. 

\item
If $l, m>0$, then define
\[
(F^{\loc}_A)^L_M(\otimes_{i=1}^la_i)\ceq \otimes_{j=1}^m
\prod_{i\in I_j}a_i, 
\]
where $I_1\sqcup \cdots \sqcup I_m=[l]$ is a unique decomposition such that $\bigsqcup_{i\in I_j}D_{r_i}(z_i)\subset D_{R_j}(w_j)$. 
\end{enumerate}
\end{dfn}

\begin{lem}
The map $F^{\loc}_A\colon \frB\to \Ob(\Lin)$ is a precosheaf by $(F^{\loc}_A)^L_M$. 
\end{lem}

\begin{proof}
It is clear that $(F^{\loc}_A)^L_L=\id_{F^{\loc}_A(L)}$ for all $L\in \frB$. Let $L, M, N\in \frB$ such that $L\subset M\subset N$, and write 
\[
L=D_{r_1}(z_1)\sqcup \cdots \sqcup D_{r_l}(z_l), \quad
M=D_{s_1}(w_1)\sqcup \cdots \sqcup D_{s_m}(w_m), \quad
N=D_{R_1}(\zeta_1)\sqcup \cdots \sqcup D_{R_n}(\zeta_n). 
\]
Take $a=a_1\otimes \cdots \otimes a_l\in F^{\loc}_A(L)$. Then, 
\[
(F^{\loc}_A)^M_N(F^{\loc}_A)^L_M(a)=
\otimes_{k=1}^n\prod_{j\in J_k}\prod_{i\in I_j}a_i, 
\]
where $I_1\sqcup \cdots \sqcup I_m=[l]$ and $J_1\sqcup \cdots \sqcup J_n=[m]$ are the decompositions satisfying
\[
\bigsqcup_{i\in I_j}D_{r_i}(z_i)\subset D_{s_j}(w_j), \quad
\bigsqcup_{j\in J_k}D_{s_j}(w_j)\subset D_{R_k}(\zeta_k). 
\]
On the other hand, 
\[
(F^{\loc}_A)^L_N(a)=\otimes_{k=1}^n\prod_{i\in I_k'}a_i,
\]
where $I_1'\sqcup \cdots \sqcup I_n'=[l]$ is the decomposition satisfying 
\[
\bigsqcup_{i\in I_k'}D_{r_i}(z_i)\subset D_{R_k}(\zeta_k). 
\]
Since $\bigsqcup_{j\in J_k}I_j=I_k'$ for each $k\in[n]$, we have $(F^{\loc}_A)^M_N(F^{\loc}_A)^L_M=(F^{\loc}_A)^L_N$. 
\end{proof}

\begin{lem}\label{lem:sqcuptensor}
For $L_1, L_2, M_1, M_2\in \frB$ such that $L_1\subset M_1$, $L_2\subset M_2$ and $M_1\cap M_2=\varnothing$, we have
\[
(F^{\loc}_A)^{L_1\sqcup L_2}_{M_1\sqcup M_2}=
(F^{\loc}_A)^{L_1}_{M_1}\otimes (F^{\loc}_A)^{L_2}_{M_2}. 
\]
\end{lem}

\begin{proof}
Write
\begin{align*}
&L_1=D_{r_{1, 1}}(z_{1, 1})\sqcup \cdots \sqcup D_{r_{1, l_1}}(z_{1, l_1}), \hspace{31pt}
L_2=D_{r_{2, 1}}(z_{2, 1})\sqcup \cdots \sqcup D_{r_{2, l_2}}(z_{2, l_2}), \\
&M_1=D_{R_{1, 1}}(w_{1, 1})\sqcup \cdots \sqcup D_{R_{1, m_1}}(w_{1, m_1}), \quad
M_2=D_{R_{2, 1}}(w_{2, 1})\sqcup \cdots \sqcup D_{R_{2, m_2}}(w_{2, m_2}), 
\end{align*}
and denote by 
\[
I_{1, 1}\sqcup \cdots \sqcup I_{1, m_1}=[l_1], \quad
I_{2, 1}\sqcup \cdots \sqcup I_{2, m_2}=[l_2]
\]
the decompositions satisfying 
\[
\bigsqcup_{i\in I_{1, j}}D_{r_{1, i}}(z_{1, i})\subset D_{R_{1, j}}(w_{1, j}), \quad
\bigsqcup_{i\in I_{2, j}}D_{r_{2, i}}(z_{2, i})\subset D_{R_{2, j}}(w_{2, j}). 
\]
Take $a^1=a_{1, 1}\otimes \cdots \otimes a_{1, l_1}\in F^{\loc}_A(L_1)$ and $a^2=a_{2, 1}\otimes \cdots \otimes a_{2, l_2}\in F^{\loc}_A(L_2)$. Then, 
\[
(F^{\loc}_A)^{L_1}_{M_1}(a^1)\otimes (F^{\loc}_A)^{L_2}_{M_2}(a^2)=
\bigl(\prod_{i\in I_{1, 1}}a_{1, i}\bigr)\otimes \cdots \otimes \bigl(\prod_{i\in I_{1, m_1}}a_{1, i}\bigr)\otimes 
\bigl(\prod_{i\in I_{2, 1}}a_{2, i}\bigr)\otimes \cdots \otimes \bigl(\prod_{i\in I_{2, m_2}}a_{2, i}\bigr). 
\]
Now, set for the symbol $x\in\{a, r, z\}$ and $y\in\{R, w\}$, 
\[
x_i\ceq 
\begin{cases}
x_{1, i} & (1\le i\le l_1) \\
x_{2, i-l_1} & (l_1<i\le l_1+l_2), 
\end{cases}
\quad
y_j\ceq 
\begin{cases}
y_{1, j} & (1\le j\le m_1)\\
y_{2, j-m_1} & (m_1<j\le m_1+m_2), 
\end{cases}
\]
and denote
\[
I_j\ceq 
\begin{cases}
I_{1, j} & (1\le j\le m_1) \\
l_1+I_{2, j-m_1} & (m_1<j\le m_1+m_2). 
\end{cases}
\]
Then, we find that $I_1\sqcup \cdots \sqcup I_{m_1+m_2}=[l_1+l_2]$ is the decomposition satisfying 
\[
\bigsqcup_{i\in I_j}D_{r_i}(z_i)\subset D_{R_j}(w_j). 
\]
Thus, we have
\[
(F^{\loc}_A)^{L_1\sqcup L_2}_{M_1\sqcup M_2}(a^1\otimes a^2)=
\bigl(\prod_{i\in I_1}a_i\bigr)\otimes\cdots \otimes \bigl(\prod_{i\in I_{m_1+m_2}}a_i\bigr)=
(F^{\loc}_A)^{L_1}_{M_1}(a^1)\otimes (F^{\loc}_A)^{L_2}_{M_2}(a^2). 
\qedhere
\]
\end{proof}

\begin{lem}\label{lem:FlocPFA}
The precosheaf $F^{\mathrm{loc}}_A\colon \frB\to \Lin$ is a prefactorization algebra by the multiplication 
\[
(F^{\loc}_A)^{L, M}_N\ceq (F^{\loc}_A)^{L\sqcup M}_N\colon 
F^{\loc}_A(L)\otimes F^{\loc}_A(M)\to F^{\loc}_A(N) \quad
(L, M, N\in \frB, L\sqcup M\subset N)
\]
and the unit $\id_{\bbC}\colon \bbC\to F^{\loc}_A(\varnothing)$. 
\end{lem}

\begin{proof}
We check the conditions (i)-(iv) in \cref{dfn:PFA}. 
\begin{clist}
\item 
The first part is clear by the definition of $(F^{\loc}_A)^{L, M}_N$. Also, the second part follows from \cref{lem:sqcuptensor}. 

\item 
Let $L_1, L_2, M\in \frB$ such that $L_1\sqcup L_2\subset M$, and write
\[
L_1=D_{r_{1, 1}}(z_{1, 1})\sqcup \cdots \sqcup D_{r_{1, l_1}}(z_{1, l_1}), \quad
L_2=D_{r_{2, 1}}(z_{2, 1})\sqcup \cdots \sqcup D_{r_{2, l_2}}(z_{2, l_2}).
\]
Take $a^1=a_{1, 1}\otimes \cdots \otimes a_{1, l_1}\in F^{\loc}_A(L_1)$ and $a^2=a_{2, 1}\otimes \cdots \otimes a_{2, l_2}\in F^{\loc}_A(L_2)$. For the symbol $x\in \{a, r, z\}$, set
\[
x_i\ceq
\begin{cases}
x_{1, i} & (1\le i\le l_1) \\
x_{2, i-l_1} & (l_1<i\le l_1+l_2),
\end{cases}
\]
then
\[
(F_A^{\loc})^{L_2, L_1}_M(a^2\otimes a^1)=
(F_A^{\loc})^{L_1\sqcup L_2}_M(a_1\otimes \cdots \otimes a_{l_1+l_2})=
(F_A^{\loc})^{L_1, L_2}_M(a^1\otimes a^2). 
\]

\item
This follows from \cref{lem:sqcuptensor} and the commutative diagram
\[
\begin{tikzcd}[row sep=huge]
F_A^{\loc}(L_1\sqcup L_2\sqcup L_3) \arrow[rr, equal] \arrow[d, "(F_A^{\loc})^{L_1\sqcup L_2\sqcup L_3}_{M_1\sqcup L_3}"']&
&
F_A^{\loc}(L_1\sqcup L_2\sqcup L_3)\arrow[d, "(F_A^{\loc})^{L_1\sqcup L_2\sqcup L_3}_{L_1\sqcup M_2}"] \\
F_A^{\loc}(M_1\sqcup L_3) \arrow[rd, "(F_A^{\loc})^{M_1\sqcup L_3}_N"']&
&
F_A^{\loc}(L_1\sqcup M_2) \arrow[ld, "(F_A^{\loc})^{L_1\sqcup M_2}_N"]\\
&
F_A^{\loc}(N)&
\end{tikzcd}
\]

\item 
This follows from the definition of $(F_A^{\loc})^{L, M}_N$. 
\qedhere
\end{clist}
\end{proof}

By \cref{lem:FlocPFA} and the construction in \cref{ss:exten}, we have a prefactorization algebra
\[
\bfF^{\loc}_A\ceq \clE_{F_A^{\loc}}\colon \frU_{\bbC}\to \Lin. 
\]
Recall that for an open subset $U\subset \bbC$,  
\[
\bfF_A^{\loc}(U)=
\varinjlim_{L\in \frB(U)}F^{\loc}_A(L),
\]
and that the canonical morphism is denoted by
\[
(\wt{F}_A^{\loc})^L_U\colon F_A^{\loc}(L)\to \bfF_A^{\loc}(U) \quad
(L\in \frB(U)). 
\]

\begin{prp}
For a commutative algebra $A$, the prefactorization algebra $\bfF_A^{\loc}\colon \frU_{\bbC}\to \Lin$ is a locally constant factorization algebra. 
\end{prp}

\begin{proof}
First, we prove that $\bfF_A^{\loc}$ is a factorization algebra. By \cref{prp:extmultisom}, we only need to check the condition (i) in \cref{dfn:FA}. Let $U\subset \bbC$ be an open subset, and $\{U_i\}_{i\in I}$ a Weiss cover of $U$. We use the notations $p, q, \pi$ defined in \eqref{eq:coeq} for $\bfF_A^{\loc}$. Fix a linear space $E$ and a linear map $f\colon \bigoplus_{i\in I}\bfF_A^{\loc}(U_i)\to E$ such that $f\circ p=f\circ q$. For $L=D_{r_1}(z_1)\sqcup \cdots \sqcup D_{r_l}(z_l)\in \frB(U)$, define a linear map $f_L\colon F_A^{\loc}(L)\to E$ as the composition 
\[
\begin{tikzcd}
F^{\loc}_A(L) \arrow[r, equal] &
F^{\loc}_A(D_r(z_1)\sqcup \cdots \sqcup D_r(z_l)) \arrow[r, "(\wt{F}_A^{\loc})^{D_r(z_1)\sqcup \cdots \sqcup D_r(z_l)}_{U_i}"] & [2.2cm]
\bfF_A^{\loc}(U_i) \arrow[r, hookrightarrow] &
\bigoplus_{i\in I}\bfF_A^{\loc}(U_i) \arrow[r, "f"] &
E,
\end{tikzcd}
\]
where we take $i\in I$ and $r\in \bbR_{>0}$ satisfying
\[
z_1, \ldots, z_l\in U_i, \quad
D_r(z_1)\sqcup \cdots \sqcup D_r(z_l)\subset U_i. 
\]
Notice that the condition $f\circ p=f\circ q$ and the commutative diagram
\[
\begin{tikzcd}[row sep=1cm, column sep=huge]
\bfF_A^{\loc}(U_i\cap U_j) \arrow[r, "(\bfF^{\loc}_A)^{U_i\cap U_j}_{U_i}"] \arrow[d, hookrightarrow]& 
\bfF_A^{\loc}(U_i) \arrow[d, hookrightarrow]&
\\
\bigoplus_{i, j\in I}\bfF_A^{\loc}(U_i\cap U_j) \arrow[r, shift left=.75ex,"p"] \arrow[r, shift right=.75ex,swap,"q"] 
& \bigoplus_{i\in I}\bfF_A^{\loc}(U_i) \arrow[r, "f"]
& E\\
\bfF_A^{\loc}(U_i\cap U_j) \arrow[u, hookrightarrow] \arrow[r, "(\bfF_A^{\loc})^{U_i\cap U_j}_{U_j}"'] & 
\bfF_A^{\loc}(U_j) \arrow[u, hookrightarrow]&
\end{tikzcd}
\]
guarantee that the definition of $f_L$ is independent of the choice of $i\in I$ and $r\in \bbR_{>0}$. Since $f_M\circ (F_A^{\loc})^L_M=f_L$ holds for $L, M\in \frB(U)$ such that $L\subset M$, the universality of $\bfF_A^{\loc}(U)=\varinjlim_{L\in \frB(U)}F_A^{\loc}(L)$ yields a unique linear map $\wt{f}\colon \bfF_A^{\loc}(U)\to E$ satisfying $\wt{f}\circ (\wt{F}_A^{\loc})^L_U=f_L$ for all $L\in \frB(U)$. We find that $\wt{f}$ is a unique linear map satisfying $\wt{f}\circ \pi=f$. Thus $(\bfF_A^{\loc}(U), \pi)$ is a coequalizer of $(p, q)$. 

It is now clear that $\bfF_A^{\loc}$ is locally constant. In fact, by \cref{lem:locFA}, it is enough to show that the linear map $(\bfF_A^{\loc})^{D_r(z)}_{D_R(w)}$ is an isomorphism for $0<r<R<\infty$ and $z, w\in \bbC$ such that $D_r(z)\subset D_R(w)$, but this is obvious since $(F_A^{\loc})^{D_r(z)}_{D_R(w)}$ is an isomorphism. 
\end{proof}

For a morphism $f\colon A\to B$ of commutative algebras and $L=D_{r_1}(z_1)\sqcup \cdots \sqcup D_{r_l}(z_l)\in \frB$, define
\[
F^{\loc}_{f, L}\colon 
F_A^{\loc}(L)\to F_B^{\loc}(L), \quad
a_1\otimes \cdots \otimes a_l\mapsto f(a_1)\otimes \cdots \otimes f(a_l)
\]
with the convention 
\[
F^{\loc}_{f, \varnothing}\ceq \id_{\bbC}\colon 
F_A^{\loc}(\varnothing)\to F_B^{\loc}(\varnothing). 
\]

\begin{lem}\label{lemfunVAtoPFA}
For a morphism $f\colon A\to B$ of commutative algebras, 
\[
F_f^{\loc}\ceq \{F_{f, L}^{\loc}\}_{L\in \frB}\colon F_A^{\loc}\to F_B^{\loc}
\]
is a morphism of prefactorization algebras. 
\end{lem}

\begin{proof}
We proceed the following steps: 
\begin{itemize}
\item
$F_f^{\loc}\colon F_A^{\loc}\to F_B^{\loc}$ is a morphism of precosheaves: 
Let $L, M\in \frB$ such that $L\subset M$. Write
\[
L=D_{r_1}(z_1)\sqcup \cdots \sqcup D_{r_l}(z_l), \quad
M=D_{R_1}(w_1)\sqcup \cdots \sqcup D_{R_m}(w_m), 
\]
and denote by $I_1\sqcup \cdots \sqcup I_m=[l]$ the decomposition satisfying 
\[
\bigsqcup_{i\in I_j}D_{r_i}(z_i)\subset D_{R_j}(w_j). 
\]
Then, for $a=a_1\otimes \cdots \otimes a_l\in F_A^{\loc}(L)$, 
\[
F_{f, M}^{\loc}(F_A^{\loc})^L_M(a)=
\otimes_{j=1}^mf\bigl(\prod_{i\in I_j}a_i\bigr)=
\otimes_{j=1}^m\prod_{i\in I_j}f(a_i)=
(F_B^{\loc})^L_MF_{f, L}^{\loc}(a). 
\]

\item
$F_f^{\loc}\colon F_A^{\loc}\to F_B^{\loc}$ is a morphism of prefactorization algebras: The condition (i) in \cref{dfn:morPFA} follows from $F_{f, L}^{\loc}\otimes F_{f, M}^{\loc}=F_{f, L\sqcup M}^{\loc}$ ($L, M\in\frB$, $L\cap M=\varnothing$). Also, the condition (ii) in \cref{dfn:morPFA} is clear by the definition $F_{f, \varnothing}^{\loc}=\id_{\bbC}$. 
\qedhere
\end{itemize}
\end{proof}

For a morphism $f\colon A\to B$ of commutative algebras, by \cref{lemfunVAtoPFA} and \cref{ss:exten}, we have a morphism 
\[
\bfF^{\loc}_f\ceq
\clE_{F^{\loc}_f}\colon \bfF_A^{\loc}\to \bfF_B^{\loc}
\]
of factorization algebras. Thus, the map $A\mapsto \bfF_A^{\loc}$ gives rise to the functor
\[
\bfF^{\loc}\colon \CAlg_{\bbC}\to \LFA. 
\]

\subsection{Adjointness and categorical equivalence}
\label{ss:adjcomalgLCPFA}

Finally, we prove that $\bfF^{\loc}\colon \CAlg_{\bbC}\to \LPFA$ is a fully faithful functor which is the left adjoint of $(-)(\bbC)\colon \LPFA\to \CAlg_{\bbC}$, and that the essential image of $\bfF^{\loc}$ consists of locally constant factorization algebras. 

\begin{dfn}\label{dfn:adjtheta}
Let $A\in \Ob(\CAlg_{\bbC})$ and $\clF\in \Ob(\LPFA)$. For a morphism $\varphi\colon \bfF_A^{\loc}\to \clF$ in $\LPFA$, define a linear map $\theta_{A, \clF}(\varphi)$ as the composition
\[
\begin{tikzcd}[column sep=huge]
A \arrow[r] \arrow[r, "(\wt{F}_A^{\loc})^{D_R(z)}_{\bbC}"]&
\bfF_A^{\loc}(\bbC) \arrow[r, "\varphi_{\bbC}"] &
\clF(\bbC).
\end{tikzcd}
\]
Here we take $R\in \bbR_{>0}$ and $z\in \bbC$, but the definition of $\theta_{A, \clF}(\varphi)$ is independent of the choice of $R$ and $z$. 
\end{dfn}

\begin{lem}\label{lem:unitisom}
Let $A\in \Ob(\CAlg_{\bbC})$. For $R\in \bbR_{>0}$ and $z\in \bbC$, the linear map $(\wt{F}^{\loc}_A)^{D_R(z)}_{D_R(z)}\colon A\to \bfF^{\loc}_A(D_R(z))$ is an isomorphism in $\CAlg_{\bbC}$. (See \cref{rmk:clFDRcalg} for the commutative algebra structure on $\bfF^{\loc}_A(D_R(0))$.)
\end{lem}

\begin{proof}
It is enough to show that $(\wt{F}^{\loc}_A)^{D_R(z)}_{D_R(z)}$ is a morphism of commutative algebras. 
\begin{itemize}
\item
$(\wt{F}_A^{\loc})^{D_R(z)}_{D_R(z)}$ preserves the multiplications: For $a, b\in A$, take $r\in \bbR_{>0}$ and $z_1, z_2\in\bbC$ such that $D_r(z_1)\sqcup D_r(z_2)\subset D_R(z)$, then
\begin{align*}
&(\wt{F}^{\loc}_A)^{D_R(z)}_{D_R(z)}(a)\cdot (\wt{F}^{\loc}_A)^{D_R(z)}_{D_R(z)}(b)\\
&=
(\bfF^{\loc}_A)^{D_r(z_1), D_r(z_2)}_{D_R(z)}
\bigl(\bigl((\bfF^{\loc}_A)^{D_r(z_1)}_{D_R(z)}\bigr)^{-1}(\wt{F}^{\loc}_A)^{D_R(z)}_{D_R(z)}(a)\otimes \bigl((\bfF^{\loc}_A)^{D_r(z_2)}_{D_R(z)}\bigr)^{-1}(\wt{F}^{\loc}_A)^{D_R(z)}_{D_R(z)}(b)\bigr)\\
&=
(\bfF^{\loc}_A)^{D_r(z_1), D_r(z_2)}_{D_R(z)}\bigl((\wt{F}^{\loc}_A)^{D_r(z_1)}_{D_r(z_1)}(a)\otimes (\wt{F}^{\loc}_A)^{D_r(z_2)}_{D_r(z_2)}(b)\bigr)\\
&=
(\wt{F}^{\loc}_A)^{D_r(z_1)\sqcup D_r(z_2)}_{D_R(z)}(F^{\loc}_A)^{D_r(z_1), D_r(z_2)}_{D_r(z_1)\sqcup D_r(z_2)}(a\otimes b)=
(\wt{F}^{\loc}_A)^{D_r(z_1)\sqcup D_r(z_2)}_{D_R(z)}(a\otimes b)\\
&=
(\wt{F}^{\loc}_A)^{D_R(z)}_{D_R(z)}(F^{\loc}_A)^{D_r(z_1)\sqcup D_r(z_2)}_{D_R(z)}(a\otimes b)=
(\wt{F}^{\loc}_A)^{D_R(z)}_{D_R(z)}(ab). 
\end{align*}

\item
$(\wt{F}^{\loc}_A)^{D_R(z)}_{D_R(z)}$ preserves the units: 
We have
\[
(\bfF^{\loc}_A)^{\varnothing}_{D_R(z)}(1_{\bfF^{\loc}_A})=
(\bfF^{\loc}_A)^{\varnothing}_{D_R(z)}(\wt{F}^{\loc}_A)^{\varnothing}_{\varnothing}(1_{\bbC})=
(\wt{F}^{\loc}_A)^{D_R(z)}_{D_R(z)}(F^{\loc}_A)^{\varnothing}_{D_R(z)}(1_{\bbC})=
(\wt{F}^{\loc}_A)^{D_R(z)}_{D_R(z)}(1_A). 
\qedhere
\]
\end{itemize}
\end{proof}

\begin{lem}
Let $A\in \Ob(\CAlg_{\bbC})$ and $\clF\in \Ob(\LPFA)$. For a morphism $\varphi\colon \bfF_A^{\loc}\to \clF$ in $\LPFA$, the linear map
\[
\theta_{A, \clF}(\varphi)\colon A\to \clF(\bbC)
\]
is a morphism in $\CAlg_{\bbC}$. 
\end{lem}

\begin{proof}
It is enough to show that $(\wt{F}^{\loc}_A)^{D_R(z)}_{\bbC}\colon A\to \bfF^{\loc}_A(\bbC)$ is a morphism of commutative algebras for $R\in \bbR_{>0}$ and $z\in \bbC$, but since $(\wt{F}^{\loc}_A)^{D_R(z)}_{\bbC}=(\bfF^{\loc}_A)^{D_R(z)}_{\bbC}\circ (\wt{F}^{\loc}_A)^{D_R(z)}_{D_R(z)}$, this follows from \cref{lem:unitisom}. 
\end{proof}

\begin{prp}\label{prp:adj}
The map
\[
\theta_{A, \clF}\colon \Hom_{\LPFA}(\bfF_A^{\loc}, \clF)\to \Hom_{\CAlg_{\bbC}}(A, \clF(\bbC)), \quad
\varphi\mapsto \theta_{A, \clF}(\varphi)
\]
is a bijection which is natural in $A\in \Ob(\CAlg_{\bbC})$ and $\clF\in \Ob(\LPFA)$. Thus, the functor $\bfF^{\loc}\colon \CAlg_{\bbC}\to \LPFA$ is the left adjoint of $(-)(\bbC)\colon \LPFA\to \CAlg_{\bbC}$. 
\end{prp}

We prove \cref{prp:adj}. It is easy to see that $\theta_{A, \clF}$ is natural in $A$ and $\clF$. For $A\in \Ob(\CAlg_{\bbC})$ and $\clF\in \Ob(\LPFA)$, we construct the inverse map of $\theta_{A, \clF}$. 

Let $f\colon A\to \clF(\bbC)$ be a morphism in $\CAlg_{\bbC}$. For an open subset $U\subset \bbC$ and $L=D_{r_1}(z_1)\sqcup \cdots \sqcup D_{r_l}(z_l)\in \frB(U)$, define a linear map $\theta'_{A, \clF}(f)^L_U$ by
\[
\theta'_{A, \clF}(f)^L_U\colon 
F_A^{\loc}(L)\to \clF(U), \quad
\otimes_{i=1}^la_i\mapsto \clF^{\{D_{r_i}(z_i)\}_{i=1}^l}_U\bigl(\otimes_{i=1}^l(\clF^{D_{r_i}(z_i)}_{\bbC})^{-1}f(a_i)\bigr)
\]
with the convention 
\[
\theta'_{A, \clF}(f)^{\varnothing}_{U}\colon F_A^{\loc}(\varnothing)=\bbC\to \clF(U), \quad
1_{\bbC}\mapsto \clF^{\varnothing}_U(1_{\clF})
\]
for $l=0$. 

\begin{lem}\label{lem:thetaLM}
Let $f\colon A\to \clF(\bbC)$ be a morphism in $\CAlg_{\bbC}$. For an open subset $U\subset \bbC$ and $L, M\in \frB(U)$ such that $L\subset M$, we have
\[
\theta'_{A, \clF}(f)^M_U\circ (F_A^{\loc})^L_M=
\theta'_{A, \clF}(f)^L_U. 
\]
\end{lem}

\begin{proof}
Write
\[
L=D_{r_1}(z_1)\sqcup \cdots \sqcup D_{r_l}(z_l), \quad
M=D_{R_1}(w_1)\sqcup \cdots \sqcup D_{R_m}(w_m), 
\]
and denote by $I_1\sqcup \cdots \sqcup I_m=[l]$ the decomposition such that 
\[
\bigsqcup_{i\in I_j}D_{r_i}(z_i)\subset D_{R_j}(w_j). 
\]
Let $a=a_1\otimes \cdots \otimes a_l\in F_A^{\loc}(L)$. Then
\begin{align*}
\theta'_{A, \clF}(f)^L_U(a)
&=
\clF^{\{D_{r_i}(z_i)\}_{i=1}^l}_U
\bigl(
\otimes_{i=1}^l(\clF^{D_{r_i}(z_i)}_{\bbC})^{-1}f(a_i)
\bigr)\\
&=
\clF^{\{D_{R_j}(w_j)\}_{j=1}^m}_U
\bigl(
\otimes_{j=1}^m\clF^{\{D_{r_i}(z_i)\}_{i\in I_j}}_{D_{R_j}(w_j)}
	\bigl(
	\otimes_{i\in I_j}(\clF^{D_{r_i}(z_i)}_{\bbC})^{-1}f(a_i)
	\bigr)
\bigr). 
\end{align*}
Since the diagram
\[
\begin{tikzcd}[row sep=huge, column sep=2.2cm]
\otimes_{i\in I_j}\clF(D_{r_i}(z_i)) \arrow[r, "\clF^{\{D_{r_i}(z_i)\}_{i\in I_j}}_{D_{R_j}(w_j)}"] \arrow[d, "\otimes_{i\in I_j}\clF^{D_{r_i}(z_i)}_{\bbC}"'] \arrow[rd, "\clF^{\{D_{r_i}(z_i)\}_{i\in I_j}}_{\bbC}" description] &
\clF(D_{R_j}(w_j)) \arrow[d, "\clF^{D_{R_j}(w_j)}_{\bbC}"] \\
\clF(\bbC)^{\otimes I_j} \arrow[r, "\mu_{I_j}"'] &
\clF(\bbC)
\end{tikzcd}
\]
commutes for each $j\in [m]$, where $\mu_{I_j}$ denotes the linear map defined in \cref{dfn:commult}, we have by \cref{rmk:ntermop}, 
\[
\theta'_{A, \clF}(f)^L_U(a)=
\clF^{\{D_{R_j}(w_j)\}_{j=1}^m}_U
\bigl(
\otimes_{j=1}^m(\clF^{D_{R_j}(w_j)}_{\bbC})^{-1}(\prod_{i\in I_j}f(a_i))
\bigr). 
\]
On the other hand, we have
\[
\theta_{A, \clF}'(f)^M_U(F_A^{\loc})^L_M(a)=
\clF^{\{D_{R_j}(w_j)\}_{j=1}^m}_U
\bigl(
\otimes_{j=1}^m(\clF^{D_{R_j}(w_j)}_{\bbC})^{-1}f(\prod_{i\in I_j}a_i)
\bigr).
\]
Thus, the lemma follows. 
\end{proof}

Let $f\colon A\to \clF(\bbC)$ be a morphism in $\CAlg_{\bbC}$. By \cref{lem:thetaLM}, for an open subset $U\subset \bbC$, the universality of $\bfF_A^{\loc}(U)=\varinjlim_{L\in \frB(U)}F_A^{\loc}(L)$ yields a unique linear map $\theta'_{A, \clF}(f)_U\colon \bfF_A^{\loc}(U)\to \clF(U)$ which satisfies
\[
\theta'_{A, \clF}(f)_U\circ (\wt{F}^{\loc}_A)^L_U=\theta'_{A, \clF}(f)^L_U
\]
%\[
%\begin{tikzcd}[row sep=huge, column sep=huge]
%F_A^{\loc}(L) \arrow[r, "(\wt{F}_A^{\loc})^L_U"] \arrow[rd, "\theta'_{A, \clF}(f)^L_U"']&
%\bfF_A^{\loc}(U) \arrow[d, "\theta'_{A, \clF}(f)_U"] \\
%&
%\clF(U)
%\end{tikzcd}
%\]
for all $L\in \frB(U)$. 

\begin{lem}\label{lem:theta'inPFA}
For a morphism $f\colon A\to \clF(\bbC)$ in $\CAlg_{\bbC}$, 
\[
\theta'_{A, \clF}(f)\ceq 
\{\theta'_{A, \clF}(f)_U\}_{U\in \frU_{\bbC}}\colon \bfF_A^{\loc}\to \clF
\]
is a morphism in $\LPFA$. 
\end{lem}

\begin{proof}
We proceed the following steps: 
\begin{itemize}
\item
$\theta'_{A, \clF}(f)$ is a morphism of precosheaves: 
Let $U, V\subset \bbC$ be open subsets such that $U\subset V$. For $L=D_{r_1(z_1)}\sqcup \cdots \sqcup D_{r_l}(z_l)\in \frB(U)$ and $a=a_1\otimes \cdots \otimes a_l\in F_A^{\loc}(L)$, 
\begin{align*}
\clF^U_V\theta'_{A, \clF}(f)_U(\wt{F}_A^{\loc})^L_U(a)
&=
\clF^{\{D_{r_i}(z_i)\}_{i=1}^l}_V
\bigl(
\otimes_{i=1}^l(\clF^{D_{r_i}(z_i)}_{\bbC})^{-1}f(a_i)
\bigr)\\
&=
\theta'_{A, \clF}(f)^L_U(a)\\
&=
\theta'_{A, \clF}(f)_V(\bfF_A^{\loc})^U_V(\wt{F}_A^{\loc})^L_U(a). 
\end{align*}
Thus, we have $\clF^U_V\circ\theta'_{A, \clF}(f)_U=\theta'_{A, \clF}(f)_V\circ (\bfF_A^{\loc})^U_V$. 

\item
$\theta'_{A, \clF}(f)$ is a morphism of prefactorization algebras: 
Let $U, V, W\subset \bbC$ be open subsets such that $U\sqcup V\subset W$. For $L=D_{r_1}(z_1)\sqcup \cdots \sqcup D_{r_l}(z_l)\in \frB(U)$, $M=D_{s_1}(w_1)\sqcup \cdots \sqcup D_{s_m}(w_m)\in \frB(V)$ and $a=a_1\otimes \cdots \otimes a_l\in F_A^{\loc}(L)$, $b=b_1\otimes \cdots \otimes b_m\in F_A^{\loc}(M)$, 
\begin{align*}
&\clF^{U, V}_W(\theta'_{A, \clF}(f)_U\otimes \theta'_{A, \clF}(f)_V)((\wt{F}_A^{\loc})^L_U\otimes (\wt{F}_A^{\loc})^M_V)(a\otimes b)\\
&=
\clF^{\{D_{r_i}(z_i), D_{s_j}(w_j)\}_{i=1, j=1}^{l, m}}_W
\bigl(
\otimes_{i=1}^l(\clF^{D_{r_i}(z_i)}_{\bbC})^{-1}(a_i)\otimes \otimes_{j=1}^m(\clF^{D_{s_j}(w_j)}_{\bbC})^{-1}(b_j)
\bigr)\\
&=
\theta'_{A, \clF}(f)^{L\sqcup M}_W(a\otimes b)\\
&=
\theta'_{A, \clF}(f)_W(\bfF_A^{\loc})^{U, V}_W((\wt{F}_A^{\loc})^L_U\otimes (\wt{F}_A^{\loc})^M_V)(a\otimes b). 
\end{align*}
Thus, we have $\clF^{U, V}_W\circ (\theta'_{A, \clF}(f)_U\otimes \theta'_{A, \clF}(f)_V)=\theta'_{A, \clF}(f)_W\circ (\bfF^{\loc}_A)^{U, V}_W$. Also, we have
\[
\theta'_{A, \clF}(f)_{\varnothing}(1_{\bfF_A^{\loc}})=
\theta'_{A, \clF}(f)^{\varnothing}_{\varnothing}(1_{\bbC})=
1_{\clF}. 
\qedhere
\]
\end{itemize}
\end{proof}

By \cref{lem:theta'inPFA}, for $A\in \Ob(\CAlg_{\bbC})$ and $\clF\in \Ob(\LPFA)$, we have a map
\[
\theta_{A, \clF}'\colon \Hom_{\CAlg_{\bbC}}(A, \clF(\bbC))\to \Hom_{\LPFA}(\bfF_A^{\loc}, \clF), \quad
f\mapsto \theta_{A, \clF}'(f). 
\]

\begin{lem}\label{lem:theta'isinv}
For $A\in \Ob(\CAlg_{\bbC})$ and $\clF\in \Ob(\LPFA)$, the map $\theta'_{A, \clF}$ is the inverse of $\theta_{A, \clF}$. 
\end{lem}

\begin{proof}
Let $\varphi\colon \bfF_A^{\loc}\to \clF$ be a morphism in $\LPFA$. For an open subset $U\subset \bbC$, $L=D_{r_1}(z_1)\sqcup \cdots \sqcup D_{r_l}(z_l)\in \frB(U)$ and $a=a_1\otimes \cdots \otimes a_l\in F_A^{\loc}(L)$, 
\begin{align*}
\theta'_{A, \clF}(\theta_{A, \clF}(\varphi))_U(\wt{F}_A^{\loc})^L_U(a)
&=
\clF^{\{D_{r_i}(z_i)\}_{i=1}^l}_U
\bigl(
\otimes_{i=1}^l(\clF^{D_{r_i}(z_i)}_{\bbC})^{-1}\varphi_{\bbC}(\wt{F}_A^{\loc})^{D_{r_i}(z_i)}_{\bbC}(a_i)
\bigr)\\
&=
\clF^{\{D_{r_i}(z_i)\}_{i=1}^l}_U
\bigl(
\otimes_{i=1}^l\varphi_{D_{r_i}(z_i)}((\bfF_A^{\loc})^{D_{r_i}(z_i)}_{\bbC})^{-1}(\wt{F}_A^{\loc})^{D_{r_i}(z_i)}_{\bbC}(a_i)
\bigr)\\
&=
\clF^{\{D_{r_i}(z_i)\}_{i=1}^l}_U
\bigl(
\otimes_{i=1}^l\varphi_{D_{r_i}(z_i)}(\wt{F}_A^{\loc})^{D_{r_i}(z_i)}_{D_{r_i}(z_i)}(a_i)
\bigr)\\
&=
\varphi_U(\bfF_A^{\loc})^{\{D_{r_i}(z_i)\}_{i=1}^l}_U
\bigl(
\otimes_{i=1}^l(\wt{F}_A^{\loc})^{D_{r_i}(z_i)}_{D_{r_i}(z_i)}(a_i)
\bigr)\\
&=
\varphi_U(\wt{F}_A^{\loc})^L_U(a). 
\end{align*}
Thus, we have $\theta'_{A, \clF}\theta_{A, \clF}(\varphi)=\varphi$. 

Let $f\colon A\to \clF(\bbC)$ be a morphism in $\CAlg_{\bbC}$. For $a\in A$, take some $R\in \bbR_{>0}$, then
\[
\theta_{A, \clF}(\theta'_{A, \clF}(f))(a)=
\theta'_{A, \clF}(f)_{\bbC}(\wt{F}_A^{\loc})^{D_R(0)}_{\bbC}(a)=
\clF^{D_R(0)}_{\bbC}=f(a). 
\]
Thus, we have $\theta_{A, \clF}\theta'_{A, \clF}(f)=f$. 
\end{proof}

By \cref{lem:theta'isinv}, the proof of \cref{prp:adj} is complete. 

\medskip

Now, let us denote by
\[
\eta_A\colon A\to \bfF_A^{\loc}(\bbC) \quad (A\in \Ob(\CAlg_{\bbC})), \quad
\epsilon_{\clF}\colon \bfF_{\clF(\bbC)}^{\loc}\to \clF \quad (\clF\in \Ob(\LPFA))
\]
the unit and counit of the adjoint functors $\bfF^{\loc}\colon \CAlg_{\bbC}\rightleftarrows \LPFA\colon (-)(\bbC)$. 

\begin{prp}\label{prp:Flocfaifull}
The functor $\bfF^{\loc}\colon \CAlg_{\bbC}\to \LPFA$ is fully faithful. 
\end{prp}

\begin{proof}
For $A\in \Ob(\CAlg_{\bbC})$, the morphism $\eta_A$ is equal to $(\wt{F}_A^{\loc})^{D_R(z)}_{\bbC}=(\bfF^{\loc}_A)^{D_R(z)}_\bbC\circ (\wt{F}^{\loc}_A)^{D_R(z)}_{D_R(z)}$. Thus, the proposition follows from \cref{lem:unitisom}. 
\end{proof}

\begin{lem}\label{lem:epsisom}
If $\clF\in \Ob(\LPFA)$ is a factorization algebra, then the counit $\epsilon_{\clF}\colon \bfF^{\loc}_{\clF(\bbC)}\to \clF$ is an isomorphism in $\LPFA$. 
\end{lem}

\begin{proof}
It is enough to show that $\epsilon_{\clF, U}\colon \bfF^{\loc}_{\clF(\bbC)}\to \clF(U)$ is a linear isomorphism for all open subsets $U\subset \bbC$. We proceed the following steps: 
\begin{itemize}
\item
For a connected open subset $U\subset \bbC$, the linear map $\epsilon_{\clF, U}$ is an isomorphism: 
Take $R\in \bbR_{>0}$ and $z\in \bbC$ such that $D_R(z)\subset U$, then we have 
\[
\epsilon_{\clF, U}\circ (\wt{F}^{\loc}_{\clF(\bbC)})^{D_R(z)}_U=
\theta'_{\clF(\bbC), \clF}(\id_{\clF(\bbC)})^{D_R(z)}_U
\]
%\[
%\begin{tikzcd}[row sep=huge, column sep=huge]
%F_{\clF(\bbC)}^{\loc}(D_R(z)) \arrow[r, "(\wt{F}_{\clF(\bbC)}^{\loc})^{D_R(z)}_U"] \arrow[rd, "\theta'_{\clF(\bbC), \clF}(\id_{\clF(\bbC)})^{D_R(z)}_U"'] &
%\bfF^{\loc}_{\clF(\bbC)}(U) \arrow[d, "\epsilon_{\clF, U}"]\\
%&
%\clF(U)
%\end{tikzcd}
%\]
Thus, it is sufficient to check that both $(\wt{F}^{\loc}_{\clF(\bbC)})^{D_R(z)}_U$ and $\theta'_{\clF(\bbC), \clF}(\id_{\clF(\bbC)})^{D_R(z)}_U$ are isomorphisms, but this follows from \cref{lem:locFAdomain}. 

\item
For an open subset $U\subset \bbC$ which has finite many connected components, the linear map $\epsilon_{\clF, U}$ is an isomorphism: In fact, if $U=U_1\sqcup \cdots \sqcup U_l$ is the decomposition into connected components, then 
\[
\epsilon_{\clF, U}\circ (\bfF^{\loc}_{\clF(\bbC)})^{\{U_i\}_{i=1}^l}_U=
\clF^{\{U_i\}_{i=1}^l}_U\circ (\otimes_{i=1}^l\epsilon_{\clF, U_i}). 
\]
%\[
%\begin{tikzcd}[row sep=huge, column sep=huge]
%\bigotimes_{i=1}^l\bfF_{\clF(\bbC)}^{\loc}(U_i) \arrow[r, "(\bfF_{\clF(\bbC)}^{\loc})^{\{U_i\}_{i=1}^l}_U"] \arrow[d, "\otimes_{i=1}^l\epsilon_{\clF, U_i}"'] & [1.5cm]
%\bfF_{\clF(\bbC)}^{\loc}(U) \arrow[d, "\epsilon_{\clF, U}"] \\
%\bigotimes_{i=1}^l\clF(U_i) \arrow[r, "\clF^{\{U_i\}_{i=1}^l}_U"'] &
%\clF(U) 
%\end{tikzcd}
%\]

\item
For arbitrary open subset $U\subset \bbC$, the linear map $\epsilon_{\clF, U}$ is an isomorphism: First, note that $\frB(U)$ is a Weiss cover of $U$. To construct the inverse map of $\epsilon_{\clF, U}$, we use the coequalizer diagram
\[
\begin{tikzcd}
\displaystyle\bigoplus_{L, M\in \frB(U)}\clF(L\cap M) \arrow[r, shift left=.75ex,"p"] \arrow[r, shift right=.75ex,swap,"q"] 
& \displaystyle\bigoplus_{L\in \frB(U)}\clF(L) \arrow[r, "\pi"]
& \clF(U),
\end{tikzcd}
\]
where $p, q, \pi$ are linear maps defined as in \eqref{eq:coeq}. Let $f\colon \bigoplus_{L\in \frB(U)}\clF(L)\to \bfF_{\clF(\bbC)}^{\loc}(U)$ be the linear map which commutes 
\[
\begin{tikzcd}[row sep=huge, column sep=huge]
\clF(L) \arrow[r, hookrightarrow] \arrow[d, "\epsilon_{\clF, L}^{-1}"']& 
\bigoplus_{L\in \frB(U)}\clF(L) \arrow[d, "f"] \\
\bfF^{\loc}_{\clF(\bbC)}(L) \arrow[r, "(\bfF_{\clF(\bbC)}^{\loc})^L_U"'] &
\bfF^{\loc}_{\clF(\bbC)}(U)
\end{tikzcd}
\]
then one can check $f\circ p=f\circ q$ since for $L, M\in \frB(U)$, the intersection $L\cap M$ has finite many connected components. Thus, by the universality of the coequalizer $\clF(U)$, there exists a unique linear map $\wt{f}\colon \clF(U)\to \bfF_{\clF(\bbC)}^{\loc}(U)$ such that $\wt{f}\circ\pi=f$. We find that $\wt{f}$ is the inverse map of $\epsilon_{\clF, U}$. 
\qedhere
\end{itemize}
\end{proof}

\begin{prp}\label{prp:catequiv}
The functors
\[
\bfF^{\loc}\colon \CAlg_{\bbC}\rightleftarrows \LFA\colon (-)(\bbC)
\]
give categorical equivalences. 
\end{prp}

\begin{proof}
This follows from \cref{prp:adj}, \cref{prp:Flocfaifull} and \cref{lem:epsisom}. 
\end{proof}

\section{Constructions of vertex algebras and factorization algebras}\label{s:constVALFA}

We continue to work over the field $\bbC$ of complex numbers, so that all linear spaces and linear maps defined over $\bbC$ unless otherwise stated. 

\subsection{Preliminaries on vertex algebras}\label{ss:VA}

In this subsection we recall the definition of vertex algebras. We refer to \cite{FBZ} for details on the theory of vertex algebras. 

Let $V$ be a linear space. A formal series 
\[
a(z)=\sum_{n\in \bbZ}z^{n}a_{n}\in (\End_{\bbC}V)\dbr{z^{\pm1}}
\]
is called a field on $V$ if $a(z)v\in V\dpr{z}$ for each $v\in V$. We denote by $\Fld V$ the linear space consisting of all fields on $V$. 

\begin{dfn}[{\cite[\S1.3.1]{FBZ}}]\label{dfn:VA}
A \emph{vertex algebra} is a tuple $(V, \vac, T, Y)$ consisting of
\begin{itemize}
\item 
a linear space $V$, called the state space, 

\item 
an element $\vac\in V$, called the vacuum, 

\item 
a linear map $T\colon V\to V$, called the translation operator, 

\item 
a linear map 
\[
Y\colon V\to \Fld V, \quad 
a\mapsto Y(a, z)=\sum_{z\in \bbZ}z^{-n-1}a_{(n)}, 
\]
called the state-field correspondence, 
\end{itemize}
which satisfies following axioms: 
\begin{clist}
\item (Vacuum axiom) 
For $a\in V$, we have $Y(\vac, z)=\id_{V}$. Also, for $a\in V$, we have $Y(a, z)\vac\in V\dbr{z}$ and $Y(a, z)\vac|_{z=0}=a$. 

\item (Translation invariance) 
We have $T\vac=0$ and $[T, Y(a, z)]=\pdd_zY(a, z)$ for $a\in V$. 

\item (Locality)
For $a, b\in V$, there exists $N\in \bbN$ such that
\[
(z-w)^N[Y(a, z), Y(b, w)]=0
\]
in $(\End_{\bbC}V)\dbr{z^{\pm1}, w^{\pm1}}$. 
\end{clist}
\end{dfn}

\begin{dfn}[{\cite[Definition 1.3.4]{FBZ}}]
Let $V$ and $W$ be vertex algebras. A linear map $f\colon V\to W$ is called a \emph{morphism of vertex algebras} if  
\[
f(\vac_V)=\vac_W, \quad
f\circ T_V=T_W\circ f, \quad
f(Y_V(a, z)b)=Y_W(f(a), z)f(b) \quad (a, b\in V). 
\]
\end{dfn}

\begin{eg}[{\cite[\S1.4]{FBZ}}]\label{eg:comVA}
A vertex algebra $V$ is called commutative if $[Y(a, z), Y(b, w)]=0$ for all $a, b\in V$. This condition is equivalent to $a_{(n)}=0$ (as an element of $\End_{\bbC}V$) for all $a\in V$ and $n\in \bbN$. 

For a commutative vertex algebra $V$, 
\[
\mu\colon V\times V\to V, \quad (a, b)\mapsto a_{(-1)}b
\]
is a commutative and associative multiplication, and the translation operator $T$ is a derivation with respect to $\mu$. Hence we obtain a commutative differential algebra $(V, \mu, T)$ with the unit $\vac$. 
Conversely, for a commutative differential algebra $(A, \cdot, \pdd)$ with the unit $1_A$, let 
\[
Y\colon A\otimes A\to A\dbr{z}, \quad
a\otimes b\mapsto 
Y(a, z)b\ceq \sum_{n\ge 0}\frac{z^n}{n!}(\pdd^na)\cdot b, 
\]
then $(A, 1_A, \pdd, Y)$ is a commutative vertex algebra. Thus, the notion of commutative vertex algebras is equivalent to that of commutative differential algebras. In what follows, we does not distinguish commutative vertex algebras and commutative differential algebras. 
\end{eg}

\begin{dfn}\label{dfn:ZgrVA}
A vertex algebra $V=\bigoplus_{\Delta\in \bbZ}V_\Delta$ with a direct sum decomposition as a linear space is called a \emph{$\bbZ$-graded vertex algebra} if the following conditions hold: 
\begin{clist}
\item 
$\vac\in V_0$. 

\item 
For $a\in V_{\Delta}$, $b\in V_{\Delta'}$ and $n\in \bbZ$, we have $a_{(n)}b\in V_{\Delta+\Delta'-n-1}$. 
\end{clist}
\end{dfn}

Let $V=\bigoplus_{\Delta\in \bbZ}V_\Delta$ be a $\bbZ$-graded vertex algebras. For $a\in V_\Delta$, we say that $a$ is a homogeneous element of $V$, and denote $\Delta(a)\ceq \Delta$. 

\begin{dfn}
Let $V=\bigoplus_{\Delta\in \bbZ}V_\Delta$ and $W=\bigoplus_{\Delta\in \bbZ} W_\Delta$ be $\bbZ$-graded vertex algebras. A morphism $f\colon V\to W$ of vertex algebras is called a \emph{morphism of $\bbZ$-graded vertex algebras} if $f(V_\Delta)\subset W_\Delta$ for all $\Delta\in \bbZ$. 
\end{dfn}

We denote by $\cbgrVA$ the category whose objects are commutative $\bbZ$-graded vertex algebras $V=\bigoplus_{\Delta\in \bbZ}V_\Delta$ such that $V_\Delta=0$ for $\Delta\ll0$, and morphisms are the morphisms of $\bbZ$-graded vertex algebras. Also, we denote by $\fcbgrVA$ the full subcategory of $\cbgrVA$ consisting of $V=\bigoplus_{\Delta\in \bbZ}V_\Delta\in \Ob(\cbgrVA)$ such that $\dim_{\bbC}V_\Delta<\infty$ for all $\Delta\in \bbZ$. 

\begin{eg}
It is known that the forgetful functor $\cbgrVA\to \CAlg$ has the left adjoint $\clJ\colon \CAlg\to \cbgrVA$. For a commutative algebra $A$, the differential algebra $\clJ\!A$ is called the jet algebra of $A$. 

For example, let us describe the construction of the jet algebra of a polynomial ring $A=\bbC[x_\lambda]_{\lambda\in \Lambda}$. As a commutative algebra, the jet algebra is just a polynomial ring $\clJ\!A\ceq \bbC[x_\lambda^{(m)}]_{\lambda\in\Lambda, m\in \bbN}$ of infinite many variables. Define a derivation $\pdd\colon \clJ\!A\to \clJ\!A$ by
\[
\pdd(x_{\lambda_1}^{(m_1)}\cdots x_{\lambda_l}^{(m_l)})\ceq 
\sum_{i=1}^lx_{\lambda_1}^{(m_1)}\cdots x_{\lambda_i}^{(m_i+1)}\cdots x_{\lambda_l}^{(m_l)}. 
\]
Also, for $\Delta\in \bbN$, define
\[
(\clJ\!A)_\Delta\ceq
\Spn_{\bbC}\bigl\{x_{\lambda_1}^{(m_1)}\cdots x_{\lambda_l}^{(m_l)}\mid l\in \bbN, \lambda_i\in \Lambda, m_i\in\bbN,\, \sum_{i=1}^l(m_i+1)=\Delta\bigr\}. 
\]
Then, we have $\clJ\!A=\bigoplus_{\Delta\in\bbN}(\clJ\!A)_\Delta\in \Ob(\cbgrVA)$. Notice that if $A$ is a polynomial ring of finite many variables, then $\clJ\!A\in \Ob(\fcbgrVA)$. 
This holds in general, i.e., we have $\clJ\!A\in \Ob(\fcbgrVA)$ for any finitely generated commutative algebra $A$. 
\end{eg}

\subsection{A general construction of vertex algebras from prefactorization algebras}
\label{ss:constVA}

In this subsection we construct a vertex algebra from a prefactorization algebra on the complex number plane $\bbC$  called holomorphic (\cref{thm:PFAVA}, see also \cref{dfn:holoPFA}). We use the following notations: 

\begin{itemize}
\item 
We denote by $S^1\ceq \{q\in \bbC\mid |q|=1\}$ the unit circle, and consider the natural action of the group $S^1\ltimes \bbC$ of isometric affine transformations on the plane $\bbC$. 

\item 
We regard $S^1\subset S^1\ltimes \bbC$ by $q\mapsto (q, 0)$. 
\end{itemize}

\vspace{4pt}

We will consider factorization algebras with values in the category $\LCS$ of locally convex spaces over $\bbC$, whose morphisms are continuous linear maps, so we first remark on the categorical properties of  $\LCS$. 

\begin{itemize}
\item 
$\LCS$ admits projective limits: 
For a projective system $(\{X_\lambda\}_{\lambda\in \Lambda}, \{f_{\lambda, \mu}\colon X_\mu\to X_\lambda\}_{\lambda\le \mu})$ in $\LCS$, the projective limit $\varprojlim X_\lambda$ in $\Lin$ equipped with the initial topology induced by the canonical morphisms $\varprojlim_{\lambda\in \lambda}X_\lambda\to X_\lambda$ is a locally convex space. Thus, the linear space $\varprojlim_{\lambda\in \Lambda}X_\lambda$ gives rise to the projective limit in $\LCS$. 

\item 
$\LCS$ admits kernels: 
For a continuous linear map $f\colon X\to Y$, the linear space $\Ker f=\{x\in X\mid f(x)=0\}\subset X$ equipped with the subspace topology is a locally convex space. Thus, the linear space $\Ker f$ gives rise to the kernel in $\LCS$. 

\item 
$\LCS$ admits inductive limits: 
For a inductive system $(\{X_\lambda\}_{\lambda\in \Lambda}, \{f_{\lambda, \mu}\colon X_\lambda\to X_\mu\}_{\lambda\le \mu})$ in $\LCS$, the inductive limit $\varinjlim_{\lambda\in \Lambda}X_\lambda$ in $\Lin$ has the structure of a locally convex space as follows: Let $\frB$ be the set consisting of balanced absorbent convex set $A\subset \varinjlim_{\lambda\in \Lambda}X_\lambda$ such that $\iota_\lambda^{-1}(A)\subset X_\lambda$ is a neighborhood of $0\in X_\lambda$ for all $\lambda\in \Lambda$. Here $\iota_\lambda\colon X_\lambda\to \varinjlim_{\lambda\in \Lambda}X_\lambda$ denotes the canonical morphism of the inductive limit. Then, there exists a unique locally convex topology on $\varinjlim_{\lambda\in \Lambda}X_\lambda$ such that $\frB$ is a neighborhood basis at $0\in \varinjlim_{\lambda\in \Lambda}X_\lambda$. The linear space $\varinjlim_{\lambda\in \Lambda}X_\lambda$ equipped with this locally convex topology gives the inductive limit in $\LCS$. 

\item 
$\LCS$ admits cokernels: 
For a continuous linear map $f\colon X\to Y$, the linear space $\Cok f=Y/f(X)$ equipped with the quotient topology is a locally convex space. Thus, the linear space $\Cok f$ gives rise to the cokernel in $\LCS$. 

\item 
$\LCS$ admits a symmetric monoidal structure: 
For locally convex spaces $X$ and $Y$, the algebraic tensor product $X\otimes Y$ has the structure of a locally convex space as follows: Let $\frB$ be the set consisting of balanced absorbent convex set $A\subset X\otimes Y$ such that $\otimes^{-1}(A)\subset X\times Y$ is a neighborhood of $(0, 0)\in X\times Y$. Here $\otimes \colon X\times Y\to X\otimes Y$ denotes the canonical bilinear map. Then, there exists a unique locally convex topology on $X\otimes Y$ such that $\frB$ is a neighborhood basis at $0\in X\otimes Y$. The linear space $X\otimes Y$ equipped with this topology is called the projective tensor product, and denote by $X\otimes_{\pi}Y$. The topology of $X\otimes_\pi Y$ is the finest locally convex topology on $X\otimes Y$ such that $\otimes \colon X\times Y\to X\otimes Y$ is continuous. 
The category $\LCS$ has the structure of a symmetric monoidal category by letting $\otimes_{\pi}$ be the tensor product and $\bbC$ be the unit. In this note, we does not use any other locally convex topologies on $X\otimes Y$, so we just denote $\otimes=\otimes_{\pi}$. 
\end{itemize}

In particular, the category $\LCS$ is a complete and cocomplete symmetric monoidal category. 

\begin{rmk}
For locally convex spaces $X$ and $Y$, there is the finest locally convex topology on $X\otimes Y$ such that $\otimes \colon X\times Y\to X\otimes Y$ is separately continuous. The linear space $X\otimes Y$ equipped with this topology is called the inductive tensor product, and denote by $X\otimes_\iota Y$. The inductive tensor product also give $\LCS$ a symmetric monoidal structure. This \cref{ss:constVA} and next \cref{ss:constcomVA} are valid if we adopt the inductive tensor product as a symmetric monoidal structure on $\LCS$. 
\end{rmk}

Let $\clF\colon \frU_{\bbC}\to \LCS$ be an $S^1\ltimes \bbC$-equivariant prefactorization algebra. We denote its structure morphism by
\[
\sigma_{(q, z), U}\colon \clF(U)\to \clF((q, z)U) \quad ((q, z)\in S^1\ltimes \bbC, U\in\frU_{\bbC}). 
\]
Notice that the costalk $\clF^z$ at $z\in \bbC$ is given by
\[
\clF^z=
\Bigl\{
(a_R)_{R\in \bbR_{>0}}\in \prod_{R\in \bbR_{>0}}\clF(D_R(z))\mid  \forall r, R\in \bbR_{>0},\ r<R\,\Longrightarrow\, \clF^{D_r(z)}_{D_R(z)}(a_r)=a_R
\Bigr\} 
\]
since $\{D_R(z)\}_{R\in \bbR_{>0}}$ is an open neighborhood basis of $z$ which is closed under finite intersections. 

\begin{dfn}\label{dfn:muzR}
For $R\in \bbR_{>0}$ and $(z_1, \ldots, z_l)\in \Conf_l(D_R(0))$, we define a morphism $\mu_{z_1, \ldots, z_l}^R$ of locally convex spaces as the composition 
\[
\begin{tikzcd}[row sep=huge, column sep=huge]
(\clF^0)^{\otimes l} \arrow[r, "\otimes_{i=1}^l\clF^0_{D_r(0)}"] & [0.7cm]
\bigotimes_{i=1}^l\clF(D_r(0)) \arrow[r, "\otimes_{i=1}^l\sigma_{(1, z_i), D_r(0)}"] & [0.8cm]
\bigotimes_{i=1}^l\clF(D_r(z_i)) \arrow[r, "\clF^{D_r(z_1), \ldots, D_r(z_l)}_{D_R(0)}"] & [1cm] 
\clF(D_R(0))
\end{tikzcd}
\]
where $r$ is any positive real number satisfying 
\[
D_r(z_1)\sqcup \cdots \sqcup D_r(z_l)\subset D_R(0).
\]
Notice that the definition of $\mu_{z_1, \ldots, z_l}^R$ is independent of the choice of $r$ because of the second part of \cref{dfn:PFA} (i) and the equation $\clF^{D_r(z)}_{D_s(z)}\circ \clF^z_{D_r(z)}=\clF^z_{D_s(z)}$ for $0<r<s$. 
\end{dfn}

\begin{rmk}\label{rmk:expremuzR}
\ 
\begin{enumerate}
\item 
When $l=1$, the morphism $\mu_z^R\colon \clF^0\to \clF(D_R(0))$ is defined for $z\in D_R(0)$, and we have, for $a\in \clF^0$, 
\[
\mu_z^R(a)=
\clF^{D_r(z)}_{D_R(0)}\sigma_{(1, z), D_r(0)}\clF^0_{D_r(0)}(a), 
\]
where $r$ is any positive real number satisfying $|z|<R-r$. 

\item 
When $l=2$, the morphism $\mu_{z, 0}^{R}\colon \clF^0\otimes \clF^0\to \clF(D_R(0))$ is defined for $z\in D_R^\times(0)\ceq D_R(0)\bs\{0\}$, and we have, for $a, b\in \clF^0$, 
\[
\mu_{z, 0}^R(a\otimes b)=
\clF^{D_r(z), D_r(0)}_{D_R(0)}
\bigl(\sigma_{(1, z), D_r(0)}\clF^0_{D_r(0)}(a)\otimes \clF^0_{D_r(0)}(b)
\bigr), 
\]
where $r$ is any positive real number satisfying $2r<|z|<R-r$. 
\end{enumerate}
\end{rmk}

\begin{dfn}
For $\Delta\in \bbZ$, define
\begin{align*}
&\clF(D_R(0))_{\Delta}\ceq 
\{a\in \clF(D_R(0))\mid \forall q\in S^1,\, \sigma_{q, D_R(0)}a=q^\Delta a\} \quad (R\in \bbR_{>0}), \\
&\clF^0_\Delta\ceq
 \Bigl\{
(a_R)_{R\in \bbR_{>0}}\in \prod_{R\in \bbR_{>0}}\clF(D_R(0))_\Delta\mid \forall r, R\in \bbR_{>0},\ r<R\,\Longrightarrow\, \clF^{D_r(0)}_{D_R(0)}(a_r)=a_R
\Bigr\}. 
\end{align*}
Note that $\sum_{\Delta\in \bbZ}\clF(D_r(0))_\Delta\subset \clF(D_r(0))$ and $\sum_{\Delta\in \bbZ}\clF^0_\Delta\subset \clF^0$ are direct sums. For $a\in \clF^0_\Delta$, we say that $a$ is a homogeneous element of weight $\Delta$, and denote $\Delta(a)\ceq \Delta$. 
\end{dfn}

\begin{rmk}
Let $0<r<R$ and $\Delta\in \bbZ$. If $a\in \clF(D_r(0))_\Delta$, then $\clF^{D_r(0)}_{D_R(0)}(a)\in \clF(D_R(0))_\Delta$, so we have a linear map
\[
\clF(D_r(0))_\Delta\to \clF(D_R(0))_\Delta, \quad a\mapsto \clF^{D_r(0)}_{D_R(0)}(a). 
\]
We abusively denote this linear map by $\clF^{D_r(0)}_{D_R(0)}\colon \clF(D_r(0))_\Delta\to \clF(D_R(0))_\Delta$. 
\end{rmk}

\begin{lem}\label{lem:sigmamu}
Let $R\in \bbR_{>0}$ and $(z_1, \ldots, z_l)\in \Conf_l(D_R(0))$. For $q\in S^1$ and homogeneous $a_1, \ldots, a_l\in \clF^0$, we have
\[
\sigma_{q, D_R(0)}\mu_{z_1, \ldots, z_l}^R(a_1\otimes \cdots \otimes a_l)=
q^{\Delta(a_1)+\cdots +\Delta(a_l)}\mu_{qz_1, \ldots, qz_l}^R(a_1\otimes \cdots \otimes a_l). 
\] 
\end{lem}

\begin{proof}
Take $r\in \bbR_{>0}$ such that 
\[
D_r(z_1)\sqcup \cdots \sqcup D_r(z_l)\subset D_R(0), 
\]
then 
\[
\sigma_{q, D_R(0)}\mu_{z_1, \ldots, z_l}^R(a_1\otimes \cdots \otimes a_l)=
\clF^{\{D_r(qz_i)\}_{i=1}^l}_{D_R(0)}
\bigl(
\otimes_{i=1}^l\sigma_{q(1, z_i), D_r(0)}\clF^0_{D_r(0)}(a_i)
\bigr).
\]
Since $q(1, z)=(1, qz)q$ in $S^1\ltimes \bbC$ and $qD_r(0)=D_r(0)$, we have
\begin{align*}
\sigma_{q, D_R(0)}\mu_{z_1, \ldots, z_l}^R(a_1\otimes \cdots \otimes a_l)
&=
\clF^{\{D_r(qz_i)\}_{i=1}^l}_{D_R(0)}
\bigl(
\otimes_{i=1}^l\sigma_{(1, qz_i), D_r(0)}\sigma_{q, D_r(0)}\clF^0_{D_r(0)}(a_i)
\bigr)\\
&=
q^{\Delta(a_1)+\cdots +\Delta(a_l)}\mu_{qz_1, \ldots, qz_l}(a_1\otimes \cdots \otimes a_l). 
\qedhere
\end{align*}
\end{proof}

The following definition is motivated by the assumptions in \cite[Theorem 5.3.3]{CG1}. Notice that the discreteness condition (see condition (iii) in \cite[Theorem 5.3.3]{CG1} and the first paragraph of \cite[\S1.3]{B}) on the weight spaces $\clF(D_R(0))_\Delta$ does not required. 

\begin{dfn}\label{dfn:holoPFA}
An $S^1\ltimes \bbC$-equivalent prefactorization algebra $\clF\colon \frU_{\bbC}\to \LCS$ is called a holomorphic prefactorization algebra if it satisfies the following conditions: 
\begin{clist}
\item 
For $0<r<R$ and $\Delta\in \bbZ$, the map $\clF^{D_r(0)}_{D_R(0)}\colon \clF(D_r(0))_\Delta\to \clF(D_R(0))_\Delta$ is a linear isomorphism. 

\item 
For $R\in \bbR_{>0}$, we have $\clF(D_R(0))_\Delta=0$ ($\Delta\ll0$). 

\item 
For $R\in \bbR_{>0}$, the locally convex space $\clF(D_R(0))$ is quasi-complete and Hausdorff. 

\item 
For $R\in \bbR_{>0}$ and homogeneous $a_1, \ldots, a_l\in \clF^0$, the function
\[
\Conf_l(D_R(0))\to \clF(D_R(0)), \quad 
(z_1, \ldots, z_l)\mapsto \mu^R_{z_1, \ldots, z_l}(a_1\otimes \cdots \otimes a_l)
\]
is holomorphic in the sense of \cref{dfn:holofunc}. 
\end{clist}
\end{dfn}

We denote by $\HolPFA$ the category whose objects are holomorphic prefactorization algebras, and morphisms are the morphisms of $S^1\ltimes \bbC$-equivariant prefactorization algebras. Also, we denote by $\HolFA$ the full sub-category of $\HolPFA$ consisting of holomorphic factorization algebras. 

\begin{rmk}\label{rmk:clF0DRiso}
Let $\clF\colon \frU_{\bbC}\to \LCS$ be a holomorphic prefactorization algebra. 
For $R\in \bbR_{>0}$ and $\Delta\in \bbZ$, notice that
\[
\clF^0_{D_R(0)}\colon \clF^0_\Delta\to \clF(D_R(0))_\Delta
\]
is a linear isomorphism because of (i) in \cref{dfn:holoPFA}. 
\end{rmk}

In the remaining of this subsection, let $\clF\colon \frU_{\bbC}\to \LCS$ be a holomorphic prefactorization algebra. We will construct the structure of a vertex algebra on $\bfV(\clF)\ceq \bigoplus_{\Delta\in \bbZ}\clF^0_\Delta$. 

First, we define
\begin{equation}\label{eq:defvac}
\vac=\vac_{\clF}\ceq \{\clF^{\varnothing}_{D_R(0)}(1_\clF)\}_{R\in \bbR_{>0}}\in \clF^0, 
\end{equation}
where $\bbC\to\clF(\varnothing)$, $1\mapsto 1_{\clF}$ denotes the unit. Since 
\[
\sigma_{q, D_R(0)}\clF^{\varnothing}_{D_R(0)}(1_{\clF})=
\clF^{\varnothing}_{D_R(0)}\sigma_{q, \varnothing}(1_{\clF})=
\clF^{\varnothing}_{D_R(0)}(1_{\clF})
\]
for all $q\in S^1$ and $r\in \bbR_{>0}$, we have $\Delta(\vac)=0$. In particular $\vac\in \bfV(\clF)$. 

Next, we define a linear map $T\colon \bfV(\clF)\to \bfV(\clF)$: 
Let $a\in \clF^0$ be homogeneous. For $R\in \bbR_{>0}$, by \cref{dfn:holoPFA} (iv), the function
\[
D_R(0)\to \clF(D_R(0)), \quad
z\mapsto \mu_z^R(a).
\]
is holomorphic. 
Thus, we have
\[
(Ta)_R\ceq \pdd_z\mu_{z}^R(a)|_{z=0}\in \clF(D_R(0)). 
\]

\begin{lem}\label{lem:defTa}
Let $a\in \clF^0$ be homogeneous.
\begin{enumerate}
\item
For $R\in \bbR_{>0}$, we have $(Ta)_R\in \clF(D_R(0))_{\Delta(a)+1}$. 

\item 
For $0<r<R$, we have $\clF^{D_r(0)}_{D_R(0)}(Ta)_r=(Ta)_R$. 
 \end{enumerate}
\end{lem}

\begin{proof}
(1) For $q\in S^1$, we have
\[
\sigma_{q, D_R(0)}(Ta)_R=
\pdd_z\sigma_{q, D_R(0)}\mu_z^R(a)|_{z=0}=
q^{\Delta(a)}\pdd_z\mu_{qz}^R(a)|_{z=0}=
q^{\Delta(a)+1}(Ta)_R, 
\]
where we used \cref{lem:sigmamu} in the second equality and \cref{lem:compder} in the third equality. 

(2) For $z\in D_r(0)$, since $\clF^{D_r(0)}_{D_R(0)}\mu_z^r(a)=\mu_z^R(a)$, we have
\[
(Ta)_R=
\pdd_z\mu_z^R(a)|_{z=0}=
\clF^{D_r(0)}_{D_R(0)}\pdd_z\mu_z^r(a)|_{z=0}=
\clF^{D_r(0)}_{D_R(0)}(Ta)_r. 
\qedhere
\]
\end{proof}

For homogeneous $a\in \clF^0$, by \cref{lem:defTa}, one can define
\[
Ta\ceq \{(Ta)_R\}_{R\in \bbR_{>0}}\in \clF^0_{\Delta(a)+1}.
\]
Now, we have a linear map
\begin{equation}\label{eq:deftrans}
T\colon \bfV(\clF)\to \bfV(\clF), \quad a\mapsto Ta.
\end{equation}

Finally, we define a linear map $Y\colon \bfV(\clF)\otimes \bfV(\clF)\to \bfV(\clF)\dpr{z}$: 
Let $a, b\in \clF^0$ be homogeneous. 
For $R\in \bbR_{>0}$, we have a holomorphic function 
\[
D_R(0)^{\times}\to \clF(D_R(0)), \quad 
z\mapsto \mu_{z, 0}^R(a\otimes b).
\]
Thus, by \cref{thm:Laurent}, the function $\mu_{z, 0}^R(a\otimes b)$ can be expanded in the form of Laurent series 
\[
\mu_{z, 0}^R(a\otimes b)=
\sum_{n\in \bbZ}z^{-n-1}(a_{(n)}b)_R \quad (z\in D_R(0)^{\times}), \quad
(a_{(n)}b)_R\in \clF(D_R(0)). 
\]

\begin{lem}\label{lem:defnpro}
Let $a, b\in \clF^0$ be homogeneous, and $n\in \bbZ$. 
\begin{enumerate}
\item 
For $R\in \bbR_{>0}$, we have $(a_{(n)}b)_R\in \clF(D_R(0))_{\Delta(a)+\Delta(b)-n-1}$. 

\item 
For $0<r<R$, we have $\clF^{D_r(0)}_{D_R(0)}(a_{(n)}b)_r=(a_{(n)}b)_R$. 
\end{enumerate}
\end{lem}

\begin{proof}
(1) Let $C=\pdd D_{R/2}(0)$ be the positivity oriented circle of radius $R/2$, then
\[
(a_{(n)}b)_R=
\frac{1}{2\pi\sqrt{-1}}\int_C z^n\mu_{z, 0}^R(a\otimes b)\,dz. 
\]
Thus, for $q\in S^1$, we have
\begin{align*}
\sigma_{q, D_R(0)}(a_{(n)}b)_R
&=
\frac{1}{2\pi\sqrt{-1}}\int_C z^n\sigma_{q, D_R(0)}\mu_{z, 0}^R(a\otimes b)\,dz=
\frac{1}{2\pi\sqrt{-1}}\int_C z^nq^{\Delta(a)+\Delta(b)}\mu_{qz, 0}^R(a\otimes b)\,dz\\
&=
\frac{1}{2\pi\sqrt{-1}}\int_C z^nq^{\Delta(a)+\Delta(b)-n-1}\mu_{z, 0}^R(a\otimes b)\,dz=
q^{\Delta(a)+\Delta(b)-n-1}(a_{(n)}b)_R, 
\end{align*}
where we used \cref{lem:sigmamu} in the second equality. 

(2) For $z\in D_r(0)$, since $\clF^{D_r(0)}_{D_R(0)}\mu^r_{z, 0}(a\otimes b)=\mu_{z, 0}^R(a\otimes b)$, we have
\[
\sum_{m\in \bbZ}z^{-m-1}\clF^{D_r(0)}_{D_R(0)}(a_{(m)}b)_r=
\sum_{m\in \bbZ}z^{-m-1}(a_{(m)}b)_R. 
\]
Thus, the claim follows. 
\end{proof}

For homogeneous $a, b\in \clF^0$ and $n\in \bbZ$, by \cref{lem:defnpro}, one can define
\[
a_{(n)}b\ceq \{(a_{(n)}b)_R\}\in \clF^0_{\Delta(a)+\Delta(b)-n-1}. 
\]
Note that \cref{dfn:holoPFA} (ii) implies $a_{(n)}b=0$ for $n\gg 0$. Thus, we have a linear map
\begin{equation}\label{eq:defvertop}
Y\colon \bfV(\clF)\otimes \bfV(\clF)\to \bfV(\clF)\dpr{z}, \quad
a\otimes b\mapsto Y(a, z)b\ceq\sum_{n\in \bbZ}z^{-n-1}a_{(n)}b. 
\end{equation}

\begin{thm}\label{thm:PFAVA}
For a holomorphic prefactorization algebra $\clF\colon \frU_{\bbC}\to \LCS$, the linear space 
\[
\bfV(\clF)=\bigoplus_{\Delta\in \bbZ}\clF^0_\Delta
\] 
has the structure of a $\bbZ$-graded vertex algebra as follows: 
\begin{itemize}
\item 
The vacuum is $\vac$ defined in \cref{eq:defvac}. 

\item
The translation operator is $T\colon \bfV(\clF)\to \bfV(\clF)$ defined in \cref{eq:deftrans}. 

\item 
The state-field correspondence is $Y\colon \bfV(\clF)\otimes \bfV(\clF)\to \bfV(\clF)\dpr{z}$ defined in \eqref{eq:defvertop}
\end{itemize}
\end{thm}

We prove \cref{thm:PFAVA}. 
The conditions (i), (ii) in \cref{dfn:ZgrVA} have already been proved, so it remains to show that the above datum satisfy the axiom of vertex algebras. 

\begin{lem}\label{lem:vac}
Let $R\in \bbR_{>0}$. For $z\in D_R^{\times}(0)$ and $a\in \bfV(\clF)$, we have
\[
\mu_{z, 0}^R(\vac\otimes a)=\clF^0_{D_R(0)}(a), \quad
\mu_{z, 0}^R(a\otimes \vac)=\mu_z^R(a). 
\]
\end{lem}

\begin{proof}
Take $r\in \bbR_{>0}$ satisfying $2r<|z|<R-r$, then
\begin{align*}
\mu_{z, 0}^R(\vac\otimes a)
&=
\clF^{D_r(z), D_r(0)}_{D_R(0)}(\sigma_{(1, z), D_r(0)}\clF^{\varnothing}_{D_r(0)}(1_{\clF})\otimes \clF^0_{D_r(0)}(a))\\
&=
\clF^{D_r(z), D_r(0)}_{D_R(0)}(\clF^{\varnothing}_{D_r(z)}\sigma_{(1, z), \varnothing}(1_\clF)\otimes \clF^0_{D_r(0)}(a))\\
&=
\clF^{D_r(z), D_r(0)}_{D_R(0)}(\clF^{\varnothing}_{D_r(z)}(1_\clF)\otimes \clF^0_{D_r(0)}(a))\\
&=
\clF^{\varnothing, D_r(0)}_{D_R(0)}(1_\clF\otimes \clF^0_{D_r(0)}(a))\\
&=
\clF^{D_r(0)}_{D_R(0)}\clF^0_{D_r(0)}(a)\\
&=
\clF^0_{D_R(0)}(a). 
\end{align*}
Here we used the fact that $\sigma_{(1, z)}\colon \clF\to (1, z)\clF$ is a morphism of precosheaves in the second equality, \cref{dfn:morequivPFA} (iv) in the third equality, \cref{dfn:pcoshf} (ii) and the second part of \cref{dfn:PFA} (i) in the forth equality, the first part of \cref{dfn:PFA} (i) and \cref{dfn:PFA} (iv) in the fifth equality. Similarly, one can prove $\mu_{z, 0}^R(a\otimes \vac)=\mu_z^R(a)$.  
\end{proof}

\begin{proof}[Proof of vacuum axiom]
For $R\in \bbR_{>0}$ and $a\in \bfV(\clF)$, the first equation in \cref{lem:vac} yields that 
\[
(\vac_{(n)}a)_R=
\begin{cases}
\clF^0_{D_R(0)}(a) & (n=-1) \\
0 & (n\neq-1). 
\end{cases}
\]
Thus, we have $Y(\vac, z)=\id_{\bfV(\clF)}$. Also, for $R\in \bbR_{>0}$ and $a\in \bfV(\clF)$, the second equation in \cref{lem:vac} yields that 
\[
(a_{(n)}\vac)_R=
\begin{cases}
0 & (n\ge 0)\\
\mu_0^R(a) & (n=-1)
\end{cases}
=
\begin{cases}
0 & (n\ge 0)\\
\clF^0_{D_R(0)}(a) & (n=-1).
\end{cases}
\]
Thus, we have $Y(a, z)\vac\in \bfV(\clF)\dbr{z}$ and $Y(a, z)\vac|_{z=0}=a$ for each $a\in \bfV(\clF)$. 
\end{proof}

\begin{lem}\label{lem:transop}
Let $R\in \bbR_{>0}$. For $z\in D_R^{\times}(0)$ and $a, b\in \bfV(\clF)$, we have
\[
\sum_{n\in \bbZ}z^{-n-1}(T(a_{(n)}b))_R=
\pdd_z\mu_{z, 0}^R(a\otimes b)+\mu_{z, 0}^R(a\otimes Tb). 
\]
\end{lem}

\begin{proof}
Take $r, s\in \bbR_{>0}$ satisfying 
\[
s<R, \quad 2r<|z|<s-r,
\]
then for $w\in \bbC$ such that $|w|<R-s$, 
\begin{align*}
\sum_{n\in \bbZ}z^{-n-1}\mu_w^R(a_{(n)}b)
&=
\sum_{n\in \bbZ}z^{-n-1}\clF^{D_s(w)}_{D_R(0)}\sigma_{(1, w), D_s(0)}(a_{(n)}b)_s=
\clF^{D_s(w)}_{D_R(0)}\sigma_{(1, w), D_s(0)}\mu_{z, 0}^s(a\otimes b)\\
&=
\clF^{D_s(w)}_{D_R(0)}\sigma_{(1, w), D_s(0)}\clF^{D_r(z), D_r(0)}_{D_s(0)}(\sigma_{(1, z), D_r(0)}\clF^0_{D_r(0)}(a)\otimes \clF^0_{D_r(0)}(b))\\
&=
\clF^{D_s(w)}_{D_R(0)}\clF^{D_r(z+w), D_r(w)}_{D_s(w)}(\sigma_{(1, z+w), D_r(0)}\clF^0_{D_r(0)}(a)\otimes\sigma_{(1, w), D_r(0)}\clF^0_{D_r(0)}(b))\\
&=
\mu_{z+w, w}^R(a\otimes b). 
\end{align*}
Thus, we have
\[
\sum_{n\in \bbZ}z^{-n-1}(T(a_{(n)}b))_R=
\pdd_w\mu_{z+w, w}^R(a\otimes b)|_{w=0}=
\pdd_z\mu_{z, 0}^R(a\otimes b)+\pdd_w\mu_{z, w}^R(a\otimes b)|_{w=0}, 
\]
where we used \cref{lem:compder} in the second equality. 

Next, we calculate $\mu_{z, 0}^R(a\otimes Tb)$. Take $s\in \bbR_{>0}$ satisfying $2s<|z|<R-s$, then
\begin{align*}
\mu_{z, 0}^R(a\otimes Tb)
&=
\clF^{D_s(z), D_s(0)}_{D_R(0)}(\sigma_{(1, z), D_s(0)}\clF^0_{D_s(0)}(a)\otimes (Tb)_s)\\
&=
\clF^{D_s(z), D_s(0)}_{D_R(0)}(\sigma_{(1, z), D_s(0)}\clF^0_{D_s(0)}(a)\otimes \pdd_w\mu^s_w(b)|_{w=0}). 
\end{align*}
Now, take $0<r<s$, then for $w\in \bbC$ such that $|w|<s-r$, 
\begin{align*}
&\clF^{D_s(z), D_s(0)}_{D_R(0)}(\sigma_{(1, z), D_s(0)}\clF^0_{D_s(0)}(a)\otimes \mu_w^s(b))\\
&=
\clF^{D_s(z), D_s(0)}_{D_R(0)}(\sigma_{(1, z), D_s(0)}\clF^0_{D_s(0)}(a)\otimes \clF^{D_r(w)}_{D_s(0)}\sigma_{(1, w), D_r(0)}\clF^0_{D_r(0)}(b))\\
&=
\clF^{D_s(z), D_s(0)}_{D_R(0)}(\sigma_{(1, z), D_s(0)}\clF^{D_r(0)}_{D_s(0)}\clF^0_{D_r(0)}(a)\otimes \clF^{D_r(w)}_{D_s(0)}\sigma_{(1, w), D_r(0)}\clF^0_{D_r(0)}(b))\\
&=
\clF^{D_s(z), D_s(0)}_{D_R(0)}(\clF^{D_r(z)}_{D_s(z)}\sigma_{(1, z), D_r(0)}\clF^0_{D_r(0)}(a)\otimes \clF^{D_r(w)}_{D_s(0)}\sigma_{(1, w), D_r(0)}\clF^0_{D_r(0)}(b))\\
&=
\mu^R_{z, w}(a\otimes b). 
\end{align*}
Thus, we have
\[
\mu_{z, 0}^R(a\otimes Tb)=\pdd_w\mu_{z, w}^R(a\otimes b)|_{w=0}, 
\]
which completes the proof. 
\end{proof}

\begin{proof}[Proof of translation axiom]
For $R\in \bbR_{>0}$ by \cref{lem:vac}, 
\[
\mu_z^R(\vac)=
\mu_{z, 0}^R(\vac\otimes \vac)=\clF^0_{D_R(0)}(\vac), 
\]
which implies $(T\vac)_R=0$. Thus, we have $T\vac=0$. Also, for $R\in \bbR_{>0}$ and $a, b\in \bfV(\clF)$, \cref{lem:transop} yields that
\[
(T(a_{(n)}b))_R=-n(a_{(n-1)}b)_R+(a_{(n)}Tb)_R.
\]
Thus, we have $[T, Y(a, z)]=\pdd_zY(a, z)$ for each $a\in \bfV(\clF)$. 
\end{proof}

\begin{lem}\label{lem:loc}
Let $a, b\in \bfV(\clF)$. There exists $N\in \bbN$ such that the function 
\[
D_R^{\times}(0)^2\bs\{z=w\}\to \clF(D_R(0)), \quad
(z, w)\mapsto (z-w)^N\mu_{z, w, 0}^R(a\otimes b\otimes c)
\]
for each $c\in \bfV(\clF)$ extends to the separately holomorphic function on $D_R^{\times}(0)^2$. 
\end{lem}

\begin{proof}
Take $N\in \bbN$ such that $a_{(n)}b=0$ and $b_{(n)}a=0$ for $n\ge N$. 
Fix $z\in D_R^{\times}(0)$. Take $r\in \bbR_{>0}$ such that $2r<|z|<R-r$, then for $w\in D_r^{\times}(z)$, a direct calculation shows that 
\[
\sum_{n\in \bbZ}(w-z)^{-n-1}\mu_{z, 0}^R(b_{(n)}a\otimes c)=
\mu_{z, w, 0}^R(a\otimes b\otimes c). 
\]
Thus, by \cref{thm:remov}, the point $w=z$ is removable singularity of the function $D_r^{\times}(z)\to \clF(D_R(0))$, $w\mapsto (z-w)^N\mu_{z, w, 0}^R(a\otimes b\otimes c)$. Similarly, one can show that for fixed $w\in D_R^{\times}(0)$, the point $z=w$ is removable singularity of the function $z\mapsto (z-w)^N\mu_{z, w, 0}^R(a\otimes b\otimes c)$. 
\end{proof}

\begin{proof}[Proof of locality]
Let $a, b\in \bfV(\clF)$, and $N\in \bbN$ be as in \cref{lem:loc}. It is enough to show that 
\begin{equation}\label{eq:proofloc}
\sum_{k=0}^N(-1)^k\binom{N}{k}a_{(m+N-k)}(b_{(n+k)}c)=
\sum_{k=0}^N(-1)^k\binom{N}{k}b_{(n+k)}(a_{(m+N-k)}c).
\end{equation}
for $m, n\in \bbZ$ and $c\in \bfV(\clF)$. First, a direct calculation shows the following (1) and (2): 
\begin{enumerate}
\item
For $R\in \bbR_{>0}$ and $z\in D_R^{\times}(0)$, take $r\in \bbR_{>0}$ such that $2r<|z|<R-r$. Then for $w\in D_r^{\times}(0)$, 
\[
\sum_{n\in \bbZ}w^{-n-1}\mu_{z, 0}^R(a\otimes b_{(n)}c)=
\mu_{z, w, 0}^R(a\otimes b\otimes c). 
\]

\item 
For $R\in \bbR_{>0}$ and $w\in D_R^{\times}(0)$, take $r\in \bbR_{>0}$ such that $2r<|w|<R-r$, then for $z\in D_r^{\times}(0)$, 
\[
\sum_{m\in \bbZ}z^{-m-1}\mu_{w, 0}^R(b\otimes a_{(m)}c)=
\mu_{z, w, 0}^R(a\otimes b\otimes c). 
\]
\end{enumerate}

For $R\in \bbR_{>0}$, by \cref{lem:loc}, we have
\[
\Res_{z=0}\Res_{w=0}z^mw^n(z-w)^N\mu_{z, w, 0}^R(a\otimes b\otimes c)=
\Res_{w=0}\Res_{z=0}z^mw^n(z-w)^N\mu_{z, w, 0}^R(a\otimes b\otimes c), 
\]
which together with above (1) and (2) implies that 
\[
\sum_{k=0}^N(-1)^k\binom{N}{k}(a_{(m+N-k)}(b_{(n+k)}c))_R=
\sum_{k=0}^N(-1)^k\binom{N}{k}(b_{(n+k)}(a_{(m+N-k)}c))_R.
\]
Thus, \eqref{eq:proofloc} follows. 
\end{proof}

Now the proof of \cref{thm:PFAVA} is complete. 

\medskip

Let $\clF, \clG\colon \frU_{\bbC}\to \LCS$ be holomorphic prefactorization algebras. For a morphism $\varphi\colon \clF\to \clG$ of holomorphic prefactorization algebras, we have a continuous linear map $\varphi^0\colon \clF^0\to \clG^0$. Since $\varphi^{0}(\clF^0_\Delta)\subset \clG^0_\Delta$ for all $\Delta\in \bbZ$, we get a linear map
\[
\bfV(\varphi)\ceq \varphi^0|_{\bfV(\clF)}\colon \bfV(\clF)\to \bfV(\clG). 
\]

\begin{prp}
For a morphism $\varphi\colon \clF\to \clG$ of holomorphic prefactorization algebras, the linear map $\bfV(\varphi)\colon \bfV(\clF)\to \bfV(\clG)$ is a morphism of $\bbZ$-graded vertex algebras. Hence, the map $\clF\mapsto \bfV(\clF)$ gives rise to the functor
\[
\bfV\colon \HolPFA\to \bgrVA. 
\]
\end{prp}

\begin{proof}
It is enough to show that $\bfV(\varphi)$ is a morphism of vertex algebras. 
\begin{itemize}
\item 
$\bfV(\varphi)$ preserves the vacuums: 
we have
\[
\bfV(\varphi)(\vac_{\clF})=
\{\varphi_{D_R(z)}\clF^{\varnothing}_{D_R(0)}(1_\clF)\}_{R\in \bbR_{>0}}=
\{\clG^{\varnothing}_{D_R(0)}\varphi_{\varnothing}(1_\clF)\}_{R\in \bbR_{>0}}=
\{\clG^{\varnothing}_{D_R(0)}(1_{\clG})\}=
\vac_{\clG}. 
\]

\item 
$\bfV(\varphi)$ commutes with the translation operators: Let $a\in \bfV(\clF)$. For $R\in \bbR_{>0}$, since a direct calculation yields that 
\[
\varphi_{D_R(0)}\mu^R_z(a)=\mu^R_z(\varphi^0(a)) \quad (z\in D_R(0)), 
\]
we have 
\[
\varphi_{D_R(0)}(Ta)_R=
\pdd_z\varphi_{D_R(0)}\mu^R_z(a)|_{z=0}=
\pdd_z\mu_z^R(\varphi^0(a))=(T\varphi^0(a))_R. 
\]
Thus $\bfV(\varphi)(Ta)=T\bfV(\varphi)(a)$ follows. 

\item 
$\bfV(\varphi)$ commutes with the state-field correspondences: Let $a, b\in \bfV(\varphi)$. For $R\in \bbR_{>0}$, since a direct calculation yields that 
\begin{align*}
\sum_{n\in \bbZ}z^{-n-1}\varphi_{D_R(0)}(a_{(n)}b)_R
&=
\varphi_{D_R(0)}\mu^R_{z, 0}(a\otimes b)=
\mu^R_{z, 0}(\varphi^0(a)\otimes \varphi^0(b))\\
&=
\sum_{n\in \bbZ}z^{-n-1}(\varphi^0(a)_{(n)}\varphi^0(b))_R \quad (z\in D_R^{\times}(0)),
\end{align*}
we have 
\[
\varphi_{D_R(0)}(a_{(n)}b)_R=
(\varphi^0(a)_{(n)}\varphi^0(b))_R \quad (n\in \bbZ). 
\]
Thus $\bfV(\varphi)(a_{(n)}b)=\bfV(\varphi)(a)_{(n)}\bfV(\varphi)(b)$ follows for all $n\in \bbZ$. 
\qedhere
\end{itemize}
\end{proof}

\subsection{Commutative vertex algebras from locally constant prefactorization algebras}\label{ss:constcomVA}

In this subsection we show that if the holomorphic prefactorization algebra $\clF\colon \frU_{\bbC}\to \LCS$ is locally constant (\cref{dfn:LPFA}), then the vertex algebra $\bfV(\clF)$ is commutative. 

Let $\clF\colon \frU_{\bbC}\to \LCS$ be a locally constant holomorphic prefactorization algebra. By applying the functor $(-)(\bbC)\colon \LPFA\to \CAlg_{\bbC}$ constructed in \cref{ss:constcomalg}, we have a commutative algebra $\clF(\bbC)$. For $R\in \bbR_{>0}$, since $\clF^{D_R(0)}_{\bbC}\colon \clF(D_R(0))\to \clF(\bbC)$ is a linear isomorphism, the linear space $\clF(D_R(0))$ inherits the structure of a commutative algebra from $\clF(\bbC)$. (See \cref{rmk:clFDRcalg}.) 

\begin{lem}\label{lem:muz0muz}
Let $R\in \bbR_{>0}$. For $z\in D_R(0)^{\times}$ and $a, b\in \clF^0$, we have 
\[
\mu^R_{z, 0}(a\otimes b)=
\mu^R_z(a)\cdot \clF^0_{D_R(0)}(b). 
\]
\end{lem}

\begin{proof}
Take $r\in \bbR_{>0}$ such that $2r<|z|<R-r$, then 
\begin{align*}
\mu^R_{z, 0}(a\otimes b)
&=
\clF^{D_r(z), D_r(0)}_{D_R(0)}\bigl(\sigma_{(1, z), D_r(0)}\clF^0_{D_r(0)}(a)\otimes \clF^0_{D_r(0)}(b)\bigr)\\
&=
\clF^{D_r(z), D_r(0)}_{D_R(0)}\bigl((\clF^{D_r(z)}_{D_R(0)})^{-1}\mu^R_z(a)\otimes (\clF^{D_r(0)}_{D_R(0)})^{-1}\clF^0_{D_R(0)}(b)\bigr)\\
&=
\mu^R_z(a)\cdot \clF^0_{D_R(0)}(b). 
\qedhere
\end{align*}
\end{proof}

\begin{prp}\label{prp:LPFAcomVA}
If $\clF\colon \frU_{\bbC}\to \LCS$ is a locally constant holomorphic prefactorization algebra, then the vertex algebra $\bfV(\clF)$ is commutative. 
\end{prp}

\begin{proof}
It is enough to show that $a_{(n)}b=0$ for $a, b\in \bfV(\clF)$ and $n\in \bbN$, but this follows from \cref{lem:muz0muz} and \cref{thm:remov}. 
\end{proof}

\begin{rmk}\label{rmk:bfVcalgstr}
Let $\clF\colon \frU_{\bbC}\to \LCS$ be a locally constant holomorphic prefactorization algebra. Notice that by \cref{lem:muz0muz}, the multiplication of the commutative vertex algebra $\bfV(\clF)$ is equal to
\[
a\cdot b=\{(a_{(-1)}b)_R\}_{R\in \bbR_{>0}}=
\bigl\{\clF^0_{D_R(0)}(a)\cdot \clF^0_{D_R(0)}(b)\bigr\}_{R\in \bbR_{>0}}. 
\]
\end{rmk}

\subsection{Locally constant factorization algebras from commutative vertex algebras}\label{ss:constLocFA}

In this subsection we construct a locally constant holomorphic factorization algebra from a commutative vertex algebra (\cref{prp:constLHFA}, \cref{prp:VbfVisom}), and discuss the relationship with the jet algebras of commutative algebras (\cref{prp:bfFjet}). 
We continue to use the notation $\frB$ defined in \eqref{eq:basisofC}. Note that $\frB$ is an open basis of $\bbC$ satisfying: 
\begin{clist}
\item 
$\varnothing\in \frB$. 

\item 
For $L, M\in \frB$ such that $L\cap M=\varnothing$, we have $L\sqcup M\in \frB$. 

\item 
For $a\in S^1\ltimes \bbC$ and $L\in \frB$, we have $aL\in \frB$. 
\end{clist}

\medskip

Let $V=\bigoplus_{\Delta\in \bbZ}V_\Delta$ be a commutative $\bbZ$-graded vertex algebra such that 
$V_\Delta=0$ for $\Delta\ll0$. 
We denote $\ol{V}\ceq \prod_{\Delta\in \bbZ}V_\Delta$ and by $\pi_\Delta\colon \ol{V}\to V_\Delta$ the projection. We naturally regard $V\subset \ol{V}$. 

As we explained in \cref{eg:comVA}, the commutative vertex algebra $V$ is a commutative $\bbC$-algebra by the multiplication 
\[
V\times V\to V, \quad (a, b)\mapsto ab\ceq a_{(-1)}b
\]
and the unit $\vac$. 
The condition $V_\Delta=0$ ($\Delta\ll 0$) guarantees that one can define a multiplication on $\ol{V}$ by
\[
\ol{V}\times \ol{V}\to \ol{V}, \quad
(a, b)\mapsto
ab\ceq \bigl\{\sum_{\Delta'\in \bbZ}\pi_{\Delta'}(a)\pi_{\Delta-\Delta'}(b)\bigr\}_{\Delta\in \bbZ}. 
\]
The linear space $\ol{V}$ has the structure of a commutative $\bbC$-algebra by this multiplication and the unit $\vac$. Thus, by applying the functor $\bfF^{\loc}\colon \CAlg_{\bbC}\to \LFA$ constructed in \cref{ss:constLCFA}, we have a locally constant factorization algebra $\bfF_{\ol{V}}^{\loc}\colon \frU_{\bbC}\to \Lin$. 

To define an $S^1\ltimes \bbC$-equivariant structure on $\bfF_{\ol{V}}^{\loc}$, we use the translation operator $T\colon V\to V$. Recall that $T$ is a derivation with respect to the multiplication on $V$. 

For $q\in S^1$, we have a linear map
\[
q^{L_0}\colon \ol{V}\to \ol{V}, \quad
a\mapsto \{q^\Delta \pi_\Delta(a)\}_{\Delta\in \bbZ}. 
\]
Also, for $z\in \bbC$, we have a linear map
\[
e^{zT}\colon \ol{V}\to \ol{V}, \quad 
a\mapsto
\bigl\{\sum_{n\in \bbN}\frac{z^n}{n!}T^n\pi_{\Delta-n}(a)\bigr\}_{\Delta\in \bbZ}. 
\]
Note that both $q^{L_0}$ and $e^{zT}$ are morphisms of commutative $\bbC$-algebras. 

\begin{dfn}
Let $(q, z)\in S^1\ltimes \bbC$. For $L=D_{r_1}(z_1)\sqcup \cdots \sqcup D_{r_l}(z_l)\in \frB$, define
\[
\sigma_{(q, z), L}\colon F_{\ol{V}}^{\loc}(L)\to F_{\ol{V}}^{\loc}((q, z)L), \quad
a_1\otimes \cdots \otimes a_l\mapsto
e^{zT}q^{L_0}a_1\otimes \cdots \otimes e^{zT}q^{L_0}a_l, 
\]
with the convention 
\[
\sigma_{(q, z), \varnothing}\ceq \id_{\bbC}\colon F_{\ol{V}}^{\loc}(\varnothing)\to F_{\ol{V}}^{\loc}((q, z)\varnothing). 
\]
for $l=0$. 
\end{dfn}

\begin{lem}\label{lem:FolVequiv}
The prefactorization algebra $F_{\ol{V}}^{\loc}\colon \frB\to \Lin$ is an $S^1\ltimes \bbC$-equivariant prefactorization algebra by the structure morphism $\sigma_{(q, z)}\ceq \{\sigma_{(q, z), L}\}_{L\in \frB}\colon F_{\ol{V}}^{\loc}\to (q, z)F_{\ol{V}}^{\loc}$. 
\end{lem}

\begin{proof}
First, for $(q, z)\in S^1\ltimes \bbC$, it is easy to see that $\sigma_{(q, z)}\colon F_{\ol{V}}^{\loc}\to (q, z)F_{\ol{V}}^{\loc}$ is a morphism of precosheaves since both $q^{L_0}, e^{zT}\colon \ol{V}\to \ol{V}$ are morphisms of commutative $\bbC$-algebras. We check the conditions (i)-(iv) in \cref{dfn:equivPFA}. 
\begin{clist}
\item
Let $(q, z), (q', z')\in S^1\ltimes \bbC$ and $L=D_{r_1}(z_1)\sqcup \cdots \sqcup D_{r_l}(z_l)\in \frB$. For $a=a_1\otimes \cdots \otimes a_l\in F_{\ol{V}}^{\loc}(L)$, 
\begin{align*}
\sigma_{(q, z), (q', z')L}\circ \sigma_{(q', z'), L}(a)
&=
e^{zT}q^{L_0}e^{z'T}q'^{L_0}a_1\otimes \cdots \otimes e^{zT}q^{L_0}e^{z'T}q'^{L_0}a_l\\
&=
e^{zT}e^{qz'T}q^{L_0}q'^{L_0}a_1\otimes \cdots \otimes e^{zT}e^{qz'T}q^{L_0}q'^{L_0}a_l\\
&=
\sigma_{(q, z)(q', z'), L}(a). 
\end{align*}

\item
This is clear by the definition of $\sigma_{(q, z), L}$. 

\item 
For $(q, z)\in S^1\ltimes \bbC$ and $L, M, N\in \frB$ such that $L\sqcup M\subset N$, since $\sigma_{(q, z)}$ is a morphism of precosheaves, 
\begin{align*}
\sigma_{(q, z), N}\circ (F_{\ol{V}}^{\loc})^{L, M}_N
&=
\sigma_{(q, z), N}\circ (F_{\ol{V}}^{\loc})^{L\sqcup M}_N=
(F_{\ol{V}}^{\loc})^{(q, z)(L\sqcup M)}_{(q, z)N}\circ \sigma_{(q, z), L\sqcup M}\\
&=
(F_{\ol{V}}^{\loc})^{(q, z)L, (q, z)M}_{(q, z)N}\circ (\sigma_{(q, z), L}\otimes \sigma_{(q, z), M}). 
\end{align*}

\item 
Clear by the convention $\sigma_{(q, z), \varnothing}=\id_{\bbC}$. 
\qedhere
\end{clist}
\end{proof}

By \cref{lem:FolVequiv} and the construction in \cref{ss:exten}, the factorization algebra $\bfF_{\ol{V}}^{\loc}$ has the $S^1\ltimes \bbC$-equivariant structure. We denote it by $\wt{\sigma}_{(q, z)}\colon \bfF_{\ol{V}}^{\loc}\to (q, z)\bfF_{\ol{V}}^{\loc}$. 

\begin{lem}\label{lem:isoVDbfFD}
For $R\in \bbR_{>0}$ and $\Delta\in \bbZ$, we have $(\wt{F}^{\loc}_{\ol{V}})^{D_R(0)}_{D_R(0)}(V_\Delta)=\bfF^{\loc}_{\ol{V}}(D_R(0))_\Delta$. 
\end{lem}

\begin{proof}
For $a\in \ol{V}$, we have
\[
\wt{\sigma}_{q, \bbC}(\wt{F}^{\loc}_{\ol{V}})^{D_R(0)}_{D_R(0)}(a)=
(\wt{F}^{\loc}_{\ol{V}})^{D_R(0)}_{D_R(0)}\sigma_{q, \bbC}(a) \quad (q\in S^1). 
\]
Thus, the lemma is clear since $V_\Delta=\{a\in\ol{V}\mid \forall q\in S^1,\,\sigma_{q, D_R(0)}(a)=q^\Delta a\}$. 
\end{proof}

Let $V=\bigoplus_{\Delta\in \bbZ}V_\Delta$ (resp. $W=\bigoplus_{\Delta\in \bbZ}W_\Delta$) be a commutative $\bbZ$-graded vertex algebra such that $V_\Delta=0$ (resp. $W_\Delta=0$) for $\Delta\ll0$. For a morphism $f\colon V\to W$ of commutative $\bbZ$-graded vertex algebras, define 
\[
\ol{f}\colon \ol{V}\to \ol{W}, \quad a\mapsto \{f(\pi_\Delta(a))\}_{\Delta\in \bbZ}, 
\]
then $\ol{f}$ is a morphism of commutative $\bbC$-algebras. Thus, we have a morphism $\bfF^{\loc}_{\ol{f}}\colon \bfF^{\loc}_{\ol{V}}\to \bfF^{\loc}_{\ol{W}}$ of factorization algebras. 

\begin{lem}\label{lem:Flocolfequiv}
For a morphism $f\colon V\to W$ of commutative $\bbZ$-graded vertex algebras, the morphism $F^{\loc}_{\ol{f}}\colon F^{\loc}_{\ol{V}}\to F^{\loc}_{\ol{W}}$ is a morphism of $S^1\ltimes \bbC$-equivariant prefactorization algebras. 
\end{lem}

\begin{proof}
Let $(q, z)\in S^1\ltimes \bbC$ and $L=D_{r_1}(z_1)\sqcup\cdots \sqcup D_{r_l}(z_l)\in \frB$. We abuse the notation $\sigma_{(q, z), L}$ for $F^{\loc}_{\ol{V}}$ and $F^{\loc}_{\ol{W}}$. Then, for $a=a_1\otimes \cdots \otimes a_l\in F_{\ol{V}}^{\loc}(L)$, 
\begin{align*}
F^{\loc}_{\ol{f}, (q, z)L}\sigma_{(q, z), L}(a)
&=
\ol{f}(e^{zT}q^{L_0}a_1)\otimes \cdots \otimes \ol{f}(e^{zT}q^{L_0}a_l)\\
&=
e^{zT}q^{L_0}\ol{f}(a_1)\otimes \cdots \otimes e^{zT}q^{L_0}\ol{f}(a_l)\\
&=
\sigma_{(q, z), L}F^{\loc}_{\ol{f}, L}(a). 
\qedhere
\end{align*}
\end{proof}

For a morphism $f\colon V\to W$ of commutative $\bbZ$-graded vertex algebras, by \cref{lem:Flocolfequiv}, the morphism $\bfF^{\loc}_{\ol{f}}\colon \bfF^{\loc}_{\ol{V}}\to \bfF^{\loc}_{\ol{W}}$ is a morphism of $S^1\ltimes \bbC$-equivariant factorization algebras. The map $V\mapsto \bfF^{\loc}_{\ol{V}}$ gives rise to the functor
\[
\bfF^{\loc}_{\ol{(\,\cdot\,)}}\colon \cbgrVA\to \FA^{\mathsf{loc}}_{S^1\ltimes \bbC}(\bbC, \Lin), 
\]
where $\FA^{\mathsf{loc}}_{S^1\ltimes \bbC}(\bbC, \Lin)$ denotes the full subcategory of $\PFA_{S^1\ltimes \bbC}(\bbC, \Lin)$ whose objects are locally constant $S^1\ltimes \bbC$-equivariant factorization algebras. 

\medskip

To make $\bfF^{\loc}_{\ol{V}}$ a holomorphic factorization algebra (\cref{dfn:holoPFA}), we need to introduce a locally convex topology on $\bfF^{\loc}_{\ol{V}}(U)$. 
In order to do this, suppose that $V=\bigoplus_{\Delta\in \bbZ}V_\Delta$ is a commutative $\bbZ$-graded vertex algebra which satisfies (i) $V_\Delta=0$ for $\Delta\ll0$ and (ii) $\dim_{\bbC}V_\Delta<\infty$ for all $\Delta\in \bbZ$. 

For each $\Delta\in \bbZ$, the finite-dimensional linear space $V_\Delta$ has a unique topology which makes $V_\Delta$ a Hausdorff topological vector space. This topology is actually defined by a single norm $\|\cdot\|_\Delta$, and hence is a locally convex topology. Since the topological vector space $V_\Delta$ is complete and Hausdorff, the linear space $\ol{V}=\prod_{\Delta\in \bbZ}V_\Delta$ equipped with the product topology is a complete Hausdorff locally convex space. Note that this locally convex topology on $\ol{V}$ is defined by the set of seminorms 
\[
\Gamma=\{\|\cdot\|_\Delta\circ \pi_\Delta\mid \Delta\in \bbZ\}. 
\]
In particular, $\ol{V}$ is a Fr\'{e}chet space. 

\begin{lem}\label{lem:olVmultcont}
\ 
\begin{enumerate}
\item
The multiplication $\ol{V}\times \ol{V}\to \ol{V}$ is continuous. 

\item 
For $q\in S^1$, the linear map $q^{L_0}\colon \ol{V}\to \ol{V}$ is continuous. 

\item 
For $z\in \bbC$, the linear map $e^{zT}\colon \ol{V}\to \ol{V}$ is continuous. 
\end{enumerate}
\end{lem}

\begin{proof}
(1) Since $\ol{V}$ is a Fr\'{e}chet space, it is enough to show that $\ol{V}\times \ol{V}\to \ol{V}$ is separately continuous. Fix $a\in \ol{V}$, and take a sequence $\{b_n\}_{n\in \bbN}$ in $\ol{V}$ which converges to $0\in\ol{V}$. For $\Delta\in \bbZ$, note that 
\[
\pi_\Delta(ab_n)=
\sum_{\Delta'\in \bbZ}\pi_{\Delta'}(a)\pi_{\Delta-\Delta'}(b_n) 
\]
is a finite sum. 
Also, note that the bilinear map $V_{\Delta_1}\times V_{\Delta_2}\to V_{\Delta_1+\Delta_2}$, $(x, y)\mapsto xy$ is continuous because $\dim_{\bbC}V_\Delta<\infty$. Thus, we have
\[
\lim_{n\to \infty}\pi_\Delta(ab_n)=
\sum_{\Delta'\in \bbZ}\lim_{n\to \infty}\pi_{\Delta'}(a)\pi_{\Delta-\Delta'}(b_n)=
\sum_{\Delta'\in \bbZ}\pi_{\Delta'}(a)\lim_{n\to \infty}\pi_{\Delta-\Delta'}(b_n)=0, 
\]
which shows the claim. 

(2) It is enough to show that $\pi_\Delta\circ q^{L_0}\colon \ol{V}\to V_\Delta$, $a\mapsto q^\Delta\pi_\Delta(a)$ is continuous for each $\Delta\in \bbZ$, but this is obvious. 

(3) It is enough to show that $\pi_\Delta\circ e^{zT}\colon \ol{V}\to V_\Delta$, $a\mapsto \sum_{n\in \bbN}(z^n/n!)T^n\pi_{\Delta-n}(a)$ is continuous for each $\Delta\in \bbZ$. This is clear since the condition $\dim_{\bbC}V_{\Delta-n}<\infty$ yields that $T^n\colon V_{\Delta-n}\to V_\Delta$ is continuous. 
\end{proof}

By \cref{lem:olVmultcont}, the $S^1\ltimes \bbC$-equivariant prefactorization algebra $F_{\ol{V}}^{\loc}\colon \frB\to \Lin$ gives rise to an $S^1\ltimes \bbC$-equivariant prefactorization algebra $F_{\ol{V}}^{\loc}\colon \frB\to \LCS$. Thus, we have an $S^1\ltimes \bbC$-equivarint factorization algebra $\bfF^{\loc}_{\ol{V}}\colon \frU_{\bbC}\to \LCS$. 

\begin{prp}\label{prp:constLHFA}
For a commutative $\bbZ$-graded vertex algebra $V=\bigoplus_{\Delta\in \bbZ}V_\Delta$ such that $V_\Delta=0$ ($\Delta\ll 0$) and $\dim_{\bbC}V_\Delta<\infty$ ($\Delta\in \bbZ$), the $S^1\ltimes \bbC$-equivariant factorization algebra $\bfF^{\loc}_{\ol{V}}\colon \frU_{\bbC}\to \LCS$ is holomorphic. 
\end{prp}

\begin{proof}
We check the conditions (i)-(iv) in \cref{dfn:holoPFA}: 
\begin{clist}
\item 
This is clear since $\bfF^{\loc}_{\ol{V}}$ is locally constant. 

\item 
This follows from \cref{lem:isoVDbfFD}. 

\item
For $R\in \bbR_{>0}$, by the definition, the locally convex space $F^{\loc}_{\ol{V}}(D_R(0))=\ol{V}$ is complete and Hausdorff. Since $(\wt{F}^{\loc}_{\ol{V}})\colon F^{\loc}_{\ol{V}}(D_R(0))\to (\bfF^{\loc}_{\ol{V}})(D_R(0))$ is an isomorphism of locally convex spaces, the condition holds. 

\item
For $a_1, \ldots, a_l\in (\bfF^{\loc}_{\ol{V}})^0$, we have
\[
((\wt{F}^{\loc}_{\ol{V}})^{D_R(0)}_{D_R(0)})^{-1}\mu^R_{z_1, \ldots, z_l}(a_1\otimes \cdots \otimes a_l)=
\prod_{i=1}^le^{z_iT}((\wt{F}^{\loc}_{\ol{V}})^{D_r(0)}_{D_r(0)})^{-1}(\bfF^{\loc}_{\ol{V}})^0_{D_r(0)}(a_i).
\]
Thus, it is enough to show that for $a_1, \ldots, a_l\in V$, the function
\[
\Conf_l(D_R(0))\to \ol{V}, \quad 
(z_1, \ldots, z_l)\mapsto \prod_{i=1}^le^{z_iT}a_i
\]
is holomorphic, or equivalently, 
\[
\Conf_l(D_R(0))\to V_\Delta, \quad 
(z_1, \ldots, z_l)\mapsto \pi_\Delta\bigl(\prod_{i=1}^le^{z_iT}a_i\bigr)
\]
is holomorphic for each $\Delta\in \bbZ$, but this is clear since
\begin{align*}
\pi_\Delta\bigl(\prod_{i=1}^le^{z_iT}a_i\bigr)
&=
\sum_{\substack{\Delta_1, \ldots, \Delta_l\in \bbZ\\ \Delta_1+\cdots+\Delta_l=\Delta}}\prod_{i=1}^l\pi_{\Delta_i}(e^{z_iT}a_i)=
\sum_{\substack{\Delta_1, \ldots, \Delta_l\in \bbZ\\ \Delta_1+\cdots+\Delta_l=\Delta}}\prod_{i=1}^l\Bigl(\sum_{m_i\in\bbN}\frac{z_i^{m_i}}{m_i!}T^{m_i}\pi_{\Delta_i-m_i}(a_i)\Bigr)\\
&=
\sum_{\substack{\Delta_1, \ldots, \Delta_l\in \bbZ\\ \Delta_1+\cdots+\Delta_l=\Delta}}\sum_{m_1, \ldots, m_l\in\bbN}\,\prod_{i=1}^l\Bigl(\frac{z_i^{m_i}}{m_i!}T^{m_i}\pi_{\Delta_i-m_i}(a_i)\Bigr). 
\qedhere
\end{align*}
\end{clist}
\end{proof}

Let $V=\bigoplus_{\Delta\in \bbZ}V_\Delta$ (resp. $W=\bigoplus_{\Delta\in \bbZ}W_\Delta$) be a commutative $\bbZ$-graded vertex algebra such that $V_\Delta=0$ (resp. $W_\Delta=0$) for $\Delta\ll0$ and $\dim_{\bbC}V_\Delta<\infty$ (resp. $\dim_{\bbC}W_\Delta<\infty$) for all $\Delta\in \bbZ$. For a morphism $f\colon V\to W$ of $\bbZ$-graded vertex algebras, since the linear map $\ol{f}\colon \ol{V}\to \ol{W}$ is continuous, the morphism $\bfF^{\loc}_{\ol{f}}\colon \bfF^{\loc}_{\ol{V}}\to \bfF^{\loc}_{\ol{W}}$ is a morphism of holomorphic factorization algebras. 
Thus, the map $V\mapsto \bfF^{\loc}_{\ol{V}}$ gives rise to the functor
\[
\bfF^{\loc}_{\ol{(\,\cdot\,)}}\colon \fcbgrVA\to \LHFA,
\]
where $\LHFA$ denotes the full subcategory of $\HolFA$ consisting of locally constant holomorphic factorization algebras. 

\begin{rmk}
Let $V\in \Ob(\fcbgrVA)$. For $R\in \bbR_{>0}$: 
\begin{enumerate}
\item 
By \cref{lem:isoVDbfFD}, the linear isomorphism $(\wt{F}^{\loc}_{\ol{V}})^{D_R(0)}_{D_R(0)}\colon \ol{V}\to \bfF^{\loc}_{\ol{V}}(D_R(0))$ gives a linear isomorphism $(\wt{F}^{\loc}_{\ol{V}})^{D_R(0)}_{D_R(0)}\colon V\to \bigoplus_{\Delta\in \bbZ}\bfF^{\loc}_{\ol{V}}(D_R(0))_\Delta$. 

\item 
By\cref{rmk:clF0DRiso}, the linear map $(\bfF^{\loc}_{\ol{V}})^0_{D_R(0)}\colon (\bfF^{\loc}_{\ol{V}})^0\to \bfF^{\loc}_{\ol{V}}(D_R(0))$ gives a linear isomorphism $(\bfF^{\loc}_{\ol{V}})^0_{D_R(0)}\colon \bfV\bfF^{\loc}_{\ol{V}}\to \bigoplus_{\Delta\in \bbZ}\bfF^{\loc}_{\ol{V}}(D_R(0))_\Delta$. 
\end{enumerate}
\end{rmk}

\begin{dfn}
For $V\in \Ob(\fcbgrVA)$, define a linear isomorphism $\eta_V$ as the composition
\[
\begin{tikzcd}[column sep=huge]
V \arrow[r, "(\wt{F}^{\loc}_{\ol{V}})^{D_R(0)}_{D_R(0)}"] & [1cm]
\displaystyle\bigoplus_{\Delta\in \bbZ}\bfF^{\loc}_{\ol{V}}(D_R(0))_\Delta \arrow[r, "((\bfF^{\loc}_{\ol{V}})^0_{D_R(0)})^{-1}"] & [1cm]
\bfV\bfF^{\loc}_{\ol{V}}
\end{tikzcd}
\]
Here we take $R\in \bbR_{>0}$, but the definition of $\eta_V$ is independent of the choice of $R$. 
\end{dfn}

\begin{prp}\label{prp:VbfVisom}
The linear isomorphism 
\[
\eta_V\colon V\to \bfV\bfF^{\loc}_{\ol{V}}
\]
is an isomorphism in $\fcbgrVA$ which is natural in $V\in \Ob(\fcbgrVA)$. 
\end{prp}

\begin{proof}
The naturality of $\eta_V$ is easy to see. It is enough to show that $\eta_V$ is a morphism of commutative vertex algebras. 
\begin{itemize}
\item
$\eta_V$ is a morphism of commutative algebras: For $a, b\in V$, we have
\begin{align*}
\eta_V(a)\cdot \eta_V(b)
&=
\bigl\{(\bfF^{\loc}_{\ol{V}})^0_{D_R(0)}\eta_V(a)\cdot (\bfF^{\loc}_{\ol{V}})^0_{D_R(0)}f(b)\bigr\}_{R\in \bbR_{>0}}\\
&=
\bigl\{(\wt{F}^{\loc}_{\ol{V}})^{D_R(0)}_{D_R(0)}(a)\cdot (\wt{F}^{\loc}_{\ol{V}})^{D_R(0)}_{D_R(0)}(b)\bigr\}_{R\in \bbR_{>0}}\\
&=
\bigl\{(\wt{F}^{\loc}_{\ol{V}})^{D_R(0)}_{D_R(0)}(ab)\bigr\}_{R\in \bbR_{>0}}=
\eta_V(ab), 
\end{align*}
where we used \cref{rmk:bfVcalgstr} in the first equality, and \cref{lem:unitisom} in the third equality. Also, since the vacuum $\vac$ of $\bfV\bfF^{\loc}_{\ol{V}}$ is defined by $\vac=\bigl\{(\bfF^{\loc}_{\ol{V}})^{\varnothing}_{D_R(0)}(1_{\bfF^{\loc}_{\ol{V}}})\bigr\}_{R\in \bbR_{>0}}$, we have
\begin{align*}
\vac
&=
\bigl\{(\wt{F}^{\loc}_{\ol{V}})^{\varnothing}_{D_R(0)}(1_{\bbC})\bigr\}_{R\in \bbR_{>0}}=
\bigl\{(\wt{F}^{\loc}_{\ol{V}})^{D_R(0)}_{D_R(0)}(F^{\loc}_{\ol{V}})^{\varnothing}_{D_R(0)}(1_{\bbC})\bigr\}_{R\in \bbR_{>0}}\\
&=
\bigl\{(\wt{F}^{\loc}_{\ol{V}})^{D_R(0)}_{D_R(0)}(\vac_V)\bigr\}=
f(\vac_V). 
\end{align*}

\item 
$\eta_V$ commutes with the translation operators: Recall the definition of the translation operator $T$ of $\bfV\bfF^{\loc}_{\ol{V}}$. Then, for $a\in V$, we have
\[
Tf(a)=
\{(Tf(a))_R\}_{R\in \bbR_{>0}}=
\{\pdd_z\mu^R_z(f(a))|_{z=0}\}_{R\in \bbR_{>0}}. 
\]
Here we used the notation $\mu_z^R$ defined in \cref{dfn:muzR} for $\bfF^{\loc}_{\ol{V}}$. Denote by 
\[
\sigma_{(q, z)}\colon F^{\loc}_{\ol{V}}\to (q, z)F^{\loc}_{\ol{V}}, \quad
\wt{\sigma}_{(q, z)}\colon \bfF^{\loc}_{\ol{V}}\to (q, z)\bfF^{\loc}_{\ol{V}}
\]
the $S^1\ltimes \bbC$-equivariant structures. For $R\in \bbR_{>0}$ and $z\in D_R(0)$, take $r\in \bbR_{>0}$ such that $|z|<R-r$, then
\begin{align*}
\mu_z^R(f(a))
&=
(\bfF^{\loc}_{\ol{V}})^{D_r(z)}_{D_R(0)}\wt{\sigma}_{(1, z), D_r(0)}(\bfF^{\loc}_{\ol{V}})^0_{D_r(0)}(f(a))=
(\bfF^{\loc}_{\ol{V}})^{D_r(z)}_{D_R(0)}\wt{\sigma}_{(1, z), D_r(0)}(\wt{F}^{\loc}_{\ol{V}})^{D_r(0)}_{D_r(0)}(a)\\
&=
(\bfF^{\loc}_{\ol{V}})^{D_r(z)}_{D_R(0)}(\wt{F}^{\loc}_{\ol{V}})^{D_r(z)}_{D_r(z)}\sigma_{(1, z), D_r(0)}(a)=
(\wt{F}^{\loc}_{\ol{V}})^{D_r(z)}_{D_R(0)}e^{zT_Va}=
(\wt{F}^{\loc}_{\ol{V}})^{D_R(0)}_{D_R(0)}e^{zT_Va}\\
&=
\sum_{n=0}^{\infty}\frac{z^n}{n!}(\wt{F}^{\loc}_{\ol{V}})^{D_R(0)}_{D_R(0)}(T^n_Va).
\end{align*}
Thus, we get
\[
\pdd_z\mu_z^R(a)|_{z=0}=
(\wt{F}^{\loc}_{\ol{V}})^{D_R(0)}_{D_R(0)}T_Va, 
\]
and hence
\[
Tf(a)=\bigl\{(\wt{F}^{\loc}_{\ol{V}})^{D_R(0)}_{D_R(0)}T_Va\bigr\}=f(T_Va). 
\qedhere
\]
\end{itemize}
\end{proof}

\begin{prp}\label{prp:bfFfullfaith}
The functor $\bfF^{\loc}_{\ol{(\,\cdot\,)}}\colon \fcbgrVA\to \LHFA$ is fully faithfull. 
\end{prp}

\begin{proof}
The faithfulness follows from \cref{prp:VbfVisom}. We prove that $\bfF^{\loc}_{\ol{(\,\cdot\,)}}$ is full. 

Let $V, W\in \Ob(\fcbgrVA)$, and $\varphi\colon \bfF^{\loc}_{\ol{V}}\to \bfF^{\loc}_{\ol{W}}$ be a morphism in $\LHFA$. Define a morphism $f\colon V\to W$ in $\fcbgrVA$ by $f\ceq \eta_W^{-1}\circ \bfV(\varphi)\circ \eta_V$. To prove that $\bfF^{\loc}_{\ol{f}}=\varphi$, it is enough to show that $\bfF^{\loc}_{\ol{f}, D_R(0)}=\varphi_{D_R(0)}$ for all $R\in \bbR_{>0}$ since $\bfF^{\loc}_{\ol{V}}$ and $\bfF^{\loc}_{\ol{W}}$ are $S^1\ltimes \bbC$-equivariant factorization algebras. By the definition of $f$, we have
\[
\bfF^{\loc}_{\ol{f}, D_R(0)}\circ (\wt{F}^{\loc}_{\ol{V}})^{D_R(0)}_{D_R(0)}|_V=
(\wt{F}^{\loc}_{\ol{W}})^{D_R(0)}_{D_R(0)}\circ f=
\varphi_{D_R(0)}\circ (\wt{F}^{\loc}_{\ol{V}})^{D_R(0)}_{D_R(0)}|_V. 
\]
Since $V\subset \ol{V}$ is dence subspace, we get 
\[
\bfF^{\loc}_{\ol{f}, D_R(0)}\circ (\wt{F}^{\loc}_{\ol{V}})^{D_R(0)}_{D_R(0)}=\varphi_{D_R(0)}\circ (\wt{F}^{\loc}_{\ol{V}})^{D_R(0)}_{D_R(0)},
\]
or equivalently $\bfF^{\loc}_{\ol{f}, D_R(0)}=\varphi_{D_R(0)}$. 
\end{proof}

\begin{prp}\label{prp:bfFjet}
Let $A$ be a finitely generated commutative algebra. For $V\in \Ob(\fcbgrVA)$ and a morphism $\varphi\colon \bfF^{\loc}_A\to \bfF^{\loc}_V$ in $\LFA$, there exists a unique morphism $\wt{\varphi}\colon \bfF^{\loc}_{\ol{\clJ\!A}}\to \bfF^{\loc}_{\ol{V}}$ in $\LHFA$ which commutes
\[
\begin{tikzcd}[row sep=huge, column sep=huge]
\bfF^{\loc}_A \arrow[r] \arrow[d, "\varphi"'] & 
\bfF^{\loc}_{\ol{\clJ\!A}} \arrow[d, "\wt{\varphi}"]\\
\bfF^{\loc}_{V} \arrow[r] & 
\bfF^{\loc}_{\ol{V}}
\end{tikzcd}
\]
Here horizontal arrows denote the morphism in $\LFA$ induced from the canonical injections $A\inj \clJ\!A\inj \ol{\clJ\!A}$ and $V\inj\ol{V}$. 
\end{prp}

\begin{proof}
Since $\bfF^{\loc}\colon \CAlg\to \LFA$ is a categorical equivalence (\cref{prp:catequiv}), there is a unique morphism $f_0\colon A\to V$ in $\CAlg$ such that $\varphi=\bfF^{\loc}_{f_0}$. Also, by the universality of jet algebra, there is a unique morphism $f\colon\clJ\!A\to V$ in $\fcbgrVA$ such that $f|_A=f_0$. The morphism $\wt{\varphi}\ceq \bfF^{\loc}_{\ol{f}}\colon \bfF^{\loc}_{\ol{\clJ\!A}}\to \bfF^{\loc}_{\ol{V}}$ in $\LHFA$ commutes the diagram in the proposition. To prove the uniqueness, let $\psi\colon \bfF^{\loc}_{\ol{\clJ A}}\to \bfF^{\loc}_{\ol{V}}$ be a morphism in $\LHFA$ which commutes the diagram. Then, by the fullness of $\bfF^{\loc}_{\ol{(\,\cdot\,)}}$, there exists a morphism $g\colon \clJ\!A\to V$ in $\fcbgrVA$ such that $\psi=\bfF^{\loc}_{\ol{g}}$. Also, the faithfulness of $\bfF^{\loc}_{\ol{(\,\cdot\,)}}$ and uniqueness of $f$ imply $g=f$. Thus, we get $\psi=\wt{\varphi}$. 
\end{proof}

\appendix

\section{Calculus of functions with values in a locally convex space}\label{s:App}

In this appendix we develop the calculus of functions with values in a locally convex space to ensure the completeness of this note and to make content more accessible to algebraists. For a locally convex space $E$, we use the following notations: 

\begin{itemize}
\item
The set $\Gamma$ of semi-norms on $E$ that defines the topology of $E$ is called a fundamental system of semi-norms for $E$. 

\item 
For a semi-norm $p\colon E\to \bbR$ and $t\in \bbR_{>0}$, we denote by $B(p, t)\ceq\{x\in E\mid p(x)<t\}$ the open ball. 

\item 
For a subset $A\subset E$, we denote by $\conv A$ the convex hull. 
\end{itemize}

\subsection{Holomorphic functions}\label{ss:holofunc}

In this subsection we define holomorphic functions with values in a locally convex space. Fix a Hausdorff locally convex space $E$ over $\bbC$. 

Let $U\subset \bbC^n$ be an open subset. For $a\in U$, the set $U\bs\{a\}$ is 
directed by letting
\[
z\preceq w \rarr
|z-a|\ge |w-a| \quad (z, w\in U\bs\{a\}). 
\]

\begin{lem}\label{lem:netUbsa}
Take a fundamental system $\Gamma$ of semi-norms for $E$. Let $U\subset \bbC^n$ be an open subset and $a\in U$. For a net $\{\xi_z\}_{z\in U\bs\{a\}}\subset E$ and $\xi\in E$, the following (i) and (ii) are equivalent: 
\begin{clist}
\item 
The net $\{\xi_z\}_{z\in U\bs\{a\}}$ converges to $\xi$. 

\item 
For $p\in \Gamma$ and $\ve\in \bbR_{>0}$, there is $\delta\in \bbR_{>0}$ such that 
\[
z\in U\bs\{a\},\, |z-a|<\delta\, \Longrightarrow\, p(\xi_z-\xi)<\ve. 
\]
\end{clist}
\end{lem}

\begin{proof}
(i) $\Longrightarrow$ (ii) is clear since $B(p, \ve)\subset E$ is a neighborhood of $0\in E$. Let us prove the converse. Since $\{B(p_1, \ve_1)\cap \cdots \cap B(p_l, \ve_l)\mid l\in \bbZ_{>0}, p_i\in \Gamma, \ve_i\in \bbR_{>0}\}$ is a neighborhood basis of $0\in E$, it is enough to show that for $p\in \Gamma$ and $\ve\in \bbR_{>0}$, there exists $z_0\in U\bs\{a\}$ such that 
\[
z\in U\bs\{a\},\, z\succeq z_0\, \Longrightarrow\, 
\xi_z-\xi\in B(p, \ve). 
\]
By the condition (ii), there is some $\delta\in \bbR_{>0}$ satisfying $z\in U\bs\{a\}$, $|z-a|<\delta$ $\Longrightarrow$ $p(\xi_z-\xi)<\ve$. Also, since $U\subset \bbC$ is open, there is $\delta'\in \bbR_{>0}$ satisfying $D_{\delta'}(a)\subset U$. Take $z_0\in (D_\delta(a)\cap D_{\delta'}(a))\bs\{a\}$, then $z_0\in U\bs\{a\}$, and $z_0$ satisfies that $z\in U\bs\{a\}$, $z\succeq z_0 $ $\Longrightarrow$ $\xi_z-\xi\in B(p, \ve)$. 
\end{proof}

\begin{rmk}
Let $U\subset \bbC^n$ be an open subset. By \cref{lem:netUbsa}, a function $f\colon U\to E$ is continuous at $a\in U$ if and only if the net $\{f(z)\}_{z\in U\bs\{a\}}$ converges to $f(a)$. 
\end{rmk}

\begin{dfn}\label{dfn:holofunc}
Let $U\subset \bbC^n$ be an open subset, and $f\colon U\to E$ a function. 
\begin{enumerate}
\item 
For $a\in U$, we say that $f$ is complex differentiable at $a$, if there exists $\xi\in E^n$ such that the net 
\[
U\bs\{a\}\to E, \quad
z\mapsto \frac{f(z)-f(a)-(z-a)\cdot{}^t\xi}{|z-a|}
\] 
converges to $0\in E$. Here we used the matrix notation 
\[
z\cdot{}^t\xi=\sum_{i=1}^nz_i\xi_i\in E
\]
for $z=(z_1, \ldots, z_n)\in \bbC^n$ and $\xi=(\xi_1, \ldots, \xi_n)\in E^n$. 

\item 
We say that $f$ is holomorphic on $U$ if $f$ is complex differentiable for all $z\in U$. 
\end{enumerate}
\end{dfn}

\begin{lem}
Take a fundamental system $\Gamma$ of semi-norms for $E$, and let $U\subset \bbC^n$ be an open subset. 
A function $f\colon U\to E$ is complex differentiable at $a\in U$ if and only if there exists $\xi \in E^n$ satisfying the following condition: 
For $p\in \Gamma$ and $\ve\in \bbR_{>0}$, there exists $\delta\in \bbR_{>0}$ such that 
\[
z\in U\bs\{a\},\, |z-a|<\delta\, \Longrightarrow\, p\Bigl(\frac{f(z)-f(a)-(z-a)\cdot{}^t\xi}{|z-a|}\Bigr)<\ve.
\]
\end{lem}

\begin{proof}
This follows from \cref{lem:netUbsa}. 
\end{proof}

\begin{lem}
Let $U\subset \bbC^n$ be an open subset. If a function $f\colon U\to E$ is complex differentiable at $a\in U$, then an element $\xi\in E^n$ satisfying the condition in \cref{dfn:holofunc} is unique. 
\end{lem}

\begin{proof}
Let $\xi, \xi'\in E^n$ are elements satisfying the condition in \cref{dfn:holofunc}, and take a fundamental system $\Gamma$ of semi-norms for $E$. It is enough to show that $p(\xi_i-\xi'_i)=0$ for all $p\in \Gamma$ and $i\in[n]$. For $\ve\in \bbR_{>0}$, there exists $\delta\in \bbR_{>0}$ such that $D_\delta(a)\subset U$ and 
\[
z\in \bbC^n,\, 0<|z-a|<\delta\, \Longrightarrow\, 
p\Bigl(\frac{f(z)-f(a)-(z-a)\cdot{}^t\xi}{|z-a|}\Bigr),\, p\Bigl(\frac{f(z)-f(a)-(z-a)\cdot{}^t\xi'}{|z-a|}\Bigr)<\frac{\ve}{2}. 
\]
Set $z=(a_1, \ldots, a_i+\delta/2, \ldots, a_n)$, then we get
\[
p(f(z)-f(a)-\xi_i),\, p(f(z)-f(a)-\xi'_i)<\ve/2. 
\]
Thus, we have $p(\xi_i-\xi_i')<\ve$ for any $\ve\in \bbR_{>0}$, which shows the claim. 
\end{proof}

Let $U\subset \bbC^n$ be an open subset. If a function $f\colon U\to E$ is complex differentiable at $a\in U$, the element 
\[
f'(a)\ceq {}^t\xi, 
\]
where $\xi\in E^n$ is the element satisfying the condition in \cref{dfn:holofunc} (1), is called the derivative of $f$ at $a$. 

For a semi-norm $p$, define 
\[
p_n\colon E^n\to \bbR, \quad \xi\mapsto \Bigl(\sum_{i=1}^np(\xi_i)^2\Bigr)^{1/2}, 
\]
then $p_n$ is a semi-norm on $E^n$. Note that this semi-norm satisfies
\[
p(z\cdot{}^t\xi)\le |z|p_n(\xi)
\]
for all $z\in \bbC^n$ and $\xi\in E^n$. 

\begin{lem}
Let $U\subset \bbC^n$ be an open subset, and $f\colon U\to E$ a function. If $f$ is complex differentiable at $a\in U$, then $f$ is continuous at $a$. 
\end{lem}

\begin{proof}
It is enough to show that for $p\in \Gamma$ and $\ve\in \bbR_{>0}$, there exists $\delta\in \bbR_{>0}$ satisfying  
\[
z\in U\bs\{a\},\, |z-a|<\delta\, \Longrightarrow\, p(f(z)-f(a))<\ve.
\]
Since $f$ is complex differentiable at $a$, there is some $\delta_0\in \bbR_{>0}$ such that
\[
z\in U\bs\{a\},\, |z-a|<\delta_0\, \Longrightarrow\, 
p\Bigl(\frac{f(z)-f(a)-(z-a)\cdot f'(a)}{|z-a|}\Bigr)<1. 
\]
Hence, if $z\in U\bs\{a\}$ satisfies $|z-a|<\delta_0$, we have
\[
p(f(z)-f(a))<p((z-a)\cdot f'(a))+|z-a|\le
(p_n(f'(a))+1)|z-a|. 
\]
Thus $\delta\ceq \min\{\delta_0, \ve/(p_n(f'(a))+1)\}\in \bbR_{>0}$ satisfies the claim. 
\end{proof}

\begin{dfn}
Let $U\subset \bbC^n$ be an open subset, and $f\colon U\to E$ a function. For $i\in[n]$ and $a=(a_1, \ldots, a_n)\in U$, we say that $f$ is complex partial differentiable at $a$ with respect to the $i$-coordinate if the function 
\[
U_i\mapsto E, \quad
z\mapsto f(a_1, \ldots, a_{i-1}, z, a_{i+1}, \ldots, a_n)
\]
is complex differentiable at $a_i$, where 
\[
U_i\ceq\{z\in \bbC\mid (a_1, \ldots, a_{i-1}, z, a_{i+1}, \ldots a_n)\in U\}. 
\]
In this case, we denote the derivative of this function at $a_i$ by
\[
\pdd_if(a)=\pdd_{z_i}f(z)|_{z=a}\ceq
\lim_{z\in U_i\bs\{a_i\}}\frac{f(a_1, \ldots, a_{i-1}, z, a_{i+1}, \ldots, a_n)-f(a_1, \ldots, a_n)}{z-a_i}\in E. 
\]
\end{dfn}

\begin{lem}
Let $U\subset \bbC^n$ be an open subset, and $f\colon U\to E$ a function. If $f$ is complex differentiable at $a\in U$, then $f$ is complex partial differentiable at $a$ with respect to all the $i$-coordinates. Moreover, we have
\[
f'(a)={}^t(\pdd_1f(a), \ldots, \pdd_nf(a)). 
\]
\end{lem}

\begin{proof}
Take a fundamental system $\Gamma$ of semi-norms for $E$. For $p\in \Gamma$ and $\ve\in \bbR_{>0}$, since $f$ is complex differentiable at $a$, there exists $\delta\in \bbR_{>0}$ such that
\[
z\in U\bs\{a\},\, |z-a|<\delta\, \Longrightarrow\, 
p\Bigl(\frac{f(z)-f(a)-(z-a)\cdot f'(a)}{|z-a|}\Bigr)<\ve
\]
Now, let us denote $f'(a)={}^t(\xi_1, \ldots, \xi_n)$. For $i\in[n]$, set $U_i\ceq \{z\in \bbC\mid (a_1, \ldots, a_{i-1}, z, a_{i+1}, \ldots, a_n)\in U\}$, then 
\[
z\in U_i\bs\{a_i\},\, |z-a_i|<\delta\, \Longrightarrow\, 
p\Bigl(\frac{f(a_1, \ldots, a_{i-1}, z, a_{i+1}, \ldots, a_n)-f(a_1, \ldots, a_n)}{z-a_i}-\xi_i\Bigr)<\ve, 
\]
which shows $\pdd_if(a)=\xi_i$. 
\end{proof}

\begin{lem}
Let $U\subset \bbC^n$ be an open subset. If functions  $\varphi\colon U\to \bbC$ and $f\colon U\to E$ are complex differentiable at $a\in U$, then the function
\[
\varphi f\colon U\to E, \quad z\mapsto \varphi(z)f(z)
\]
is also complex differentiable at $a$. Moreover, we have
\[
(\varphi f)(a)=\varphi'(a)f(a)+\varphi(a)f'(a), 
\]
or equivalently, 
\[
\pdd_i(\varphi f)(a)=
(\pdd_i\varphi(a))f(a)+\varphi(a)(\pdd_i f(a)). 
\]
\end{lem}

\begin{proof}
For $z\in U\bs\{a\}$, 
\begin{align*}
&\frac{(\varphi f)(z)-(\varphi f)(a)-(z-a)\cdot (\varphi'(a)f(a)+\varphi(a)f'(a))}{|z-a|}\\
&=
\frac{(\varphi(z)-\varphi(a)-(z-a)\cdot \varphi'(a))f(z)}{|z-a|}+
\frac{\varphi(a)(f(z)-f(a)-(z-a)\cdot f'(a))}{|z-a|}+
\frac{(z-a)\cdot \varphi'(a)(f(z)-f(a))}{|z-a|}.
\end{align*}
Denote by $F_1(z)$ (resp. $F_2(z)$, $F_3(z)$) the first (resp. second, third) term on the right hand side of this equation. It is enough to show that $\lim_{z\in U\bs\{a\}}F_1(z)=\lim_{z\in U\bs\{a\}}F_2(z)=\lim_{z\in U\bs\{a\}}F_3(z)=0$. Take a fundamental system $\Gamma$ of semi-norms for $E$. 

The equation $\lim_{z\in U\bs\{a\}}F_2(z)=0$ is clear by the differentiability of $f$. Also $\lim_{z\in U\bs\{a\}}F_3(z)=0$ follows from the continuity of $f$. In fact, for $p\in \Gamma$, we have
\[
p(F_3(z))\le p_n\bigl(\varphi'(a)(f(z)-f(a))\bigr)\le
\Bigl(\sum_{i=1}^n|\pdd_i\varphi(a)|^2p(f(z)-f(a))^2\Bigr)^{1/2}. 
\]

We prove $\lim_{z\in U\bs\{a\}}F_1(z)=0$. Let $p\in \Gamma$ and $\ve\in \bbR_{>0}$. Since $\varphi$ is differentiable at $a$, there exists some $\delta_1\in \bbR_{>0}$ such that 
\[
z\in U\bs\{a\},\, |z-a|<\delta_1\, \Longrightarrow\, 
\frac{|\varphi(z)-\varphi(a)-(z-a)\cdot \varphi'(a)|}{|z-a|}<\frac{\ve}{1+p(f(a))}. 
\]
Also, since $f$ is continuous at $a$, there exists $\delta_2\in \bbR_{>0}$ such that 
\[
z\in U,\, |z-a|<\delta_2\, \Longrightarrow \, 
p(f(z))<1+p(f(a)). 
\]
Set $\delta\ceq \{\delta_1, \delta_2\}\in \bbR_{>0}$, then for $z\in U\bs\{a\}$ satisfying $|z-a|<\delta$, we have
\[
p(F_1(z))=
\frac{|\varphi(z)-\varphi(a)-(z-a)\cdot \varphi'(a)|}{|z-a|}p(f(z))<\ve, 
\]
which shows $\lim_{z\in U\bs\{a\}}F_1(z)=0$. 
\end{proof}

\begin{lem}\label{lem:compder}
Let $U\subset \bbC^m$ and $V\subset \bbC^n$ be open subsets. For functions $\varphi\colon U\to V$ and $f\colon V\to E$, if $\varphi$ is complex differentiable at $a\in U$, and $f$ is complex differentiable at $\varphi(a)\in V$, then the function $\varphi\circ f\colon U\to E$ is complex differentiable at $a$. Moreover, we have
\[
(f\circ \varphi)'(a)=\varphi'(a)f'(\varphi(a)), 
\]
or equivalently, 
\[
\pdd_i(f\circ \varphi)(a)=
\sum_{j=1}^n\pdd_i\varphi_j(a)\cdot \pdd_jf(\varphi(a)) \quad (i\in[n]). 
\]
Here we denote $\varphi=(\varphi_1, \ldots, \varphi_n)$. 
\end{lem}

\begin{proof}
For $z\in U\bs\{a\}$, 
\begin{align*}
&\frac{(f\circ \varphi)(z)-(f\circ \varphi)(a)-(z-a)\cdot \varphi'(a)f'(\varphi(a))}{|z-a|}\\
&=
\frac{f(\varphi(z))-f(\varphi(a))-(\varphi(z)-\varphi(a))\cdot f'(\varphi(a))}{|z-a|}+
\frac{(\varphi(z)-\varphi(a)-(z-a)\cdot \varphi'(a))\cdot f'(\varphi(a))}{|z-a|}.
\end{align*}
Denote by $F_1(z)$ (resp. $F_2(z)$) the first (resp. second) term on the right hand side of this equation. It is enough to show that $\lim_{z\in U\bs\{a\}}F_1(z)=\lim_{z\in U\bs\{a\}}F_2(z)=0$. Take a fundamental system $\Gamma$ of semi-norms for $E$. 

The equation $\lim_{z\in U\bs\{a\}}F_2(z)=0$ follows from differentiability of $\varphi$ at $a$. In fact, for $p\in \Gamma$, we have
\[
p(F_2(z))\le \frac{|\varphi(z)-\varphi(a)-(z-a)\cdot \varphi'(a)|}{|z-a|}p_n(f'(\varphi(a))). 
\]

We prove $\lim_{z\in U\bs\{a\}}F_1(z)=0$. Define a function $\wt{F}_1$ by
\[
\wt{F}_1\colon U\to E, \quad
z\mapsto 
\begin{cases}
\displaystyle\frac{f(\varphi(z))-f(\varphi(a))-(\varphi(z)-\varphi(a))\cdot f'(\varphi(a))}{|\varphi(z)-\varphi(a)|} & (\varphi(z)\neq \varphi(a)) \\
0 & (\varphi(z)=\varphi(a)), 
\end{cases}
\]
then we have
\[
F_1(z)=\frac{|\varphi(z)-\varphi(a)|}{|z-a|}\wt{F}_1(z) \quad (z\in U\bs\{a\}). 
\]
Since $\varphi$ is differentiable at $a$, there is $\delta_1\in \bbR_{>0}$ satisfying
\[
z\in U\{a\},\, |z-a|<\delta_1\, \Longrightarrow 
\frac{|\varphi(z)-\varphi(a)|}{|z-a|}<1+|\varphi'(a)|. 
\]
Here $|\varphi'(a)|$ denotes the matrix norm defined by
\[
|A|\ceq \Bigl(\sum_{i=1}^m\sum_{j=1}^n|a_{i, j}|^2\Bigr)^{1/2} \quad (A=\{a_{i, j}\}_{i=1, j=1}^{m, n}). 
\]
Now, let $p\in \Gamma$ and $\ve\in \bbR_{>0}$. Since $f$ is differentiable at $\varphi(a)$, there exists $\delta_2\in \bbR_{>0}$ such that 
\[
w\in V\bs\{\varphi(a)\},\, |w-\varphi(a)|<\delta_2\, \Longrightarrow\, 
p\Bigl(\frac{f(w)-f(\varphi(a))-(w-\varphi(a))\cdot f'(\varphi(a))}{|w-\varphi(a)|}\Bigr)<\frac{\ve}{1+|\varphi'(a)|}. 
\]
Also, since $\varphi\colon U\to V$ is continuous at $a$, there exists $\delta_2'\in \bbR_{>0}$ such that 
\[
z\in U,\, |z-a|<\delta_2'\, \Longrightarrow\, 
|\varphi(z)-\varphi(a)|<\delta_2.
\]
Set $\delta\ceq \min\{\delta_1, \delta_2'\}\in \bbR_{>0}$, then for $z\in U\bs\{a\}$ satisfying $|z-a|<\delta$, we have
\[
p(F_1(z))\le(1+|\varphi(a)|)p(\wt{F}_1(z))<\ve, 
\]
which shows $\lim_{z\in U\bs\{a\}}F_1(z)=0$. 
\end{proof}

\subsection{Uniform convergence}
\label{ss:unifconv}

Let $X$ be a set, and $E$ a locally convex space over $\bbC$. 
For a semi-norm $p\colon E\to \bbR$ and a map $f\colon X\to E$, define
\[
p_X(f)\ceq \sup\{p(f(x))\mid x\in X\}\in [0, \infty]. 
\]

\begin{dfn}\label{dfn:unifconv}
For a net $\{f_\lambda\}_{\lambda\in\Lambda}$ consisting of maps $f_\lambda\colon X\to E$ and a map $f\colon X\to E$, we say that $\{f_\lambda\}_{\lambda\in \Lambda}$ converges uniformly to $f$ if the following condition holds: 
For any $\ve\in \bbR_{>0}$ and $p\in \Gamma$, where $\Gamma$ is a fundamental system of semi-norms for $E$, there exists $\lambda_0\in \Lambda$ such that 
\[
\lambda\in \Lambda,\ \lambda\ge \lambda_0\ \Longrightarrow\ 
q_X(f_\lambda-f)<\ve. 
\]
Note that this definition of uniform convergence is independent of the choice of $\Gamma$. 
\end{dfn}

%We focus on the uniform convergence of a sequence $\{f_n\}_{n\in \bbN}$ of maps $f_n\colon X\to E$. 

\begin{lem}\label{lem:uniCauchy}
Let $X$ be a set, and $E$ a quasi-complete locally convex space. 
For a sequence $\{f_n\}_{n\in \bbN}$ of maps $f_n\colon X\to E$, the following conditions are equivalent: 
\begin{clist}
\item 
The sequence $\{f_n\}_{n\in \bbN}$ converges uniformly. 

\item 
For any $p\in \Gamma$ and $\ve\in \bbR_{>0}$, where $\Gamma$ is a fundamental system of semi-norms for $E$, there exists $n_0\in \bbN$ such that
\[
m, n\in \bbN,\, m, n\ge n_0\, \Longrightarrow\, p_X(f_m-f_n)<\ve. 
\]
\end{clist}
\end{lem}

\begin{proof}
(i) $\Longrightarrow$ (ii) is easy to see. We prove the converse. For $x\in X$, the condition (ii) implies that $\{f_n(x)\}_{n\in \bbN}$ is a bounded Cauchy sequence in $E$. Thus, by the quasi-completeness of $E$, there exists a limit point of $\{f_n(x)\}_{n\in \bbN}$. For $x\in X$, choose a limit point $f(x)$ of $\{f_n(x)\}_{n\in \bbN}$, then we have a map $f\colon X\to E$. We claim that $\{f_n\}_{n\in \bbN}$ converges uniformly to $f$. 
Let $p\in \Gamma$ and $\ve\in \bbR_{>0}$. By the condition (ii), there exists $n_0\in \bbN$ such that
\[
m, n\in \bbN,\, m, n\ge n_0\ \Longrightarrow\, 
p(f_m(x)-f_n(x))<\ve/2 \quad (x\in X). 
\]
Since $q\colon E\to \bbR$ is continuous, we have by letting $\mu\to \infty$
\[
m\in \bbN,\, m\ge n_0\ \Longrightarrow\  
p(f_m(x)-f(x))\le \ve/2 \quad (x\in X), 
\]
which yields that 
\[
m\in \bbN,\, m\ge n_0\ \Longrightarrow\  
p_X(f_m-f)\le \ve/2<\ve. 
\qedhere
\]
\end{proof}

\begin{prp}\label{prp:Mtest}
Let $X$ be a set, $E$ a quasi-complete locally convex space, and $\{f_n\}_{n\in \bbN}$ a sequence of maps $f_n\colon X\to E$. For each $p\in \Gamma$, where $\Gamma$ is a fundamental system of semi-norms for $E$, if there exists a sequence $\{M_p(n)\}_{n\in \bbN}$ of non-negative real numbers satisfying the following conditions (i) and (ii), then $\sum_{n\in \bbN}f_n$ converges uniformly. 
\begin{clist}
\item 
For any $n\in \bbN$, we have $p_X(f_n)\le M_p(n)$. 

\item 
The series $\sum_{n\in \bbN}M_p(n)$ converges. 
\end{clist}
\end{prp}

\begin{proof}
For $n\in \bbN$, set $s_n\ceq \sum_{k=0}^nf_k$. It is enough to show that $\{s_n\}_{n\in \bbN}$ satisfies the condition (ii) in \cref{lem:uniCauchy}. Let $p\in \Gamma$ and $\ve\in \bbR_{>0}$. Since $\{t_n\ceq \sum_{k=0}^nM_p(k)\}_{n\in \bbN}$ is a Cauchy sequence, there exists $n_0\in \bbN$ such that 
\[
m, n\in \bbN,\ m, n\ge n_0\ \Longrightarrow\ 
|t_m-t_n|<\ve. 
\]
Thus, if $m, n\in \bbN$, $n>m\ge n_0$, then
\[
p_X(s_m-s_n)\le \sum_{k=m+1}^np_X(f_k)\le 
\sum_{k=m+1}^nt_p(k)=
|t_m-t_n|<\ve. 
\qedhere
\]
\end{proof}

Now, fix a topological space $X$ and a locally convex space $E$. 

\begin{dfn}\label{dfn:locuniconv}
For a net $\{f_\lambda\}_{\lambda\in \Lambda}$ consisting of maps $f_\lambda\colon X\to E$ and a map $f\colon X\to E$, we say that $\{f_\lambda\}_{\lambda\in \Lambda}$ converges locally uniformly to $f$ if every point $x\in X$ has a neighborhood $N\subset X$ such that $\{f_\lambda|_N\}_{\lambda\in \Lambda}$ converges uniformly to $f|_N$. 
\end{dfn}

\begin{prp}\label{prp:unifconvcont}
Let $\{f_\lambda\}_{\lambda\in \Lambda}$ be a net consisting of continuous maps $f_\lambda\colon X\to E$, and $f\colon E\to X$ be a map. If $\{f_\lambda\}_{\lambda\in \Lambda}$ converges locally uniformly to $f$, then $f\colon X\to E$ is continuous. 
\end{prp}

\begin{proof}
Take a directed fundamental system $\Gamma$ of semi-norms for $X$, and fix $x\in X$. It is enough to show that for $\ve\in \bbR_{>0}$ and $p\in \Gamma$, there is a neighborhood $N\subset X$ of $x\in X$ such that $f(N)\subset f(x)+B(p, \ve)$. Since $\{f_\lambda\}_{\lambda\in \Lambda}$ converges locally uniformly to $f$, there exists a neighborhood $M\subset X$ of $x\in X$ and $\lambda_0\in \Lambda$ such that 
\[
\lambda\in \Lambda,\ \lambda\ge \lambda_0\ \Longrightarrow\ 
p_M(f_\lambda-f)<\ve/3. 
\]
Also, since $f_{\lambda_0}$ is continuous, one can find a neighborhood $M'\subset X$ of $x\in X$ such that $f_{\lambda_0}(M')\subset f_{\lambda_0}(x)+B(p, \ve/3)$. Thus, for $y\in M\cap M'$, 
\[
p(f(y)-f(x))\le p(f(y)-f_{\lambda_0}(y))+p(f_{\lambda_0}(y)-f_{\lambda_0}(x))+p(f_{\lambda_0}(x)-f(x))<\ve,
\]
which means that $f(M\cap M')\subset f(x)+B(p, \ve)$. 
\end{proof}

\subsection{Riemann integrals}

To define complex integrals of functions with values in a locally convex space, we need to expand the notion of Riemann integrals to locally convex settings. 

Fix a Hausdorff locally convex space $E$ over $\bbR$. For a bounded closed interval $I=[a, b]$ ($a, b\in \bbR$, $a<b$), we denote by
\[
D(I)\ceq 
\{(x_i)_{i=0}^l\mid l\in \bbZ_{>0},\, a=x_0<x_1<\cdots <x_l=b\}
\]
the set of partitions of $I$. For a partition $\Delta=\{x_i\}_{i=0}^l\in D(I)$, we use the following notations: 
\begin{itemize}
\item
We call $l(\Delta)\ceq l$ the length of $\Delta$. 

\item 
For $i\in[l]$, set $\Delta_i\ceq x_i-x_{i-1}$. We call $|\Delta|\ceq \max\{\Delta_i\mid i\in[l]\}$ the width of $\Delta$. 

\item 
For $i\in[l]$, set $I_i(\Delta)\ceq [x_{i-1}, x_i]$. We call an element $\{\xi_i\}_{i=1}^l\in \prod_{i=1}^lI_i(\Delta)$ a representative of $\Delta$. 
\end{itemize}

A pair $(\Delta, \xi)$ of a partition $\Delta\in D(I)$ and its representative is called a tagged partition. We denote by $\wt{D}(I)$ the set of tagged partitions. This set $\wt{D}(I)$ has the following two directed orders: 
\begin{itemize}
\item 
The directed order $\le$ defined by widths: 
\[
(\Delta, \xi)\le (\Delta', \xi')\rarr
|\Delta|\ge |\Delta'|. 
\]

\item 
The directed order $\subset$ defined by inclusions: 
\[
(\Delta, \xi)\subset (\Delta', \xi')\rarr
\Delta\subset \Delta'. 
\]
\end{itemize}

\begin{dfn}
Let $I\subset \bbR$ be a bounded closed interval, and $f\colon I\to E$ a function. 
\begin{enumerate}
\item 
For a tagged partition $(\Delta, \xi)\in \wt{D}(I)$, 
\[
s_f(\Delta, \xi)\ceq \sum_{i=1}^{l(\Delta)}\Delta_if(\xi_i)\in E
\]
is called a Riemann sum of $f$. 

\item 
We say that $f$ is (Riemann) integrable over $I$ if the net
\[
(\wt{D}(I), \le)\to E, \quad 
(\Delta, \xi)\mapsto s_f(\Delta, \xi)
\]
converges. In this case, we denote
\[
\int_If=\int_If(x)dx\ceq \lim_{(\Delta, \xi)\in \wt{D}(I)}s_f(\Delta, \xi). 
\]
\end{enumerate}
\end{dfn}

\begin{rmk}
Take a fundamental system $\Gamma$ of semi-norms for $E$, and let $I\subset \bbR$ be a bounded closed interval. A function $f\colon I\to E$ is integrable over $I$ if and only if there exists $s\in E$ satisfying the following condition: For $p\in \Gamma$ and $\ve\in \bbR_{>0}$, there is $\delta\in \bbR_{>0}$ such that 
\[
(\Delta, \xi)\in \wt{D}(I),\, |\Delta|<\delta\, \Longrightarrow\, 
p(s_f(\Delta, \xi)-s)<\ve. 
\]
\end{rmk}

\begin{lem}\label{lem:fsubboun}
Let $I\subset \bbR$ be a closed interval, and $f\colon I\to E$ a function. If the net
\[
(\wt{D}(I), \subset)\to E, \quad (\Delta, \xi)\mapsto s_f(\Delta, \xi)
\]
converges, then $f$ is a bounded function, i.e., $f(I)\subset E$ is bounded. 
\end{lem}

\begin{proof}
Take a fundamental system $\Gamma$ of semi-norms for $E$. It is enough to show that for $p\in \Gamma$, there exists $t\in \bbR_{>0}$ such that $f(I)\subset B(p, t)$. By assumption, we have $s\in E$ and $\Delta^0\in D(I)$ satisfying
\[
(\Delta, \xi)\in \wt{D}(I),\, \Delta\supset \Delta^0\, \Longrightarrow\, 
p(s_f(\Delta, \xi)-s)<1. 
\]
Denote $\Delta^0=\{x_i\}_{i=0}^l$, and set $\delta\ceq \min\{\Delta^0_i\mid i\in[l]\}\in \bbR_{>0}$. Let us prove that 
\[
t\ceq \delta^{-1}
\bigl(1+p(s)+\sum_{i=1}^l\Delta^0_ip(f(x_i))\bigr)\in \bbR_{>0}
\]
satisfies $f(I)\subset B(p, t)$. 

Let $x\in I$, and take $j\in[l]$ such that $x\in I_j(\Delta^0)$. For each $i\in[l]$, set
\[
\xi_i\ceq 
\begin{cases}
x_i & (i\neq j) \\
x & (i=j),
\end{cases}
\]
then $\xi\ceq \{\xi_i\}_{i=1}^l$ is a representative of $\Delta^0$. Thus, we have
\[
p(s_f(\Delta^0, \xi))=
p(s_f(\Delta^0, \xi)-s+s)<
1+p(s). 
\]
Since $s_f(\Delta^0, \xi)=\Delta^0_jf(x)+\sum_{i=1, i\neq j}^l\Delta^0_if(x_i)$, we get
\[
\delta p(f(x))\le \Delta^0_jp(f(x))\le 
p(s_f(\Delta^0, \xi))+\sum_{\substack{i=1\\ i\neq j}}^l\Delta_i^0p(f(x_i))<
1+p(s)+\sum_{i=1}^l\Delta^0_ip(f(x_i)), 
\]
which proves $f(x)\in B(p, t)$. 
\end{proof}

\begin{prp}\label{prp:intetwonet}
Let $I\subset \bbR$ be a closed interval, and $f\colon I\to E$ a function. The followings are equivalent: 
\begin{clist}
\item 
The function $f$ is integrable over $I$. 

\item
The net 
\[
(\wt{D}(I), \subset)\to E, \quad
(\Delta, \xi)\mapsto s_f(\Delta, \xi)
\]
converges. 
\end{clist}
\end{prp}

\begin{proof}
(i) $\Longrightarrow$ (ii) is clear since $\Delta\subset \Delta'$ implies $|\Delta|\ge |\Delta'|$. Let us prove the converse. Take a fundamental system $\Gamma$ of semi-norms for $E$, and denote by $s$ the limit point of the net $(\wt{D}(I), \subset)\to E$, $(\Delta, \xi)\mapsto s_f(\Delta, \xi)$. It is enough to show that for $p\in \Gamma$ and $\ve\in \bbR_{>0}$, there exists $\delta\in \bbR_{>0}$ such that 
\[
(\Delta, \xi)\in \wt{D}(I),\, \Delta<\delta\, \Longrightarrow\, 
p(s_f(\Delta, \xi)-s)<\ve. 
\]
By assumption, there is $\Delta^0\in D(I)$ satisfying 
\[
(\Delta, \xi)\in \wt{D}(I),\, \Delta\supset \Delta^0\, \Longrightarrow\, 
p(s_f(\Delta, \xi)-s)<\ve/2.
\]
Note that by \cref{lem:fsubboun}, we have $a\ceq \sup_{x, y\in I}p(f(x)-f(y))\in \bbR_{\ge0}$. We prove that 
\[
\delta\ceq \min\{\Delta^0_i, \ve/2l(\Delta^0)(a+1)\mid i\in[l(\Delta^0)]\}\in \bbR_{>0}
\]
satisfies the condition in the claim. 

For $(\Delta, \xi)\in \wt{D}(I)$ such that $|\Delta|<\delta$, set $\Delta'\ceq \Delta^0\cup \Delta$, and denote
\[
\Delta^0=\{x_i\}_{i=0}^l, \quad \Delta=\{y_j\}_{j=0}^m, \quad
\Delta'=\{z_k\}_{k=0}^n. 
\]
For each $i\in\{0, 1, \ldots, l\}$ (resp. $j\in\{0, 1, \ldots, m\}$), there exists a unique element $p_i\in\{0, 1, \ldots, n\}$ ($q_j\in \{0, 1, \ldots, n\}$) satisfying $z_{p_i}=x_i$ (resp. $z_{q_j}=y_j$). Set
\[
M\ceq \{j\in[m]\mid \forall i\in[l],\, x_i\notin \mathring{I}_j(\Delta)\}, \quad
N\ceq \{k\in[n]\mid \exists j\in[m],\, z_k=y_j\land z_{k-1}=y_{j-1}\},
\]
where $\mathring{I}_j(\Delta)$ stands for the interior of $I_j(\Delta)\subset \bbR$, then $M\to N$, $j\mapsto q_j$ is a bijection. In particular, one can define a representative $\xi'=\{\xi'_k\}_{k=1}^n$ of $\Delta'$ by
\[
\xi'_k\ceq
\begin{cases}
\xi_j & (k\in N \text{ with } j\in M, k=q_j) \\
z_k & (k\in[n]\bs N).
\end{cases}
\]
Since $\Delta'\supset \Delta^0$, we have 
\[
p(s_f(\Delta', \xi')-s)<\ve/2. 
\]
Thus, it suffices to prove that $p(s_f(\Delta, \xi)-s_f(\Delta', \xi'))<\ve/2$. 

By the bijectively of $M\to N$, $j\mapsto q_j$ and the definition of $\xi'$, note that 
\begin{align*}
s_f(\Delta, \xi)-s_f(\Delta', \xi')
&=
\sum_{j\in M}(y_j-y_{j-1})f(\xi_j)+\sum_{j\in[m]\bs M}(y_j-y_{j-1})f(\xi_j)\\
&\quad 
-\sum_{k\in N}(z_k-z_{k-1})f(\xi_k')-\sum_{k\in[n]\bs N}(z_k-z_{k-1})f(\xi_k')\\
&=
\sum_{j\in[m]\bs M}(y_i-y_{j-1})f(\xi_j)-\sum_{k\in[n]\bs N}(z_k-z_{k-1})f(z_k). 
\end{align*}
For $j\in[m]\bs M$, since $|\Delta|<\delta\le \min\{\Delta^0_i\mid i\in[l]\}$, there exists a unique element $i_j\in[l]$ satisfying $x_{i_j}\in \mathring{I}_j(\Delta)$. Hence, we have a bijection
\[
[m]\bs M\to \{i\in[l]\mid \exists j\in[m],\, x_i\in \mathring{I}_j(\Delta)\}, \quad j\mapsto i_j. 
\]
Set
\[
N_1\ceq \{k\in[n]\mid \forall j\in[m],\, z_k\neq y_j\}, \quad
N_2\ceq \{k\in[n]\mid \exists j\in[m],\, z_k=y_j\land z_{k-1}\neq y_{j-1}\}
\]
then $[n]\bs N=N_1\sqcup N_2$, and both
\[
[m]\bs M\to N_1, \quad j\mapsto p_{i_j}, \quad 
[m]\bs M\to N_2, \quad j\mapsto q_j
\]
are bijective. Thus, 
\begin{align*}
\sum_{k\in[n]\bs N}(z_k-z_{k-1})f(z_k)
&=
\sum_{j\in[m]\bs M}(z_{p_{i_j}}-z_{p_{i_j}-1})f(z_{p_{i_j}})+\sum_{j\in[m]\bs M}(z_{q_j}-z_{q_j-1})f(z_{q_j})\\
&=
\sum_{j\in [m]\bs M}(x_{i_j}-y_{j-1})f(x_{i_j})+\sum_{j\in[m]\bs M}(y_j-x_{i_j})f(y_j), 
\end{align*}
and hence
\[
s_f(\Delta, \xi)-s_f(\Delta', \xi')=
\sum_{j\in[m]\bs M}(y_j-x_{i_j})(f(\xi_j)-f(y_j))+\sum_{j\in[m]\bs M}(x_{i_j}-y_{j-1})(f(\xi_j)-f(x_{i_j})).
\]
Therefore, by the injectively of $[m]\bs M\to [l]$, $j\mapsto i_j$ and $|\Delta|<\delta\le \ve/2l(a+1)$, we have
\[
p(s_f(\Delta, \xi)-s_f(\Delta', \xi'))\le 
\sum_{j\in[m]\bs M}(y_j-y_{j-1})a\le \#([m]\bs M)|\Delta|a<\ve/2. 
\qedhere
\]
\end{proof}

\begin{lem}\label{lem:sfboun}
Let $I\subset \bbR$ be a bounded closed interval. For a bounded function $f\colon I\to E$, the set $\{s_f(\Delta, \xi)\mid (\Delta, \xi)\in \wt{D}(I)\}\subset E$ is bounded. 
\end{lem}

\begin{proof}
Take a fundamental system $\Gamma$ of semi-norms for $E$, and let $p\in \Gamma$. For $(\Delta, \xi)\in \wt{D}(I)$, we have
\[
p(s_f(\Delta, \xi))\le 
\sum_{i=1}^{l(\Delta)}\Delta p(f(\xi_i))\le 
\bigl(\sum_{i=1}^{l(\Delta)}\Delta_i\bigr)\sup_{x\in I}p(f(x))=
(b-a)\sup_{x\in I} p(f(x)), 
\]
which proves the lemma. 
\end{proof}

The proof of the following proposition based on that of \cite[Proposition 2.2]{WYY}. 

\begin{prp}\label{prp:intequiv}
Let $E$ be a quasi-complete Hausdorff locally convex space, and take a fundamental system $\Gamma$ of semi-norms for $E$. For a bounded closed interval $I\subset \bbR$ and a function $f\colon I\to E$, the following conditions are equivalent: 
\begin{clist}
\item 
The function $f$ is integrable over $I$. 

\item 
For $p\in \Gamma$ and $\ve\in \bbR_{>0}$, there exists $\delta\in \bbR_{>0}$ such that 
\[
(\Delta, \xi), (\Delta', \xi')\in \wt{D}(I),\, |\Delta|, |\Delta'|<\delta\, \Longrightarrow\, 
p(s_f(\Delta, \xi)-s_f(\Delta', \xi'))<\ve. 
\]

\item 
For $p\in\Gamma$ and $\ve\in \bbR_{>0}$, there exists $\Delta_0\in D(I)$ such that 
\[
(\Delta, \xi), (\Delta', \xi')\in \wt{D}(I),\, \Delta, \Delta'\supset \Delta^0\, \Longrightarrow\, 
p(s_f(\Delta, \xi)-s_f(\Delta', \xi'))<\ve. 
\]

\item 
For $p\in \Gamma$ and $\ve\in \bbR_{>0}$, there exists $\Delta^0\in D(I)$ such that if $\xi$ and $\xi'$ are representatives of $\Delta^0$, then 
\[
p(s_f(\Delta^0, \xi)-s_f(\Delta^0, \xi'))<\ve. 
\]
\end{clist}
\end{prp}

\begin{proof}
(i) $\Longleftrightarrow$ (ii): The condition (ii) means that the net $(\wt{D}(I), \le)\to E$, $(\Delta, \xi)\mapsto s_f(\Delta, \xi)$ is a Cauchy net. Thus, this follows from \cref{lem:sfboun} and the quasi-completeness of $E$. 

(i) $\Longleftrightarrow$ (iii): The condition (iii) means that the net $(\wt{D}(I), \subset)\to E$, $(\Delta, \xi)\mapsto s_f(\Delta, \xi)$ is a Cauchy net. Thus, this follows from \cref{prp:intetwonet}, \cref{lem:sfboun} and the quasi-completeness of $E$. 

(iii) $\Longrightarrow$ (iv): This is trivial. 

(iv) $\Longrightarrow$ (iii): 
For $p\in \Gamma$ and $\ve\in \bbR_{>0}$, there exists $\Delta^0\in D(I)$ such that if $\xi$ and $\xi'$ are representative of $\Delta^0$, then $p(s_f(\Delta^0, \xi)-s_f(\Delta^0, \xi'))<\ve/2$. Fix a representative $\xi^0\ceq \{\xi^0_i\}_{i=1}^l$ of $\Delta^0$. It is enough to show that 
\[
(\Delta, \xi)\in \wt{D}(I),\, \Delta\supset \Delta^0\, \Longrightarrow\, 
p(s_f(\Delta, \xi)-s_f(\Delta^0, \xi^0))<\ve/2. 
\]
First, set 
\[
A\ceq \sum_{i=1}^lA_i, \quad
A_i\ceq \{\Delta^0_if(x)\mid x\in I_i(\Delta^0)\}, 
\]
then clearly, we have $p(a-b)<\ve/2$ for $a, b\in A$. Hence, it follows that $p(a)<\ve/2$ for all $a\in \conv(A-A)$. Now, take $(\Delta, \xi)\in \wt{D}(I)$ such that $\Delta\supset \Delta^0$, and denote
\[
\Delta^0=\{x_i\}_{i=0}^l, \quad \Delta=\{y_j\}_{j=0}^m. 
\]
For each $i\in[l]$, there is a unique element $j_i\in[m]$ such that $y_{j_i}=x_i$. 
Since $[m]=\bigsqcup_{i=1}^l\{j\in[m]\mid j_{i-1}<j\le j_i\}$, we have
\begin{align*}
s_f(\Delta, \xi)-s_f(\Delta^0, \xi^0)
&=
\sum_{j=1}^m\Delta_jf(\xi_j)-\sum_{i=1}^l\Delta_i^0f(\xi_i^0)=
\sum_{i=1}^l\sum_{j_{i-1}<j\le j_i}\Delta_jf(\xi_j)-\sum_{i=1}^l\Delta_i^0f(\xi_i^0)\\
&=
\sum_{i=1}^l\sum_{j_{i-1}<j\le j_i}\frac{\Delta_j}{\Delta_i^0}(\Delta_i^0f(\xi_j)-\Delta_i^0f(\xi_i^0)).
\end{align*}
Thus,
\[
s_f(\Delta, \xi)-s_f(\Delta^0, \xi^0)\in
\sum_{i=1}^l\conv(A_i-A_i)=\conv\bigl(\sum_{i=1}^l(A_i-A_i)\bigr)\subset \conv(A-A). 
\]
Now the claim is clear. 
\end{proof}

Let $I\subset \bbR$ be a bounded closed interval. For a bounded function $f\colon I\to E$, a semi-norm $p\colon E\to \bbR$ and $\Delta\in D(I)$, define
\[
a_f(p, \Delta)\ceq 
\sum_{i=1}^{l(\Delta)}\Delta_i\sup_{x, y\in I_i(\Delta)}p(f(x)-f(y))\in \bbR_{\ge0}. 
\]

\begin{lem}\label{lem:afpD}
Let $E$ be a quasi-complete Hausdorff locally convex space, and take a fundamental system $\Gamma$ of semi-norms for $E$. Also, let $I\subset \bbR$ be a bounded closed interval, and $f\colon I\to E$ a bounded function. The function $f$ is integrable over $I$ if the following condition holds: For $p\in \Gamma$ and $\ve\in \bbR_{>0}$, there exists $\Delta\in D(I)$ such that $a_f(p, \Delta)<\ve$. 
\end{lem}

\begin{proof}
For $p\in \Gamma$ and $\ve\in \bbR_{>0}$, take $\Delta\in D(I)$ such that $a_f(p, \Delta)<\ve$. Then, if $\xi$ and $\xi'$ are representative of $\Delta$, we have
\[
p(s_f(\Delta, \xi)-s_f(\Delta, \xi'))\le 
\sum_{i=1}^{l(\Delta)}\Delta_ip(f(\xi_i)-f(\xi'_i))\le 
a_f(p, \Delta)<\ve. 
\]
Thus, the condition (iv) in \cref{prp:intequiv} holds. 
\end{proof}

\begin{prp}
Let $E$ be a quasi-complete Hausdorff locally convex space, and $I\subset \bbR$ a bounded closed interval. Any continuous function $f\colon I\to E$ is integrable. 
\end{prp}

\begin{proof}
Take a fundamental system $\Gamma$ of semi-norms for $E$, and denote $I=[a, b]$. Since $I\subset \bbR$ is compact and $f$ is continuous, the function $f$ is uniformly continuous. Hence, for $p\in \Gamma$ and $\ve\in \bbR_{>0}$, there exists $\delta\in \bbR_{>0}$ such that 
\[
x, y\in I,\, |x-y|<\delta\, \Longrightarrow\, 
p(f(x)-f(y))<\frac{\ve}{2(b-a)}. 
\]
Take a partition $\Delta\in D(I)$ satisfying $|\Delta|<\delta$, then
\[
a_f(p, \Delta)\le \frac{\ve}{2(b-a)}\sum_{i=1}^{l(\Delta)}\Delta_i=\frac{\ve}{2}<\ve. 
\]
Thus, the proposition follows from \cref{lem:afpD}. 
\end{proof}

\begin{lem}\label{lem:inttriaineq}
Let $I\subset \bbR$ be a bonded closed interval. For a continuous function $f\colon I\to E$ and a continuous semi-norm $p\colon E\to \bbR$, we have
\[
p\bigl(\int_If(x)dx\bigr)\le \int_Ip(f(x))dx
\]
\end{lem}

\begin{proof}
For a tagged partition $(\Delta, \xi)\in \wt{D}(I)$, notice that
\[
p(s_f(\Delta, \xi))\le s_{pf}(\Delta, \xi). 
\]
Since $p\colon E\to \bbR$ is continuous, we have
\[
p\bigl(\int_If(x)dx\bigr)=\lim_{(\Delta, \xi)\in \wt{D}(I)}p(s_f(\Delta, \xi))\le
\lim_{(\Delta, \xi)\in \wt{D}(I)}s_{pf}(\Delta, \xi)=
\int_Ip(f(x))dx. 
\qedhere
\]
\end{proof}

Let $I\subset \bbR$ be an interval, and $f\colon I\to E$ a continuous function. For $a, b\in I$, denote
\[
\int_a^bf(x)dx\ceq 
\begin{cases}
\displaystyle\int_{[a, b]}f(x)dx & (a\le b) \vspace{0.2cm}\\
-\displaystyle\int_{[b, a]}f(x)dx  & (a>b). 
\end{cases}
\]
Then, for a continuous semi-norm $p\colon E\to \bbR$, we have
\[
p\bigl(\int_a^bf(x)dx\bigr)\le \abs{\int_a^bp(f(x))dx}. 
\]

\begin{rmk}
Let $I\subset \bbR$ be an interval, and $f\colon I\to E$ a function. As in \cref{dfn:holofunc}, we can define (real) differentiability of $f$, and all results in \cref{ss:holofunc} hold with minor modifications. 
\end{rmk}

\begin{prp}
$I\subset \bbR$ be an interval and $x_0\in I$. For a continuous function $f\colon I\to E$, the function
\[
I\to E, \quad x\mapsto \int_{x_0}^xf(t)dt
\]
is differentiable, and we have
\[
\frac{d}{dx}\int_{x_0}^xf(t)dt=f(x) \quad (x\in I). 
\]
\end{prp}

\begin{proof}
Let us denote $\wt{f}\colon I\to E$, $x\mapsto \int_{x_0}^xf(t)dt$. We prove that $\wt{f}$ is differentiable at any point $a\in I$. Take a fundamental system $\Gamma$ of semi-norms for $E$. For $p\in \Gamma$ and $\ve\in \bbR_{>0}$, since $f$ is continuous at $a$, there exists $\delta\in \bbR_{>0}$ such that 
\[
t\in I,\, |t-a|<\delta\, \Longrightarrow\, 
p(f(t)-f(a))<\ve. 
\]
Thus, if $x\in I\bs\{a\}$ satisfies $|x-a|<\delta$, then
\[
p\Bigl(\frac{\wt{f}(x)-\wt{f}(a)}{x-a}-f(a)\Bigr)\le 
\frac{1}{|x-a|}\abs{\int_a^xp(f(t)-f(a))dt}<\ve, 
\]
which shows $\wt{f}'(a)=f(a)$. 
\end{proof}

\begin{lem}
Let $I$ be an interval, and $f\colon I\to E$ a differentiable function. The followings are equivalent: 
\begin{clist}
\item 
The function $f$ is constant. 

\item 
For all $x\in I$, we have $f'(x)=0$. 
\end{clist}
\end{lem}

\begin{proof}
(i) $\Longrightarrow$ (ii) is clear. We prove the converse. 
Take a fundamental system $\Gamma$ of semi-norms for $E$, and fix $x_0\in I$. To prove $f(x)=f(x_0)$ for $x\in I$, it is enough to show that $p(f(x)-f(x_0))=0$ for all $p\in \Gamma$. Define a function $\wt{f}$ by
\[
\wt{f}\colon I\to \bbR, \quad x\mapsto p(f(x)-f(x_0)), 
\]
then $\wt{f}$ is differentiable at any $a\in I$, and we have $\wt{f}'(a)=0$. In fact, for $x\in I$, since
\[
\abs{\frac{\wt{f}(x)-\wt{f}(a)}{x-a}}\le 
p\Bigl(\frac{f(x)-f(a)}{x-a}\Bigr), 
\]
the claim follows from the continuity of $p$ and $f'(a)=0$. Thus, the function $\wt{f}\colon I\to \bbR$ is constant, and hence
\[
p(f(x)-f(x_0))=\wt{f}(x)=\wt{f}(x_0)=0. 
\qedhere
\]
\end{proof}

Let $I\subset \bbR$ be an interval, and $f\colon I\to E$ a function. A differentiable function $F\colon I\to E$ is called a primitive function of $f$ if it satisfies $F'(x)=f(x)$ for all $x\in I$. 

\begin{prp}\label{prp:fundthm}
Let $I=[a, b]$ be an interval, and $f\colon I\to E$ a continuous function. If $F\colon I\to E$ is a primitive function of $f$, then we have
\[
\int_If=F(b)-F(a). 
\]
\end{prp}

\begin{proof}
Let us denote $\wt{f}\colon I\to E$, $x\mapsto \int_a^xf(t)dt$. Then, since
\[
\wt{f}'(x)=f(x)=F'(x) \quad (x\in I), 
\]
the function $\wt{f}-F\colon I\to E$ is constant. Thus, we have
\[
\int_If=\wt{f}(b)-\wt{f}(a)=F(b)-F(a). 
\qedhere
\]
\end{proof}

\begin{rmk}
Let $X$ be a set. As in \cref{dfn:unifconv}, we can define a uniform convergence of a net $\{f_\lambda\}_{\lambda\in \Lambda}$ consisting of maps $f_\lambda\colon X\to E$, and all results in \cref{ss:unifconv} hold with minor modification. 
\end{rmk}

\begin{prp}\label{prp:unifconvint}
Let $I\subset \bbR$ be a bounded closed interval. If a net $\{f_\lambda\colon I\to E\}_{\lambda\in \Lambda}$ consisting of integrable functions converges uniformly to an integrable function $f\colon I\to E$, then we have
\[
\lim_{\lambda\in \Lambda}\int_If_\lambda=\int_I f. 
\]
\end{prp}

\begin{proof}
Take a fundamental system $\Gamma$ of semi-norms for $E$. For $p\in \Gamma$ and $\ve\in \bbR_{>0}$, since $\{f_\lambda\}_{\lambda\in \Lambda}$ converges uniformly to $f$, there exists $\lambda_0\in \Lambda$ such that 
\[
\lambda\in \Lambda,\, \lambda\ge\lambda_0\, \Longrightarrow\, 
p_I(f_\lambda-f)<\frac{\ve}{b-a}. 
\]
Thus, if $\lambda\in \Lambda$ satisfies $\lambda\ge\lambda_0$, then
\[
p\Bigl(\int_If_\lambda(x)dx-\int_If(x)dx\Bigr)\le 
\int_Ip(f_\lambda(x)-f(x))dx<\ve, 
\]
where we used \cref{lem:inttriaineq} in the first inequality. 
\end{proof}

\subsection{Cauchy-Goursat theorem and Cauchy's integral formula}

Fix a quasi-complete Hausdorff locally convex space $E$ over $\bbC$. We use the following notations about curves in $\bbC$: 
\begin{itemize}
\item 
A function $\gamma\colon I\to \bbC$ of class $C^1$, where $I\subset \bbR$ is a bounded closed interval, is called a parametrized smooth curve if  $\gamma'(t)\neq 0$ for all $t\in I$. 

\item
 A continuous function $\gamma\colon I\to \bbC$, where $I\subset \bbR$ is a bounded closed interval, is called a parametrized piecewise smooth curve if there exists a partition $\Delta\in D(I)$ such that $\gamma\colon I_i(\Delta)\to \bbC$ is a parametrized smooth curve for all $i\in[l(\Delta)]$. 
 
\item
For parametrized piecewise smooth curves $\gamma_1\colon [a_1, b_1]\to \bbC$ and $\gamma_2\colon [a_2, b_2]\to \bbC$, a reparametrization from $\gamma_1$ to $\gamma_2$ is a bijective function $\varphi\colon [a_2, b_2]\to [a_1, b_1]$ of class $C^1$ satisfying (i) $\varphi(a_2)=a_1$ and $\varphi(b_2)=b_1$, (ii) $\gamma_2=\gamma_1\circ \varphi$, (iii) $\varphi'(t)>0$ for $t\in[a_2, b_2]$. We denote $\gamma_1\sim \gamma_2$ if there exists a reparametrization from $\gamma_1$ to $\gamma_2$. This relation $\sim$ is an equivalence relation on the set of smooth parametrized curves. 

\item 
A piecewise smooth curve in $\bbC$ is an equivalence class $[\gamma]$ of a piecewise smooth parametrized curve $\gamma\colon I\to \bbC$ by the equivalence relation $\sim$. For a smooth curve $C$ in $\bbC$, a piecewise smooth parametrized curve $\gamma\colon I\to \bbC$ such that $C=[\gamma]$ is called a parametrization of $C$. Note that for a piecewise smooth curve $C$, there exists a parametrization $\gamma\colon I\to \bbC$ whose domain is $I=[0, 1]$. 

\item 
For a piecewise smooth curve $C$ in $\bbC$, the image of $C$ is $\gamma(I)\subset \bbC$, where $\gamma\colon I\to \bbC$ is a parametrization of $C$. Note that the image of $C$ is independent of the choice of parametrization. We abusively denote by $C$ the image of $C$. 

\item 
For a piecewise smooth curve $C$ in $\bbC$, the initial (resp. final) point of $C$ is $\gamma(a)$ (resp. $\gamma(b)$), where $\gamma\colon [a, b]\to \bbC$ is a parametrization of $C$. Note that both initial and final points are independent of the choice of parametrization. 

\item 
For a piecewise smooth curve $C$, the length $l(C)\in \bbR_{\ge0}$ is defined by 
\[
l(C)\ceq \sum_{i=1}^{l(\Delta)}\int_{I_i(\Delta)}|\gamma'(t)|dt.
\]
Here we take a parametrization $\gamma\colon I\to \bbC$ of $C$ and a partition $\Delta\in D(I)$ such that $\gamma\colon I_i(\Delta)\to \bbC$ is a parametrized smooth curve for all $i\in[l(\Delta)]$, but the definition of $l(C)$ is independent of the choices made. 

\item 
For piecewise smooth curves $C_1, \ldots, C_n$ such that the final point of $C_i$ is equal to the initial point of $C_{i+1}$, we define a smooth curve $C_1+\cdots +C_n$ as follows: Take a parametrization $\gamma_i\colon[0, 1]\to \bbC$ of $C_i$ for each $i\in[n]$, and define a piecewise smooth parametrized curve $\gamma\colon [0, 1]\to \bbC$ by
\[
\gamma(t)\ceq 
\gamma_i(nt-i+1)
\quad (i\in[n],\, (i-1)/n\le t<i/n).
\]
Then $C_1+\cdots +C_n\ceq [\gamma]$. Note that the definition of $C_1+\cdots +C_n$ is independent of the choices of parametrization. 

\item 
For a piecewise smooth curve $C$, we define a smooth curve $C^-$ as follows: Take a parametrization $\gamma\colon[a, b]\to \bbC$, and define a piecewise smooth parametrized curve $\gamma^-$ by
\[
\gamma^-\colon [a, b]\to \bbC, \quad 
t\mapsto \gamma(a+b-t). 
\]
Then $C^-\ceq [\gamma^-]$. Note that the definition of $C^-$ is independent of the choice of parametrization. 

\item 
For $z, w\in \bbC$, the line segment $L(z, w)$ is a piecewise smooth curve defined by the piecewise smooth parametrized curve
\[
\gamma_{z, w}\colon [0, 1]\to \bbC, \quad 
t\mapsto z+t(w-z). 
\]

\item 
A triangle $T$ is a piecewise smooth curve $T=L(z_1, z_2)+L(z_2, z_3)+L(z_3, z_1)$ where $z_1, z_2, z_3\in \bbC$ are three different points. 
\end{itemize}

\begin{dfn}
Let $U\subset \bbC$ be a domain. For a piecewise smooth curve $C\subset U$ and a continuous function $f\colon U\to E$, define
\[
\int_Cf=\int_Cf(z)dz\ceq
\sum_{i=1}^{l(\Delta)}\int_{I_i(\Delta)}\gamma'(t)(f\circ \gamma)(t) dt.
\]
Here we take a parametrization $\gamma\colon I\to \bbC$ of $C$ and a partition $\Delta\in D(I)$ such that $\gamma\colon I_i(\Delta)\to \bbC$ is a parametrized smooth curve for all $i\in[l(\Delta)]$, but the definition of $\int_Cf$ is independent of the choices made. 
\end{dfn}

\begin{lem}
Let $U\subset \bbC$ be a domain. For a piecewise smooth curve $C\subset U$, a continuous function $f\colon U\to E$, and a continuous semi-norm $p\colon E\to \bbR$, we have
\[
p\bigl(\int_Cfdz\bigr)\le \sup_{z\in C}p(f(z))\cdot l(C). 
\]
\end{lem}

\begin{proof}
Take a parametrization $\gamma\colon I\to \bbC$ of $C$ and a partition $\Delta\in D(I)$ such that $\gamma\colon I_i(\Delta)\to \bbC$ is a parametrized smooth curve for all $i\in[l(\Delta)]$. Then 
\[
p\bigl(\int_Cfdz\bigr)\le 
\sum_{i=1}^{l(\Delta)}p\bigl(\int_{I_i(\Delta)}\gamma'(t)(f\circ \gamma)(t)dt\bigr)
\le \sum_{i=1}^{l(\Delta)}\int_{I_i(\Delta)}|\gamma'(t)|p((f\circ\gamma)(t))dt\le 
\sup_{z\in C}p(f(z))\cdot l(C),
\]
where we used \cref{lem:inttriaineq} in the second inequality. 
\end{proof}

\begin{prp}\label{prp:unifcompint}
Let $U\subset \bbC$ be a domain, and $C\subset U$ a piecewise smooth curve. For a net $\{f_\lambda\}_{\lambda\in \Lambda}$ consisting of continuous functions $f_\lambda\colon U\to E$ and a continuous function $f\colon U\to E$, if $\{f_\lambda\}_{\lambda\in \Lambda}$ converges uniformly to $f$ on $C$, then we have
\[
\lim_{\lambda\in \Lambda}\int_Cf_\lambda=\int_Cf. 
\]
\end{prp}

\begin{proof}
Take a parametrization $\gamma\colon I\to \bbC$ of $C$ and a partition $\Delta\in D(I)$ such that $\gamma\colon I_i(\Delta)\to \bbC$ is a parametrized smooth curve for all $i\in[l(\Delta)]$. Then for $\lambda\in \Lambda$, 
\[
\int_If_\lambda=\sum_{i=1}^{l(\Delta)}\int_{I_i(\Delta)}\gamma'(t)(f_\lambda\circ\gamma)(t)dt.
\]
Since the net $\{\gamma'(f_\lambda\circ \gamma)\}_{\lambda\in \Lambda}$ converges uniformly to $\gamma'(f\circ \gamma)$ on $I_i(\Delta)$, we have
\[
\lim_{\lambda\in \Lambda}\int_If_\lambda=
\sum_{i=1}^{l(\Delta)}\lim_{\lambda\in \Lambda}\int_{I_i(\Delta)}\gamma'(t)(f_\lambda\circ \gamma)(t)dt=
\sum_{i=1}^{l(\Delta)}\int_{I_i(\Delta)}\gamma'(t)(f\circ \gamma)(t)dt=
\int_If, 
\]
where we used \cref{prp:unifconvint} in the second equality. 
\end{proof}

\begin{prp}\label{prp:primint}
Let $U\subset \bbC$ be a domain, $C\subset U$ a piecewise smooth curve, and $f\colon U\to E$ a continuous function. If $F\colon U\to E$ is a primitive function of $f$, then we have
\[
\int_Cf=F(z_1)-F(z_0), 
\]
where $z_0$, $z_1$ are the initial and final points of $C$ respectively. 
\end{prp}

\begin{proof}
Take a parametrization $\gamma\colon I\to \bbC$ of $C$ and a partition $\Delta\in D(I)$ such that $\gamma\colon I_i(\Delta)\to \bbC$ is a parametrized smooth curve for all $i\in[l(\Delta)]$. Denote $I_i(\Delta)=[a_{i-1}, a_i]$ for each $i\in[l(\Delta)]$, then
\[
\int_Cf=\sum_{i=1}^{l(\Delta)}\int_{a_{i-1}}^{a_i}\gamma'(t)(f\circ \gamma)(t)dt. 
\]
Since $F\circ \gamma\colon I_i(\Delta)\to E$ is a primitive function of $I_i(\Delta)\to E$, $t\mapsto \gamma'(t)(f\circ \gamma)(t)$, by \cref{prp:fundthm}, we have
\[
\int_Cf=
\sum_{i=1}^{l(\Delta)}\bigl((F\circ \gamma)(a_i)-(F\circ \gamma)(a_{i-1})\bigr)=
F(z_1)-f(z_0). 
\qedhere
\]
\end{proof}

\begin{thm}\label{thm:Goursat}
Let $U\subset \bbC$ be a domain, and $T\subset U$ be a triangle whose interior is in $U$. For a holomorphic function $f\colon U\to E$,we have
\[
\int_Tf=0. 
\]
\end{thm}

\begin{proof}
Take a fundamental system $\Gamma$ of semi-norms for $E$. Since $E$ is Hausdorff, it is enough to show that $p(\int_Tf)=0$ for all $p\in \Gamma$. As in the scalar-valued case, one can construct a sequence $\{T_n\}_{n\in \bbN}$ of triangles such that: 
\begin{itemize}
\item 
$T_0=T$. 

\item 
For each $n\in \bbN$, the diameter $d_n$ and length $l_n$ of $T_n$ satisfy $d_n=d_0/2^n$ and $l_n=l_0/2^n$. 

\item 
For each $n\in \bbN$, we have $p(\int_Tf)\le 4^np(\int_{T_n}f)$. 
\end{itemize}
The condition $d_n\le d_0/2^n$ yields that there is a point $a\in \bbC$ which is in the interior or on the boundary of $T_n$ for all $n\in \bbN$. For any $\ve\in \bbR_{>0}$, since $f$ is holomorphic at $a\in U$, there exists $\delta\in \bbR_{>0}$ such that $D_\delta(a)\subset U$ and 
\[
0<|z-a|<\delta\, \Longrightarrow\, 
p\Bigl(\frac{f(z)-f(a)}{z-a}-f'(a)\Bigr)<\frac{\ve}{d_0l_0}. 
\]
Also, since $d_n\le d_0/2^n$ and $a$ is in the interior or on the boundary of $T_n$, there exists $n\in \bbN$ such that $T_n\subset D_\delta(a)$. Fix such $n\in \bbN$. 
%\[
%n\in \bbN,\, n\ge n_0\, \Longrightarrow\, T_n\subset D_\delta(a). 
%\]
Now, notice that the function $U\to E$, $z\mapsto f(a)+(z-a)f'(a)$ has a primitive function, then by \cref{prp:primint}, 
\[
p\bigl(\int_{T_n}f\bigr)=
p\bigl(\int_{T_n}\bigl(f(z)-f(a)-f'(a)(z-a)\bigr)dz\bigr)<
\sup_{z\in T_n}|z-a|\cdot\frac{\ve l_n}{d_0l_0}\le
\sup_{z\in T_n}|z-a|\cdot\frac{\ve}{2^nd_0}. 
\]
Since $|z-a|\le d_n\le d_0/2^n$ for all $z\in T_n$, we have
\[
p\bigl(\int_Tf\bigr)\le 4^np\bigl(\int_{T_n}f\bigr)<\ve, 
\]
which shows the claim. 
\end{proof}

\begin{cor}\label{cor:Goursat}
Let $U\subset \bbC$ be a domain, $T\subset U$ a triangle whose interior is in $U$, and $f\colon U\to E$ a continuous function. If $f$ is holomorphic on $U$ except a point $z_0\in U$, then we have
\[
\int_Tf=0. 
\]
\end{cor}

\begin{proof}
If $z_0$ is in the exterior of $T$, then the corollary follows immediately from \cref{thm:Goursat}, so we may assume that $z_0$ is in the interior or on the boundary of $T$. 

First, consider the case that $z_0$ is a vertex of $T$. Take a fundamental system $\Gamma$ of semi-norms for $E$. It is enough to show that $p(\int_Tf)=0$ for all $p\in \Gamma$. Since $f$ is continuous and $T\subset \bbC$ is compact, there exists $t\in \bbR_{>0}$ such that $p(f(z))\le t$ ($z\in T$). For $\ve\in \bbR_{>0}$, by taking two points on the edges of $T$, one can make a sub-triangle $T'\subset T$ such that: 
\begin{itemize}
\item 
The point $z_0$ is a vertex of $T'$. 

\item 
The length of $T'$ satisfies $l(T')<\ve/t$. 

\item
$\int_Tf=\int_{T'}f$. 
\end{itemize}
Thus, we have
\[
p(\int_Tf)=p(\int_{T'}f)\le \sup_{z\in T}f(z)\cdot l(T')<\ve, 
\]
which shows the claim. 

If $z_0$ is not a vertex of $T$, divide $T$ into two or three triangles by connecting $z_0$ and vertices of $T$, then $\int_Tf=0$ follows from the previous case. 
\end{proof}

Let $U\subset \bbC$ be a domain and $a\in U$. We say that $U$ is star-shaped with respect to $a$ if $\{a+t(z-a)\mid t\in[0, 1]\}\subset U$ for all $z\in U$. 

\begin{thm}\label{thm:Cauchy}
Let $U\subset \bbC$ be a star-shaped domain with respect to $a_0\in U$, and $f\colon U\to E$ a continuous function. If $f$ is holomorphic on $U$ except a point $z_0\in U$, then: 
\begin{enumerate}
\item 
There exists a primitive function $F\colon U\to E$ of $f$. 

\item 
For any piecewise smooth closed curve $C\subset U$, we have
\[
\int_Cf=0. 
\]
\end{enumerate}
\end{thm}

\begin{proof}
(1) Note that $L(a_0, z)\subset U$ for all $z\in U$ since $U$ is star-shaped with respect to $a_0$. Define a function $F$ by
\[
F\colon U\to E, \quad z\mapsto \int_{L(a_0, z)}f(\zeta)d\zeta. 
\]
We prove that $F$ is a primitive function of $f$. Let $a\in U$, and take $R\in \bbR_{>0}$ such that $D_R(a)\subset U$. For $z\in D_R(a)$, we have a triangle $T(z)\ceq L(a_0, a)+L(a, z)+L(z, a_0)\subset U$ whose interior in $U$. Hence, for $z\in D_R(a)$, by \cref{cor:Goursat}, 
\[
\int_{T(z)}f(\zeta)d\zeta=0, 
\]
or equivalently, 
\[
F(z)-F(a)=\int_{L(a, z)}f(\zeta)d\zeta. 
\]
Now, take a fundamental system $\Gamma$ of semi-norms for $E$. For $p\in \Gamma$ and $\ve\in \bbR_{>0}$, since $f$ is continuous at $a$, there exists $\delta\in \bbR_{>0}$ such that $\delta<R$ and 
\[
\zeta\in U,\, |\zeta-a|<\delta\, \Longrightarrow\, p(f(\zeta)-f(a))<\ve/2. 
\]
Thus, if $z\in U\bs\{a\}$ satisfies $|z-a|<\delta$, then
\[
p\Bigl(\frac{F(z)-F(a)}{z-a}-f(a)\Bigr)=
\frac{1}{|z-a|}p\Bigl(\int_{L(a, z)}\bigl(f(\zeta)-f(a)\bigr)d\zeta\Bigr)\le
\frac{L(a, z)}{|z-a|}\sup_{z\in L(a, z)}(f(\zeta)-f(a))\le \ve/2<\ve, 
\]
which shows $F'(a)=f(a)$. 

(2) This follows from (1) and \cref{prp:primint}. 
\end{proof}

\begin{thm}\label{thm:Cauchyformu}
Let $U\subset \bbC$ be a domain, and $f\colon U\to E$ a holomorphic function. For an open disk $D$ such that $\ol{D}\subset U$, we have
\[
f(a)=\frac{1}{2\pi\sqrt{-1}}\int_{\pdd D}\frac{f(z)}{z-a}dz \quad (a\in D), 
\]
where $\pdd D$ denotes the positively oriented boundary of $D$. 
\end{thm}

\begin{proof}
Let $a\in D$. Define a function $\wt{f}$ by
\[
\wt{f}\colon U\to E, \quad z\mapsto
\begin{cases}
\displaystyle\frac{f(z)-f(a)}{z-a} & (z\neq a) \\
f'(a) & (z=a),
\end{cases}
\]
then $\wt{f}$ is continuous on $U$ and holomorphic on $U\bs\{a\}$. Take an open disk $D'$ such that $\ol{D}\subset D'\subset U$. Then since $D'$ is star-shaped with respect to the center, we have by \cref{thm:Cauchy} (3), 
\[
\int_{\pdd D}\wt{f}=0, 
\]
and hence
\[
\int_{\pdd D}\frac{f(z)}{z-a}dz=
\int_{\pdd D}\frac{f(a)}{z-a}dz=2\pi\sqrt{-1}f(a). 
\qedhere
\]
\end{proof}

\subsection{Laurent series expansion}

Fix a quasi-complete Hausdorff locally convex space $E$ over $\bbC$. 

\begin{lem}\label{lem:powserunif}
Let $a\in \bbC$ and $R\in \bbR_{>0}$. For a sequence $\{a_n\}_{n\in \bbN}\subset E$, if the series $\sum_{n=0}^{\infty}(z-a)^na_n$ converges in $E$ for all $z\in D_R(a)$, then the series $\sum_{n=0}^{\infty}f_n$ of functions $f_n\colon D_R(a)\to E$, $z\mapsto (z-a)^na_n$ converges uniformly on compact sets. 
\end{lem}

\begin{proof}
Let $K\subset D_R(a)$ be a compact set, and take a fundamental system $\Gamma$ of semi-norms for $E$. Set $r\ceq \max\{|z-a|\mid z\in K\}$, then $r<R$. Fix $z_0\in D_R(a)$ satisfying $r<|z_0-a|<R$. Since $\sum_{n=0}^{\infty}(z_0-a)^na_n$ converges, we have $\lim_{n\to \infty}(z_0-a)^na_n=0$. Hence, for $p\in \Gamma$, there exists $M_p\in \bbR_{>0}$ such that $p((z_0-a)^na_n)<M_p$ for all $n\in \bbN$. For $p\in \Gamma$, we have
\[
p((z-a)^na_n)=\frac{|z-a|^n}{|z_0-a|^n}p(|z_0-a|^na_n)\le 
\frac{r^n}{|z_0-a|^n}M_p \quad (z\in K).
\]
Thus, by \cref{prp:Mtest}, the series $\sum_{n\in \bbN}f_n$ converges uniformly on $K$. 
\end{proof}

\begin{thm}\label{thm:Taylor}
Let $U\subset \bbC$ be a domain, and $f\colon U\to E$ a holomorphic function. For $a\in U$ and $R\in \bbR_{>0}$ such that $D_R(a)\subset U$, there exists a unique expansion
\[
f(z)=\sum_{n=0}^{\infty}(z-a)^na_n \quad (z\in D_R(a)). 
\] 
Moreover, we have
\[
a_n=\frac{1}{2\pi\sqrt{-1}}\int_{\pdd D_r(a)}\frac{f(\zeta)}{(\zeta-a)^{n+1}}d\zeta \quad (r\in \bbR_{>0}, r<R). 
\]
for all $n\in \bbN$
\end{thm}

\begin{proof}
The uniqueness of expansion is clear. In fact, if we have an expansion $f(z)=\sum_{n=0}^{\infty}(z-a)^na_n$, then for $r\in \bbR_{>0}$ satisfying $r<R$, by \cref{lem:powserunif} and \cref{prp:unifcompint}, 
\[
\int_{\pdd D_r(a)}\frac{f(\zeta)}{(\zeta-a)^{n+1}}d\zeta=
\sum_{m=0}^{\infty}\int_{\pdd D_r(a)}(\zeta-a)^{m-n-1}a_md\zeta=
\sum_{m=0}^{\infty}2\pi\sqrt{-1}\delta_{m, n}a_m=2\pi\sqrt{-1}a_n. 
\]

We prove the existence of expansion. Fix $r\in \bbR_{>0}$ such that $r<R$. For each $z\in D_R(a)$, take $r_0\in \bbR_{>0}$ satisfying $|z-a|<r_0<R$, then we have by \cref{thm:Cauchyformu}, 
\begin{equation}\label{eq:fzintsum}
f(z)=\frac{1}{2\pi\sqrt{-1}}\int_{\pdd D_{r_0}(a)}\frac{f(\zeta)}{\zeta-a}d\zeta=
\frac{1}{2\pi\sqrt{-1}}\int_{\pdd D_{r_0}(a)}\sum_{n=0}^{\infty}\Bigl(\frac{z-a}{\zeta-a}\Bigr)^n\frac{f(\zeta)}{\zeta-a}d\zeta. 
\end{equation}
Note that for $p\in \Gamma$, where $\Gamma$ is a fundamental system of semi-norms for $E$, we have
\[
p\Bigl(\Bigl(\frac{z-a}{\zeta-a}\Bigr)^n\frac{f(\zeta)}{\zeta-a}\Bigr)\le
\frac{|z-a|^n}{r_0^{n+1}}\cdot\sup\{p(f(\zeta))\mid \zeta\in \pdd D_{r_0}(a)\}<\infty. 
\]
Thus, by \cref{prp:Mtest}, the series on the right hand side of \eqref{eq:fzintsum} converges uniformly on $\pdd D_{r_0}(a)$, and hence by \cref{prp:unifcompint}, 
\[
f(z)=
\sum_{n=0}^{\infty}(z-a)^n\frac{1}{2\pi\sqrt{-1}}\int_{\pdd D_{r_0}(a)}\frac{f(\zeta)}{(\zeta-a)^{n+1}}d\zeta=
\sum_{n=0}^{\infty}(z-a)^n\frac{1}{2\pi\sqrt{-1}}\int_{\pdd D_r(a)}\frac{f(\zeta)}{(\zeta-a)^{n+1}}d\zeta, 
\]
where we used \cref{thm:Cauchy} in the second equality. 
\end{proof}

\begin{lem}\label{lem:powserunif2}
Let $a\in \bbC$ and $R\in \bbR_{>0}$. For a sequence $\{a_n\}_{n\in \bbN}$, if the series $\sum_{n=0}^{\infty}(z-a)^{-n-1}a_n$ converges in $E$ for all $z\in D_R^{\times}(a)$, the series $\sum_{n\in \bbN}f_n$ of functions $f_n\colon D_R^{\times}(a)\to E$, $z\mapsto (z-a)^{-n-1}a_n$ converges uniformly on compact sets. 
\end{lem}

\begin{proof}
Let $K\subset D_R^{\times}(a)$ be a compact set, and take a fundamental system $\Gamma$ of semi-norms for $E$. Set $r\ceq \min\{|z-a|\mid z\in K\}$, then $0<r<R$. Fix $z_0\in D_R^{\times}(a)$ satisfying $0<|z_0-a|<r$. Since the series $\sum_{n=0}^{\infty}(z_0-a)^{-n-1}a_n$ converges, we have $\lim_{n\to \infty}(z_0-a)^{-n-1}a_n=0$. Hence, for $p\in \Gamma$, there exists $M_p\in \bbR_{>0}$ such that $p((z_0-a)^{-n-1}a_n)<M_p$ for all $n\in \bbN$. For $p\in \Gamma$, we have
\[
p((z-a)^{-n-1}a_n)=
\frac{|z_0-a|^{n+1}}{|z-a|^{n+1}}p((z_0-a)^{-n-1}a_n)\le 
\frac{|z_0-a|^{n+1}}{r^{n+1}}M_p \quad (z\in K).
\]
Thus, by \cref{prp:Mtest}, the series $\sum_{n\in \bbN}f_n$ converges uniformly on $K$. 
\end{proof}

\begin{thm}\label{thm:Laurent}
Let $a\in \bbC$ and $R\in \bbR_{>0}$. For a holomorphic function $f\colon D_R^{\times}(a)\to E$, there exists a unique expansion, called the Laurent series expansion
\[
f(z)=\sum_{n=0}^{\infty}(z-a)^{-n-1}a_{-n-1}+\sum_{n=0}^{\infty}(z-a)^na_n \quad (z\in D_R^{\times}(a)). 
\]
Moreover, we have
\[
a_n=\frac{1}{2\pi\sqrt{-1}}\int_{\pdd D_r(a)}\frac{f(\zeta)}{(\zeta-a)^{n+1}} \quad (r\in \bbR_{>0}, r<R). 
\]
for all $n\in \bbZ$. 
\end{thm}

\begin{proof}
As in the proof of \cref{thm:Taylor}, one can prove the uniqueness of expansion using \cref{lem:powserunif} and \cref{lem:powserunif2}.

We prove the existence of expansion. Fix $r\in \bbR_{>0}$ such that $r<R$. For $z\in D_R^{\times}(a)$, take $r_1, r_2\in \bbR_{>0}$ satisfying $r_1<|z-a|<r_2<R$, then by \cref{thm:Cauchy} and \cref{thm:Cauchyformu}, 
\begin{equation}\label{eq:Laurent}
\begin{aligned}
f(z)
&=
-\frac{1}{2\pi\sqrt{-1}}\int_{\pdd D_{r_1}(a)}\frac{f(\zeta)}{\zeta-z}d\zeta+
\frac{1}{2\pi\sqrt{-1}}\int_{\pdd D_{r_2}(a)}\frac{f(\zeta)}{\zeta-z}d\zeta\\
&=
\frac{1}{2\pi\sqrt{-1}}\int_{\pdd D_{r_1}(a)}\sum_{n=0}^{\infty}\Bigl(\frac{z-a}{\zeta-a}\Bigr)^n\frac{f(\zeta)}{\zeta-a}d\zeta+
\frac{1}{2\pi\sqrt{-1}}\int_{\pdd D_{r_2}(a)}\sum_{n=0}^{\infty}\Bigl(\frac{\zeta-a}{z-a}\Bigr)^n\frac{f(\zeta)}{z-a}d\zeta. 
\end{aligned}
\end{equation}
As in the proof of \cref{thm:Taylor}, one can show that the series on the right hand side of \eqref{eq:Laurent} converge uniformly on $\pdd D_{r_1}(a)$ and $\pdd D_{r_2}(a)$ respectively. Thus, by \cref{prp:unifcompint}, 
\begin{align*}
f(z)
&=
\sum_{n=0}^{\infty}(z-a)^{-n-1}\frac{1}{2\pi\sqrt{-1}}\int_{\pdd D_{r_1}(a)}(\zeta-a)^nf(\zeta)d\zeta+
\sum_{n=0}^{\infty}(z-a)^n\frac{1}{2\pi\sqrt{-1}}\int_{\pdd D_{r_2}(a)}\frac{f(\zeta)}{(\zeta-a)^{n+1}}d\zeta\\
&=
\sum_{n=0}^{\infty}(z-a)^{-n-1}\frac{1}{2\pi\sqrt{-1}}\int_{\pdd D_r(a)}(\zeta-a)^nf(\zeta)d\zeta+
\sum_{n=0}^{\infty}(z-a)^n\frac{1}{2\pi\sqrt{-1}}\int_{\pdd D_r(a)}\frac{f(\zeta)}{(\zeta-a)^{n+1}}d\zeta,
\end{align*}
where we used \cref{thm:Cauchy} in the second equality. 
\end{proof}

\begin{dfn}
Let $a\in \bbC$ and $R\in \bbR_{>0}$. For a holomorphic function $f\colon D_R^{\times}(a)\to E$, denote by
\[
f(z)=\sum_{n=0}^{\infty}(z-a)^{-n-1}a_{-n-1}+\sum_{n=0}^{\infty}(z-a)^na_n \quad (z\in D_R^{\times}(a)). 
\]
the Laurent series expansion. 
\begin{enumerate}
\item 
The point $a$ is called a removable singularity of $f$ if $a_n=0$ for all negative integers $n<0$. 

\item 
The point $a$ is called a pole of $f$ if there exists a negative integer $n<0$ such that $a_n\neq 0$ and $a_m=0$ for all $m<n$. In this case, the element $\Res_{z=a}f(z)\ceq a_{-1}\in E$ is called the residue of $f$ at $a$. 

\item 
The point $a$ is called a essential singularity of $f$ if it is neither removable or a pole. 
\end{enumerate}
\end{dfn}

\begin{thm}\label{thm:remov}
Let $a\in \bbC$ and $R\in \bbR_{>0}$. For a holomorphic function $f$, the followings are equivalent: 
\begin{clist}
\item 
The point $a$ is a removable singularity of $f$. 

\item 
The net $\{f(z)\}_{z\in D_R(a)\bs\{a\}}$ converges in $E$ (see \cref{ss:holofunc} for the directed order on $D_R(a)\bs\{a\}$). 

\item
For $p\in \Gamma$, where $\Gamma$ is a fundamental system of semi-norms for $E$, there exists a positive real number $0<r<R$ such that $\{p(f(z))\mid z\in D_r^{\times}(a)\}\subset \bbR$ is bounded. 
\end{clist}
\end{thm}

\begin{proof}
Denote by $f(z)=\sum_{n=0}^{\infty}(z-a)^{-n-1}a_{-n-1}+\sum_{n=0}^{\infty}(z-a)^na_n$ ($z\in D_R^{\times}(a))$ the Laurent series expansion. 

(i) $\Longrightarrow$ (ii): Since $a$ is removable, we have 
\[
f(z)=\sum_{n=0}^{\infty}(z-a)^na_n \quad (z\in D_R^{\times}(a)).
\]
By \cref{prp:unifconvcont}, the function of $z$ on the right hand side is continuous. Thus, 
\[
\lim_{z\in D_R(a)\bs\{a\}}f(z)=a_0\in E. 
\]

(ii) $\Longrightarrow$ (iii): Set $\xi\ceq \lim_{z\in D_R(a)\bs\{a\}}f(z)\in E$. For $p\in \Gamma$, there exists $r\in \bbR_{>0}$ such that 
\[
z\in D_R(a)\bs\{a\},\, |z-a|<r\, \Longrightarrow\, 
p(f(z)-\xi)<1. 
\]
Thus, we have $p(f(z))<1+p(f(\xi))$ for all $z\in D_r^{\times}(a)$. 

(iii) $\Longrightarrow$ (i): Let $n\in \bbN$. It is enough to show that $p(a_{-n-1})=0$ for all $p\in \Gamma$. By the condition (iii), there exists a positive real number $0<r_0<R$ and $M\in \bbR_{>0}$ such that $p(f(z))<M$ for $z\in D_{r_0}^{\times}(a)$. For $\ve\in \bbR_{>0}$, take $r\in \bbR_{>0}$ such that $r<r_0$ and $Mr^{n+1}<\ve$, then
\[
p(a_{-n-1})=
p\Bigl(\frac{1}{2\pi\sqrt{-1}}\int_{\pdd D_r(a)}(\zeta-a)^nf(\zeta)d\zeta\Bigr)\le 
\frac{Mr^n}{2\pi}l(\pdd D_r(a))=
Mr^{n+1}<\ve, 
\]
which shows $p(a_{-n-1})=0$. 
\end{proof}

%%%%%%%%%%%%%%%%%%%%%%%%%%%%%%%%%%%%%%%%%%%%%%%%%%%%%%%%%%%%%%%%%%%%%%%%%%%%%%%%%%%%%%%%%%
%%%%%%%%%%%%%%%%%%%%%%%%%%%%%%%%%%%%%%%%%%%%%%%%%%%%%%%%%%%%%%%%%%%%%%%%%%%%%%%%%%%%%%%%%%
%%%%%%%%%%%%%%%%%%%%%%%%%%%%%%%%%%%%%%%%%%%%%%%%%%%%%%%%%%%%%%%%%%%%%%%%%%%%%%%%%%%%%%%%%%

\end{document}